\documentclass[reqno,11pt]{amsart}

\usepackage[all,pdf]{xy}
\usepackage{epsfig}
\usepackage{amsmath}
\usepackage{amssymb}
\usepackage{amscd}
\usepackage{graphicx}

\usepackage[colorlinks=true]{hyperref}

\hypersetup{urlcolor=red, citecolor=blue}
\makeatletter
\@namedef{subjclassname@2020}{%
	\textup{2020} Mathematics Subject Classification}

\newtheorem{thm}{Theorem}[section]
\newtheorem{lem}[thm]{Lemma}

\newtheorem{prop}[thm]{Proposition}

\newtheorem{cor}[thm]{Corollary}

\theoremstyle{definition}
\newtheorem{defn}[thm]{Definition}
\newtheorem{notation}[thm]{Notation}
\newtheorem{exam}[thm]{Example}
\theoremstyle{remark}

\theoremstyle{remark}
\newtheorem{rem}[thm]{Remark}

\def \N {\mathbb N}

\def \R {\mathbb R}

\def\a{\alpha}

\def\m{\mathcal M}
\def\n{\mathcal N}

\def\o{\mathcal{O}}
\def\per{{\mathcal{P}er}}
\def\e{\mathcal{E}rg}
\def\lip{{\rm Lip}}
\newcommand{\supp}{\operatorname{supp}}
\def\w{\mathbb{W}}
\def\ZZ{\mathcal{Z}}
\def\X{\mathfrak{X}}
\def\D{\mathfrak{D}}
\def\s{\mathfrak{S}}
\def\i{\mathfrak{I}}
\def\v{\mathbb V}
\def\b{\kern 0.02em \flat \kern 0.02em}

\numberwithin{equation}{section}

\begin{document}
\title[Typical periodic optimization]{Typical periodic optimization for dynamical systems: Symbolic dynamics}

\author[W. Huang, O. Jenkinson, L. Xu and Y. Zhang]{Wen Huang, Oliver Jenkinson, Leiye Xu and Yiwei Zhang}

\address[W. Huang]{School of Mathematical Sciences, University of Science and
	Technology of China, Hefei, Anhui 230026, P.R. China}
\email{wenh@mail.ustc.edu.cn}

\address[O.~Jenkinson]{School of Mathematical Sciences, Queen Mary University of London, Mile End Road,
London, E1 4NS, UK}
\email{o.jenkinson@qmul.ac.uk}

\address[L. Xu]{School of Mathematical Sciences, University of Science and
	Technology of China, Hefei, Anhui 230026, P.R. China}
\email{leoasa@mail.ustc.edu.cn}

\address[Y.~Zhang]{School of Mathematical Sciences and Big Data, Anhui University of Science and Technology, Huainan, Anhui 232001, P.R. China}
\email{zhangyw@aust.edu.cn, yiweizhang831129@gmail.com}




\begin{abstract}
We develop a new theory of maximizing sets in dynamical systems, for the study of ergodic optimization in systems with weak
 hyperbolicity but where the Ma\~n\'e cohomology lemma does not hold.
This leads to new solutions of the Typical Periodic Optimization problem in the Lipschitz category:
existence of an open dense 
set of Lipschitz functions such that each member 
has a unique maximizing measure and this measure is periodic (an equi-distribution on a single periodic orbit).
The theory yields a structural theorem, 
that isolates the part of the system responsible for any robust non-periodic optimization.

The structural
theorem is developed further in the setting of symbolic dynamics: given any shift space, for typical Lipschitz functions the maximizing measure is shown to be either periodic or supported on the Markov boundary of the shift space.
It follows that Contreras' Typical Periodic Optimization theorem \cite{C02}
for shifts of finite type can be extended to a wide class of shift spaces, 
including every sofic shift.

The structural theorem is used to provide the first known example of a shift space 
where   Typical Periodic Optimization fails despite periodic measures being 
dense in the set of all invariant measures.
\end{abstract}

\maketitle

\setcounter{tocdepth}{1}
\tableofcontents

\section{ Introduction}\label{introsection}

Let $T:X\to X$ be a continuous map on a compact metric space, considered as a discrete time dynamical system,
with $\m(X,T)$ the space of $T$-invariant Borel probability measures.
Assume that $T$ has countably many periodic orbits
$\o_1,\o_2,\ldots$,
and let $\mu_n$ be the (unique) 
measure in $\m(X,T)$ supported by $\o_n$.
A continuous function $f:X\to\R$ has the \emph{periodic optimization} property if
there is a (necessarily unique) $n\in\N$ such that
\begin{equation}\label{periodicoptimization}
\int f\, d\mu_n > \int f\, d\mu\quad\text{for all }\mu\in\m(X,T)\setminus\{\mu_n\}\,.
\end{equation}

A remarkable phenomenon observed experimentally in the 1990s, notably by Hunt \& Ott \cite{HO96a, HO96b},
is that for chaotic maps $T$, if the function $f$ is sufficiently regular (e.g.~smooth, or Lipschitz) then it seems to 
\emph{typically} have the periodic optimization property, despite the fact that periodic measures constitute a topologically small subset of $\m(X,T)$.
Hunt \& Ott conjectured in \cite{HO96a, HO96b} that for suitable chaotic maps $T$, the periodic optimization property is typical in a probabilistic sense, and Yuan \& Hunt \cite{YH99} conjectured that the periodic optimization property is typical
in a topological sense when $T$ is either uniformly expanding, or an Axiom A diffeomorphism.
Progress towards the Hunt-Ott Conjecture has been made in the context of symbolic dynamics by
Bochi \& Zhang \cite{bochizhang}, Ding, Li \& Zhang \cite{DLZ24}, and very recently for circle expanding maps by
Gao, Shen \& Zhang \cite{GSZ25}. Work in the spirit of the Yuan-Hunt conjecture started earlier,
with advances made by Bousch
\cite{B01, B08}, Contreras, Lopes \& Thieullen \cite{CLT01}, Morris \cite{morrisnonlinearity},
and Quas \& Siefken \cite{quassiefken}, culminating in the breakthrough article by Contreras \cite{C02},
who showed that if $T$ is open and distance-expanding, and $\lip(X)$ is the space of Lipschitz functions on $X$, then
there is an \emph{open and dense} subset $P\subseteq \lip(X)$ such that each $f\in P$ has the periodic optimization property.
Subsequently, Huang, Lian, Ma, Xu \& Zhang \cite{HLMXZ19} (see also \cite{Boc19}) generalised Contreras' theorem to all uniformly hyperbolic systems (including Axiom A diffeomorphisms, and two-sided shifts of finite type),
and Li \& Zhang \cite{LZ24} proved the analogous result in the setting of expanding Thurston maps.
While the methods of \cite{HLMXZ19, LZ24} differ from those of \cite{C02}, a common
theme in all of these approaches is a combination of some form of closing lemma, shadowing, and perturbation based on a Ma\~n\'e
cohomology lemma.

While the Yuan-Hunt conjecture was effectively resolved by \cite{C02, HLMXZ19}, the spirit of the Hunt-Ott conjecture
is that the periodic optimization property may be typical 
for a considerably wider
class of chaotic dynamical systems;
this intuition was formulated as a \emph{meta-conjecture} in \cite{bochi_slides}, and articulated as the 
\emph{typically periodic optimization conjecture} in \cite{Jen19}.
Related phenomena in parallel areas, involving some form of typical periodic optimization, include Ma\~n\'e's 
conjecture\footnote{As noted by Contreras \cite[p.~4]{contreras24}, some of the original motivation for the general area of ergodic optimization, of which this article forms a part, was provided by Aubry-Mather theory,
and the prospect of using its techniques to answer questions in topological and symbolic dynamics; however, ergodic optimization has now developed to such an extent that the direction of influence can be reversed,
with its techniques (in particular \cite{bressaudquas, HLMXZ19})
being imported into Aubry-Mather theory to supply
the main ideas for the proof of an important conjecture in that field.}
in Lagrangian dynamics
(see \cite{contreras24,Ma92,Ma96}), and the finiteness conjecture of Lagarias \& Wang \cite{lagariaswang} 
concerning the joint spectral radius for sets of matrices
(although refuted by Bousch \& Mairesse \cite{bouschmairesse}, cf.~\cite{jenkinsonpollicott, morrissidorov},
a \emph{typical} version of it is still conjectured, see \cite{maesumi}).

In this article, we develop an approach to extending typical periodic optimization results beyond the uniformly hyperbolic setting of \cite{C02, HLMXZ19}. By contrast with the approaches of  \cite{C02, HLMXZ19,LZ24},
there is no reliance on a shadowing lemma or a Ma\~n\'e cohomology lemma; indeed typical periodic optimization is shown to hold for systems where these two results are false.
Rather, we develop a new theory of maximizing sets (to avoid confusion with existing notions we call these \emph{maximizable} sets), and \emph{countable maximizable families}.
To fix ideas we concentrate on the space $\lip(X)$ of Lipschitz real-valued functions on $X$,
and say that a dynamical system $T:X\to X$
has the \emph{typical periodic optimization (TPO)} property\footnote{Note that the TPO property is stronger than stating that periodic optimization is a generic property of $\lip(X)$, i.e.~that
functions with the periodic optimization property constitute a residual subset of $\lip(X)$.},
or simply that it \emph{has TPO},
if
there is an open and dense subset $P\subseteq \lip(X)$ such that each $f\in P$ has the periodic optimization property (\ref{periodicoptimization}).
For systems admitting a countable maximizable family, we establish a general \emph{structural theorem},
expressing a certain dense open subset of $\lip(X)$ as a union of two open sets,
one corresponding to periodic measures, the other to (potentially) non-periodic measures.
This effectively isolates the part of the system that may preclude typical periodic optimization,
and further analysis of this part (corresponding to the so-called \emph{boundary} of the system) is needed to reveal whether or not periodic optimization is typical.
Notably, the structural theorem may apply to the boundary dynamical system itself, and this procedure (of successively identifying a boundary, then applying the structural theorem to it) may be applied recursively, yielding in some cases a TPO theorem via a finite descent argument.  Conversely, the structural theorem also allows us to construct the first examples of chaotic systems which do \emph{not} have typical periodic optimization (despite periodic measures being dense in the space of all invariant measures), by suitable choice of a desired boundary system without the TPO property, and construction of a dynamical system whose boundary is 
precisely the desired one and moreover robustly supports non-periodic optimization.

This general approach has an inherent flexibility, and provides for TPO theorems to be established for a variety of
dynamical systems enjoying some hyperbolicity (including, for example, systems admitting a countable maximizable family
consisting of uniformly hyperbolic subsystems).
Here we pursue this programme in the setting of symbolic dynamics (i.e.~where $X$ is a set of symbol sequences over a finite alphabet, and $T$ is the shift map), a context in which typical periodic optimization has not previously been conjectured, beyond the case of shifts of finite type\footnote{Contreras' TPO theorem \cite{C02} is for distance-expanding open maps, which in the context of symbolic dynamics is precisely the class of shifts of finite type
(see 
\cite[Thm.~1]{Pa66}); the case of TPO for two-sided shifts of finite type was proved in \cite{HLMXZ19}.}.
It turns out that the TPO property is widespread in symbolic dynamical systems, and not restricted to shifts of finite type. 
For example, it can be shown to hold for all \emph{sofic} systems
 (i.e.~shift spaces that are the continuous image of a shift of finite type under a shift-commuting map, 
see e.g.~\cite{LM95, Wei73}):

\begin{thm}\label{sofictpointrotheorem}
Every sofic shift has the 
TPO
property.
\end{thm}

For arbitrary shift spaces $X$, the role of the boundary in the structural theorem described above is played by
 the \emph{Markov boundary} $\partial_M X$,  a canonical subshift of $X$
introduced by Thomsen \cite{Tho06}, 
with the property that $\partial_M X=\emptyset$ when $X$ is sofic.
The Markov boundary is a potential obstacle to TPO: in the non-sofic case it
supports invariant measures, 
and potentially supports non-periodic invariant measures $m$ that are \emph{robustly maximizing}, in the sense
that $\left\{f\in \lip(X): \int f\, dm > \int f\, d\mu\text{ for all }\mu\in\m(X,T)\setminus\{m\}\right\}$ has non-empty interior, a 
situation
that would
preclude TPO.
We show that indeed this can occur, even in cases where every invariant measure can be approximated by periodic ones
(previously, the possible existence of such dynamical systems was an open question):

\begin{thm}\label{non_tpo}
There exist 
shift spaces with periodic measures dense in the space of all invariant measures,
but which do not have TPO.
\end{thm}

However, the Markov boundary is the \emph{only} obstacle to the TPO property: 

\begin{thm}\label{partialX_tpo_implies_X_tpo}
A 
shift space whose Markov boundary has TPO will itself have TPO.
\end{thm}

Theorem \ref{partialX_tpo_implies_X_tpo}, a consequence of the general structural theorem, is 
a crucial tool for recognising the TPO property in a wider class of shift spaces. 
For example,
 $\partial_M X$ may be particularly simple, consisting of a single fixed point,
or a union of periodic orbits; such subshifts certainly have TPO, therefore so does $X$.
In particular, any \emph{$S$-gap shift} (see e.g.~\cite[\S 1.2]{LM95}) is either sofic, or has Markov boundary 
equal to a singleton fixed point, therefore:

\begin{thm}\label{sgaptpointrotheorem}
Every $S$-gap shift has the 
TPO
property.
\end{thm}

The analogue of Theorem \ref{sgaptpointrotheorem}
also holds for more general variants such as \emph{$S$-graph shifts}
(see \cite{Di22}).
More generally, Theorems \ref{sofictpointrotheorem} and \ref{partialX_tpo_implies_X_tpo}
mean that a shift space has TPO whenever its Markov boundary is sofic: for example
this is the case for the well known (non-sofic) \emph{context free shift} (see \cite[\S 1.2]{LM95}),
whose Markov boundary is a non-periodic subshift of finite type
(see Theorem \ref{context_free_theorem}).

Since the Markov boundary of a shift space is itself a shift space, it defines
an operator on the space of all shift spaces, an observation that,
together with Theorems \ref{sofictpointrotheorem} and \ref{partialX_tpo_implies_X_tpo},
motivates the definition
of the class of \emph{eventually sofic} shifts, consisting of those whose iterated Markov boundary is sofic.
We deduce:

\begin{thm}\label{eventuallysofictpointrotheorem}
Every eventually sofic shift has the 
TPO
property.
\end{thm}

We investigate the class of eventually sofic shifts, introducing a construction that
realises eventually sofic shifts of every \emph{level} $n$
(i.e.~such that $\partial^n_M X$ is sofic and non-empty), thereby generating a wide class of 
systems with typical periodic optimization.
As a specific example, if $X$ is the well known (and sofic) \emph{even shift} (see e.g.~\cite[\S 1.2]{LM95}),
consisting of 0-1 sequences
such that between any two 1's there is an even number of 0's,
then let $X_1$ be the set of sequences on $\{0,1,2\}$ such that between any two 2's there is an $X$-allowed
block whose length is some power of two; the shift space $X_1$ is non-sofic, but is eventually sofic of level 1,
and in particular has the TPO property by Theorem \ref{eventuallysofictpointrotheorem}.
If $X_2$ is defined as
the set of sequences on $\{0,1,2,3\}$ such that between any two 3's there is an $X_1$-allowed
block whose length is some power of two, then $X_2$ is eventually sofic of level 2,
and has TPO; continuing in this way yields a construction of eventually sofic shift spaces of any level $n$,
all of which have TPO.

There is, however, another class of shift space, disjoint from the eventually sofic class,
that has the typical periodic optimization property.
Shifts in this class arise from the observation,
another consequence of the structural theorem,
that the TPO property on $\partial_M X$ is not a \emph{necessary} condition for TPO
to hold  on $X$:

\begin{thm}\label{X_tpo_does_not_imply_partialX_tpo}
There exist 
shift spaces with TPO, but whose Markov boundary does not have TPO.
\end{thm}

To describe shift spaces corresponding to Theorem \ref{X_tpo_does_not_imply_partialX_tpo},
define a non-sofic shift space $X$ to be \emph{fragile} if its Markov boundary does not robustly support non-periodic optimization
(i.e.~if the set of functions with a maximizing measure supported by $\partial_M X$ has empty interior in $\lip(X)$), and \emph{eventually fragile} if some iterated Markov boundary is fragile.
We establish the following:

\begin{thm}\label{eventually_fragile_tpo_intro}
Every eventually fragile shift space has TPO.
\end{thm}

As a concrete class of examples of fragile shift spaces, let $Y$ be a shift on the alphabet $\{0,1\}$ that is minimal
(i.e.~all orbits are dense) but not uniquely ergodic (i.e.~supports more than one invariant measure);
for instance $Y$ might be chosen as an Oxtoby shift (see \cite{oxtoby}), or from the class of Toeplitz shifts
(see e.g.~\cite{downarowicz, jacobskeane, markleypaul, williams}). Then let $X$ be the set of sequences on $\{0,1,2\}$
such that between any two 2's there is a $Y$-allowed block.
The shift space $X$ is fragile, and by Theorem \ref{eventually_fragile_tpo_intro} it has the TPO property,
though its Markov boundary $\partial_M X=Y$ does not (cf.~Theorem \ref{X_tpo_does_not_imply_partialX_tpo}).

The organisation of this article is as follows.
The setting for Sections \ref{preliminariessection} to \ref{closingpropertysection} is that of
arbitrary topological dynamical systems,
while that of Sections \ref{symbolicdynamicssection} to \ref{magicmorsesection} is 
symbolic dynamical systems.
In Section \ref{preliminariessection} we introduce notation,
together with concepts from ergodic optimization,
in the context of topological dynamical systems $T:X\to X$ and continuous functions $f:X\to\R$.
Further background on ergodic optimization can be found in e.g.~\cite{Boc18, Jen19}.

In Section \ref{maximizableminimaxsection}, 
we 
develop a new theory of so-called \emph{maximizable sets},
those closed invariant sets supporting at least one maximizing measure for a given function.
For systems without a Ma\~n\'e lemma, the relation between maximizing measures and their support can be subtle, 
and this approach facilitates a more 
detailed analysis.
This leads to the key concept of a \emph{minimax set}, a maximizable set such that all maximizing measures with support within it necessarily have support equal to it. 
Such sets, which always exist, are shown to enjoy an important quantitative bound on Birkhoff sums.

In Section \ref{completelymaximizingsetssection}, the notion of a \emph{completely maximizing set} is introduced, a closed invariant set whose invariant measures are all maximizing\footnote{A completely maximizing set need not have the stronger property, exploited in various previous work 
in ergodic optimization,
 that a measure is maximizing \emph{if and only if} its support lies in the set; this stronger property is referred to here as \emph{subordination maximizing} (cf.~Definition \ref{subordinationmaximizingdefn}).}
(such sets need not exist, though there are important consequences if they do).
A key criterion is given, in terms of Birkhoff sums, for a minimax set to be completely maximizing.

Lipschitz functions are considered in Section \ref{maximizablefamiliessection}, where
closed invariant subsets of $X$ are used to define various subsets of $\lip(X)$ in terms of maximizing measures:
a \emph{maximizable family} is 
a collection of closed $T$-invariant sets 
where
every Lipschitz function has
at least one maximizable set in the collection.
Such families induce a 
cover of $\lip(X)$, and
for countable families, the interiors of the sets in the cover
together form an open dense subset of $\lip(X)$. This observation forms the basis of our approach,
when combined with 
\emph{subsystem reduction} (Theorem \ref{densedense}) which relates denseness in a space of restricted Lipschitz functions to denseness in an open set of Lipschitz functions defined on the whole space $X$.

In Section \ref{periodictposection} we consider
the typical periodic optimization (TPO) property, in the setting of general topological dynamical systems.
After clarifying some of its basic properties, including a necessary condition for TPO to hold,
the tools developed in Sections \ref{maximizableminimaxsection}, \ref{completelymaximizingsetssection} and \ref{maximizablefamiliessection}
are used to establish the important structural theorem (Theorem \ref{abstracttpovariant}) from which much of the subsequent development follows.
In particular, the importance of countable maximizable families, with either fragile boundary or TPO, is highlighted.

Section \ref{closingpropertysection}
marks a departure from the setting of completely arbitrary topological dynamical systems, focusing on 
a certain weak hyperbolicity condition, the \emph{subset closing property}, enjoyed by certain systems.
For Lipschitz functions, if a minimax set has the subset closing property,
we prove that a completely maximizing set is guaranteed to exist.

From Section \ref{symbolicdynamicssection} onwards, attention is focused on  
symbolic dynamical systems.
For any
shift space $X$, we show that a canonical 
subshift $X_w$ can be associated to each of its allowed words $w$,
allowing the subset closing property to be established.
This implies (Theorem \ref{l-6})
that if a Lipschitz function
has no maximizing measures in the Markov boundary $\partial_M X$, then there exists a 
completely maximizing
(and dynamically minimal) subshift contained in $X_w$.

This leads, in Section \ref{SFTmaximizablefamily}, to a certain family
of subshifts of finite type defined in terms of $X_w$,
and by adjoining the Markov boundary $\partial_M X$ this family is shown to be maximizable.
Since subshifts of finite type have TPO, 
this family represents a canonical choice of the
countable maximizable family described in the structural theorem (Theorem \ref{abstracttpovariant}).
From this we immediately deduce that 
the periodic optimization property is typical for all
sofic shifts
(Theorem \ref{sofictpointrotheorem} above),
and can formulate a symbolic dynamics version of the structural theorem,
the \emph{typical periodic or boundary optimization theorem} (Theorem \ref{general_tpbo_theorem}).
The class of eventually sofic shifts is introduced and investigated,
and members of this class, including $S$-gap shifts, $S$-graph shifts, and the context free shift, are also shown to have TPO.

In Section \ref{specimen_section}, a different route to constructing shift spaces with typical periodic optimization is pursued.
The class of (eventually) fragile shifts is introduced, and shown to have TPO. Particular examples of fragile shifts are the \emph{specimen} shifts, those which satisfy the variable specification condition and whose Markov boundary is minimal but not uniquely ergodic.

In Section \ref{TPO_preserving}, operations on shift spaces that preserve the TPO property are considered,
and a mechanism is given for constructing shift spaces $X$ with any given shift space $Y$ as their Markov boundary.
This leads to the construction of eventually sofic shifts of any level $n$ (i.e.~such that $\partial_M^n X$ is sofic and non-empty), and therefore to various examples of non-sofic shifts with TPO.
The method also yields specific examples of eventually specimen shifts,
including shift spaces which have TPO despite their Markov boundary not having this property.

Finally, in Section \ref{magicmorsesection}, we 
consider the problem of whether typical periodic optimization is a universal phenomenon
for shift spaces where periodic measures are (weak$^*$) dense in the space of all invariant measures.
We show that it is not, by explicitly constructing a counterexample.
The strategy consists of choosing a suitable (minimal, uniquely ergodic, non-periodic) shift $Z$
(for example the well known Morse shift),
then constructing a shift space $X$ by interspersing a new symbol between a certain
sparse set of $Z$-allowed blocks, in such a way that the Markov boundary of $X$ is equal to $Z$,
and the unique invariant (non-periodic) measure on $Z$ is robustly maximizing in $\lip(X)$.
Once again, the theoretical underpinning of this construction is the structural theorem (Theorem \ref{abstracttpovariant}).

\subsection*{Acknowledgments} The authors are grateful to Klaus Thomsen for inspirational discussion regarding the Markov boundary, and to the anonymous referees for their careful reading and many helpful suggestions which improved the manuscript.
W. Huang is partially supported by National Key R\&D Program of China (Nos.~2024YFA1013602, 2024YFA1013600) and National Natural Science Foundation of China grants (Nos.~12090012, 12031019, 12090010).
L. Xu is partially supported by National Key R\&D Program of China (Nos.~2024YFA1013602, 2024YFA1013600) and National Natural Science Foundation   (Nos.~12031019, 12371197, 12426201).
Y. Zhang is partially supported by National Natural Science Foundation (Nos.~12161141002, 12271432), and USTC-AUST Math Basic Discipline Research Center.
	This version of the article has been accepted for publication, after peer review (when applicable) but is 
not the Version of Record and does not reflect post-acceptance improvements, or any corrections. The Version of Record is 
available online at:
 \url{http://dx.doi.org/10.1007/s00222-026-01411-x}

\section{ Preliminaries}\label{preliminariessection}

Let $\N$ denote the set of strictly positive integers, and let $\N_0$ denote the set of non-negative integers.

Let $X$ be a compact metric space, with distance function $d$.
Let $C(X)$ denote the set of continuous real-valued functions on $X$, equipped with its Banach norm
$$\|f\|_{C(X)}
:=\max_{x\in X}|f(x)|\,.$$

Let $\lip(X)\subseteq C(X)$ denote the set of Lipschitz real-valued functions on $X$, equipped with its Banach norm
$$\|f\|_{\lip(X)}:= \|f\|_{C(X)} + |f|_{\lip(X)}\,,$$ where
$$|f|_{\lip(X)}:= \sup\left\{ \frac{|f(x)-f(y)|}{d(x,y)} : x,y\in X, x\neq y\right\}.$$

\begin{notation}
By a \emph{(topological) dynamical system} we mean a pair $(X,T)$, where $X$ is a compact metric space, and $T:X\to X$ is continuous.
The class of dynamical systems will be denoted by $\D$.
\end{notation}

For $X$ a metric space, let $\m(X)$ denote the set of Borel probability measures on $X$.
By the \emph{support} of a measure $\mu\in\m(X)$, denoted $\supp(\mu)$, we mean the smallest closed subset $Y\subseteq X$
such that $\mu(Y)=1$.
For a dynamical system $(X,T)$,
 the set of $T$-invariant Borel probability measures on $X$
will be denoted by $\m(X,T)$.
When equipped with the weak$^*$ topology, $\m(X)$ is compact,
as is its closed subset $\m(X,T)$.
Let $\e(X,T) \subseteq \m(X,T)$ denote the set of ergodic measures, in other words the set of extreme points of the simplex $\m(X,T)$.
A closed subset $Y\subseteq X$ will be called \emph{$T$-invariant}, or simply \emph{invariant}, if $T(Y)=Y$, and 
$\m(Y,T)$
will denote the set of those $T$-invariant Borel probability measures $\mu\in\m(X,T)$ with $\text{supp}(\mu)\subseteq Y$.
The set $\m(Y,T)$ is a (weak$^*$) closed subset of $\m(X,T)$, by the portmanteau theorem.
The support of any $\mu\in\m(X,T)$
is a non-empty closed $T$-invariant subset of $X$.
See e.g.~\cite[Ch.~3]{furstenberg}, \cite[Ch.~4]{glasner}, \cite[Ch.~6]{waltersbook}
for background on the ergodic theory of topological dynamical systems.

The key objects in ergodic optimization 
(see e.g.~ \cite{Boc18, Jen19} for surveys of this area)
are the following:

\begin{notation}
For a dynamical system $(X,T)\in\D$, and a function $f\in C(X)$, its 
\emph{maximum ergodic average} is defined
as
\begin{equation}\label{beta(f)defn}
\beta(X,T,f)=\beta(f):=\sup_{\mu\in\m(X,T)} \int f\, d\mu\,.
\end{equation}
Any $m\in \m(X,T)$ satisfying $\int f\, dm=\beta(f)$
is called a \emph{maximizing measure} for $f$, or an \emph{$f$-maximizing measure}.
Let 
\begin{equation}\label{mxtfdefn}
\m_{\max}(X,T,f)=\m_{\max}(f) := \left\{ \text{$f$-maximizing measures} \right\}.
\end{equation}
\end{notation}

Every $f\in C(X)$ has at least one maximizing measure:
the supremum in (\ref{beta(f)defn}) is attained, since $\m(X,T)$ is weak$^*$ compact.

Many of our results involve triples $(X,T,f)$, where $(X,T)\in\D$ is a dynamical system, and $f\in C(X)$ is a continuous real-valued function, so it is convenient to introduce the following notation.

\begin{notation}
Let us write
$$
\X:=\left\{(X,T,f): (X,T)\in\D,\ f\in C(X)\right\}.
$$
\end{notation}

Since $\beta(X,T,f+c)=\beta(X,T,f)+c$, and $\m_{\max}(X,T,f+c)=\m_{\max}(X,T,f)$, for all $(X,T,f)\in\X$, $c\in\R$,
it will often be convenient to assume that $\beta(X,T,f)=0$, thus motivating the following notation:

\begin{notation}
Let us write
$$
\X_0:=
\left\{ (X,T,f)\in\X:\ \beta(X,T,f)=0\right\} .
$$
\end{notation}

\section{ Maximizable sets and minimax sets}\label{maximizableminimaxsection}

A feature of uniformly hyperbolic dynamical systems $(X,T)$ is that for sufficiently regular (e.g.~Lipschitz) functions, 
maximizing measures can be characterised purely in terms of their support.
More precisely, for such a function $f$, it can be shown that there exists a closed invariant set $Y$ such that a measure is $f$-maximizing if and only if its support is contained in $Y$,
an insight  closely related to
the Ma\~n\'e Lemma 
(cf.~Section \ref{introsection}, and Remarks \ref{subordinationremark}(b) and \ref{sofic_non_mane}).
For general $(X,T,f)\in \X$, however,
the relation between maximizing measures and their support can be more subtle
(see e.g.~\cite{BJ,Mor10} for situations where knowledge of the support reveals little about a specific maximizing measure).
In order to progress beyond the uniformly hyperbolic setting, in this section 
we begin by developing a new theory of maximizing
sets\footnote{Here, the term `maximizing set' is used only in a loose sense, to refer to subsets of $X$ that can give information about maximizing measures. Although all the maximizing sets that we study are closed and invariant, and 
are at least \emph{maximizable} in the sense of Definition \ref{maximizabledefn}, in the wider literature various non-invariant maximizing sets have also proved useful, notably
 arising as 
the set of maxima for $f+\varphi-\varphi\circ T$ for certain $\varphi\in C(X)$
(e.g.~such sets were a key tool in \cite{B00}, and
have been termed \emph{action sets} in \cite{CLT01}, and \emph{contact loci} in \cite{GLT09}). Other useful 
maximizing sets for a given $(X,T,f)$ include its \emph{non-wandering (or Aubry) set} (see \cite{CLT01, garibaldi}), and its 
\emph{Mather set} (see \cite{morrisplms}).},
with slightly weaker properties than those demanded in the existing literature. The weakest such notion is as follows:

\begin{defn}\label{maximizabledefn}
For $(X,T,f)\in\X$,
a closed $T$-invariant set $Y$ is
\emph{maximizable} for $(X,T,f)$ (or simply for $f$), 
if 
\begin{equation}\label{maximizableeqn}
\m(Y,T)\cap\m_{\max}(f)\neq\emptyset\,,
\end{equation}
in other words if
some $f$-maximizing measure 
$\mu$ satisfies $\text{supp}(\mu)\subseteq Y$. 
For brevity we may say that $Y$ is \emph{$f$-maximizable}.
\end{defn}

Note that if $f$ has a unique maximizing measure $\mu$, its maximizable sets are precisely the closed invariant subsets containing $\text{supp}(\mu)$;
more generally, if $f$ has several maximizing measures, the maximizable sets are the closed invariant subsets containing the support of \emph{at least one} maximizing measure.
Amongst the $f$-maximizable sets, it will be of key importance to consider those whose support is in a sense minimal:

\begin{defn}\label{minimaxdefn}
If $(X,T,f)\in\X$,
a maximizable set $Y$ is said to be
\emph{minimax} for $(X,T,f)$ (or simply for $f$) if all measures $\mu\in\m(Y,T)\cap \m_{\max}(f)$ satisfy $\text{supp}(\mu)=Y$.
\end{defn}

\begin{rem}\label{notdynamicallyminimal}
The terminology \emph{minimax} in Definition \ref{minimaxdefn} is short for
\emph{minimal maximizable}, which in turn is motivated by Lemma \ref{minmaxexistence} below,
where such sets are seen to be minimal elements for a certain partial order on maximizable sets.
Note, however, that a minimax set need not be minimal\footnote{Henceforth, to avoid any possible confusion,
the term \emph{dynamically minimal} (rather than simply \emph{minimal}) will be used to refer to a 
dynamical system $(X,T)$ in which all orbits are dense.}
 in the dynamical sense (see e.g.~\cite[Defn.~1.3.2]{KH95}) of being equal to the closure of any orbit within it: for example every ergodic measure $\mu\in\m(X,T)$ is known (see \cite{Jen06b})
to be uniquely maximizing for some continuous $f$ (i.e.~$\m_{\max}(f)=\{\mu\}$), and in this case $\text{supp}(\mu)$ is a minimax set for $f$, irrespective of whether  $\text{supp}(\mu)$ is dynamically minimal. 
On the other hand, any maximizable set that is dynamically minimal is automatically a minimax set (since if an
invariant measure has support inside a dynamically
minimal set $Y$, then its support must equal $Y$).
\end{rem}

A key observation is that minimax sets always exist:

\begin{lem}\label{minmaxexistence}
Every $(X,T,f)\in\X$ has a minimax set.
\end{lem}
\begin{proof}
Let $\mathcal{Y}_f$ denote the collection of $f$-maximizable sets, equipped with the partial order $\subseteq$
of set-theoretic inclusion.
We claim that $(\mathcal{Y}_f,\subseteq)$ is such that every totally ordered subset $\{Y_\alpha\}_{\alpha \in A}$
of $\mathcal{Y}_f$
has its lower bound 
$Y_\infty := \bigcap_{\alpha \in A} Y_\alpha$ in $\mathcal{Y}_f$.

To see this, first note that since each $Y_\alpha$ is closed and $T$-invariant, $Y_\infty = \bigcap_{\alpha \in A} Y_\alpha$ 
is also closed and $T$-invariant.
For each $\alpha\in A$, the set
$
K_\alpha := \m_{\max}(f) \cap \m(Y_\alpha,T)
$ is closed, therefore compact, and non-empty since $Y_\alpha$ is maximizable.
Since $\{K_\alpha\}_{\alpha \in A}$ is totally ordered with respect to $\subseteq$,
and each $K_\alpha$ is compact and non-empty,
it follows that $\bigcap_{\alpha \in A} K_\alpha \neq \emptyset$.
Then any $\mu\in \bigcap_{\alpha \in A} K_\alpha$
satisfies $\mu\in\m_{\max}(f)$, and $\supp(\mu)\subseteq Y_\alpha$ for all $\alpha\in A$,
hence $\supp(\mu)\subseteq Y_\infty$, so $\mu\in \m_{\max}(f) \cap \m(Y_\infty,T)$,
thus $Y_\infty$ is maximizable, in other words $Y_\infty \in\mathcal{Y}_f$.

Therefore, by Zorn's lemma, $(\mathcal{Y}_f,\subseteq)$ has a minimal element $Y\in \mathcal{Y}_f$,
in the sense that if $Y'\in\mathcal{Y}_f$ satisfies $Y'\subseteq Y$, then $Y'=Y$.
We claim that any such $Y$ is a minimax set for $f$.
To see this, suppose on the contrary that the support $Y':=\text{supp}(\mu)$ of some $\mu\in\m(Y,T)\cap \m_{\max}(f)$ is a proper subset of $Y$. Since $Y'$ is closed and $T$-invariant, with $\mu\in\m(Y',T)\cap\m_{\max}(f)$, this means that $Y'\in\mathcal{Y}_f$,
with $Y'\subseteq Y$ but $Y'\neq Y$, thereby contradicting the fact that $Y$ is a minimal element of  
$(\mathcal{Y}_f,\subseteq)$.
\end{proof}

\begin{notation} For a map $T:X\to X$, a function $f:X\to\R$, and $n\in\N$, we shall write
$$
S_nf:=\sum_{i=0}^{n-1} f\circ T^i\,.
$$
\end{notation}

Minimax sets $Y$ turn out to have a very useful property: given any non-empty open subset $A\subseteq Y$,
all sufficiently long orbit segments in $Y$ that are disjoint from $A$
give
Birkhoff averages strictly smaller than the maximum ergodic average.

\begin{lem}\label{minimaxnegativeBirkhoff}
Suppose $(X,T,f)\in\X$, and that $Y\subseteq X$ is minimax for $f$.
For any non-empty open subset $A$ of $Y$, there exists $N_A\in\N$ such that
if $x\in Y$, and $n\ge N_A$ 
satisfies $T^i(x)\notin A$ for $0\le i\le n-1$, then 
\begin{equation}\label{strictlynegativeBirkhoffaveragefinite}
\frac{1}{n} S_nf(x) < \beta(f)\,.
\end{equation}
\end{lem}
\begin{proof}
It suffices to prove (\ref{strictlynegativeBirkhoffaveragefinite}) in the case that $\beta(f)=0$, in other words to show
that there exists $N_A\in\N$ such that
if $x\in Y$, $n\ge N_A$, and $T^i(x)\notin A$ for $0\le i \le n-1$, then 
\begin{equation}\label{strictlynegativeBirkhoffaveragefinitebetafequalszero}
S_nf(x) < 0\,.
\end{equation}
If (\ref{strictlynegativeBirkhoffaveragefinitebetafequalszero}) were false then there would exist a sequence
$(N_j)_{j=1}^\infty$ of natural numbers, with $\lim_{j\to\infty} N_j=\infty$, and a sequence $x_j\in Y$, such that
\begin{equation}\label{Titeratesnotin}
T^i(x_j)\notin A\quad\text{for all }0\le i\le N_j-1,\ j\in\N\,,
\end{equation}
 with 
\begin{equation}\label{nonnegativefaverage}
S_{N_j}f(x_j)\ge0\quad\text{for all }j\in\N\,.
\end{equation}
Defining the measure $\nu_j\in\m(Y)$ by
$$
\nu_j:= \frac{1}{N_j}\sum_{i=0}^{N_j-1} \delta_{T^i(x_j)},
$$
then (\ref{Titeratesnotin}) gives
\begin{equation}\label{Titeratesnotinnuj}
\nu_j(A)=0 \quad\text{for all }j\in\N\,,
\end{equation}
and (\ref{nonnegativefaverage}) gives
\begin{equation}\label{nonnegativefaveragenuj}
\int f\, d\nu_j \ge0\quad\text{for all }j\in\N\,.
\end{equation}
The weak$^*$ compactness of $\m(Y)$ means that the sequence $\nu_j$ has a subsequential limit $\nu\in\m(Y)$,
and in fact $\nu\in\m(Y,T)$.
Since $A$ is open then (\ref{Titeratesnotinnuj}) gives
\begin{equation}\label{Titeratesnotinnu}
\nu(A)=0 \,.
\end{equation}
So $\supp(\nu)\subseteq Y\setminus A$, and in particular the support of $\nu$ is a proper closed subset of $Y$.
But $\nu\in\m(Y,T)$ is also $f$-maximizing, because
(\ref{nonnegativefaveragenuj}) gives
\begin{equation}\label{nonnegativefaveragenu}
\int f\, d\nu \ge0\,.
\end{equation}
But $Y$ is a minimax set for $f$, so there are no $f$-maximizing measures whose support is
 a proper closed subset of $Y$.
This contradiction completes the proof.
\end{proof}

A consequence is the following bound on Birkhoff sums for points in a minimax set; it is conveniently stated in the case that the maximum ergodic average is zero.  

\begin{prop}\label{minimaxnegativeBirkhoffcor} (Minimax Birkhoff bound)

\noindent
Suppose $(X,T,f)\in\X_0$, and that $Y\subseteq X$ is a minimax set for $f$.
For any non-empty open subset $A$ of $Y$, there exists $N_A\in\N$ such that
if $x\in Y$ and $n\in\N$ satisfy $T^i(x)\notin A$ for $0\le i\le n-1$, then
\begin{equation}\label{coroSnfbound}
S_nf(x) \le N_A\, \|f\|_{C(X)}\,.
\end{equation}
\end{prop}
\begin{proof}
Let $N_A$ be as in Lemma \ref{minimaxnegativeBirkhoff}, and
suppose $x\in Y$ satisfies $T^i(x)\notin A$ for $0\le i\le n-1$.
If $n< N_A$ then there is the straightforward bound
$$
S_nf(x) \le n \|f\|_{C(X)} \le N_A\, \|f\|_{C(X)}\,,
$$
while
if $n\ge N_A$ then Lemma \ref{minimaxnegativeBirkhoff} gives
$$S_nf(x)<0 \le N_A\, \|f\|_{C(X)}\,,$$
so indeed  (\ref{coroSnfbound}) holds.
\end{proof}

In what follows, for a given triple  $(X,T,f)\in\X$, 
with $Z\subseteq X$ a non-empty closed $T$-invariant subset,
it will be important to consider the ergodic optimization problem for the
restricted triple
$(Z,T|_Z,f|_Z)\in\X$. 
For this it will be convenient to write
$(Z,T,f)$ instead of $(Z,T|_Z,f|_Z)$, leading to the following notation:

\begin{notation}
If $(X,T,f)\in\X$, 
and $Z\subseteq X$ is a non-empty closed $T$-invariant subset, we write
$$
\beta\left(Z,T,f\right)
= \max_{m\in\m(Z,T)} \int f|_Z\, dm
= \max_{m\in\m(Z,T)} \int f\, dm\,,
$$
and
$$
\m_{\max}(Z,T,f)
= \left\{ m\in\m(Z,T):  \int f\, dm = \beta\left(Z,T,f\right) \right\} \,.
$$
\end{notation}

The following is straightforward:

\begin{lem}\label{obvious}
If $(X,T,f)\in\X$,
and $Z\subseteq X$ is a closed non-empty $T$-invariant set, then
\begin{itemize}
\item[(a)] $\beta(Z,T,f) \le \beta(X,T,f)$.
\item[(b)] $Z$ is maximizable for $(X,T,f)$ if and only if $\beta(Z,T,f) = \beta(X,T,f)$.
\item[(c)] If  $\beta(Z,T,f) = \beta(X,T,f)$
then
$\m_{\max}(Z, T,f) \subseteq  \m_{\max}(X, T,f)$.
\end{itemize}
\end{lem}
\begin{proof}
(a) Immediate from the definitions, since $\m(Z,T)\subseteq\m(X,T)$.

(b) If $Z$ is maximizable for $(X,T,f)$ then there exists $\mu\in\m(Z,T)$ with $\int f\, d\mu=\beta(X,T,f)$, so that
$\beta(Z,T,f)=\max_{m\in\m(Z,T)}\int f\, dm\ge \int f\, d\mu=\beta(X,T,f)$. 
However $\beta(Z,T,f) \le \beta(X,T,f)$ by (a), so the equality $\beta(Z,T,f) = \beta(X,T,f)$ follows.

Conversely, if $\beta(Z,T,f) = \beta(X,T,f)$ then there exists $\mu \in \m_{\max}(Z,T,f)\subseteq \m(Z,T)$ 
with $\int f \, d\mu = \beta(Z,T,f)=\beta(X,T,f)$. 
So $\mu\in\m(Z,T)\cap \m_{\max}(X,T,f)$, therefore $Z$ is maximizable for $(X,T,f)$.

(c) If $\mu\in\m_{\max}(Z,T,f)$ then $\int f\, d\mu=\beta(Z,T,f) = \beta(X,T,f)$, so $\mu\in\m_{\max}(X,T,f)$,
 and the result follows.
\end{proof}

\section{ Completely maximizing sets}\label{completelymaximizingsetssection}

The following 
is a strengthening of the notion of a maximizable set:

\begin{defn}\label{completelymaximizingdefn}
For  $(X,T,f)\in\X$,
a closed non-empty $T$-invariant set $Y$ 
is called
\emph{completely maximizing} for  $(X,T,f)$ (or simply for $f$)
if 
\begin{equation}\label{mcmmax}
\m(Y,T)\subseteq\m_{\max}(f)\,,
\end{equation}
in other words if every invariant measure with support inside $Y$
is $f$-maximizing.
\end{defn}

We record the following straightforward consequence of the definition:

\begin{lem}\label{subsetcompletelymaximizing}
If $(X,T,f)\in\X$ then
\begin{itemize}
\item[(a)]
Any completely maximizing set for $f$ is maximizable for $f$.
\item[(b)]
Any  closed non-empty $T$-invariant subset
of a completely maximizing set for $f$ is itself
a completely maximizing set for $f$.
\end{itemize}
\end{lem}

\begin{rem}\label{completelymaximizingremark}
In general, a completely maximizing set need not exist: for example if $X$ is dynamically minimal but supports exactly two ergodic invariant measures $\mu, \mu'$, then  for any $f\in C(X)$ with $\int f\, d\mu \neq \int f\,d\mu'$,
the set $X$ is maximizable, and indeed minimax, though not completely maximizing.
But $X$
is the only closed non-empty $T$-invariant subset, so there are no completely maximizing sets for $f$.
\end{rem}

Any closed non-empty invariant set is completely maximizing for some continuous function
(e.g.~a constant), so 
a completely maximizing set
need not be dynamically minimal.
As noted in Remark \ref{notdynamicallyminimal},
a minimax set need not be dynamically minimal either.
However, a set which is both minimax and completely maximizing must be dynamically minimal:

\begin{lem}\label{completelymaximizingminimaximpliesdynamicallyminimal}
Suppose $(X,T,f)\in\X$.
If $Y$ is minimax for $f$, and also completely maximizing for $f$, then $Y$ is dynamically minimal.
\end{lem}
\begin{proof}
Let $Y'$ be any non-empty closed $T$-invariant subset of $Y$.
By Lemma \ref{subsetcompletelymaximizing}(b), $Y'$ is completely maximizing for $f$.
In particular, there exists an $f$-maximizing measure $\mu$ with $\supp(\mu)\subseteq Y' \subseteq Y$.
But $Y$ is minimax, so necessarily $\supp(\mu)=Y$, from which it follows that $Y'=Y$.
So  any non-empty closed $T$-invariant subset of $Y$ must equal $Y$; in other words, $Y$ is dynamically minimal.
\end{proof}

The following strong notion of maximizing set allows
the precise characterisation of maximizing measures
in terms of their support:

\begin{defn}\label{subordinationmaximizingdefn}
For $(X,T,f)\in\X$,
a closed $T$-invariant set $Y$ is called
\emph{subordination maximizing} for $(X,T,f)$ (or simply for $f$)
if 
\begin{equation}\label{mcequalsmmax}
\m(Y,T) = \m_{\max}(f)\,,
\end{equation}
in other words an invariant measure 
is $f$-maximizing if and only if its support is contained in $Y$.
\end{defn}

\begin{rem}\label{subordinationremark}
\begin{itemize}
\item[(a)]
Clearly every 
subordination maximizing set is
completely maximizing.
A subordination maximizing set need not be unique, since there exist closed invariant sets
$Y\neq Y'$ with $\m(Y,T)=\m(Y',T)$: for example if $(X,T)\in\D$ is surjective, uniquely ergodic, but not dynamically minimal,
then this holds for $Y=X$, and $Y'$ the support of the unique invariant probability measure.
\item[(b)]
The existence of a subordination maximizing set is equivalent 
(by work of Morris \cite[Thm.~1, Prop.~1]{M07})
to the so-called\footnote{In fact Bousch \cite[p.~290]{B01} articulated this
as a subordination \emph{principle}, having proved that it holds for all Walters functions (cf.~\cite{walters});
for more general continuous functions it is a property that may or may not hold,
hence our usage of the term subordination \emph{property}. In this article the term subordination \emph{principle} is reserved to mean that the subordination property holds for all Lipschitz functions (cf.~Definition \ref{lipschitz_subordination_principle}).}
 \emph{subordination property} 
 introduced by Bousch \cite[p.~290]{B01}:
 $f\in C(X)$ has the subordination property
if, for any $\mu,\nu\in\m(X,T)$, with $\nu\in\m_{\max}(f)$ and $\supp(\mu)\subseteq\supp(\nu)$, then 
$\mu\in\m_{\max}(f)$.
\end{itemize}
\end{rem}

Morris \cite[Thm.~1]{M07} established a characterisation of the subordination property, in other words the existence of subordination maximizing sets, in terms of uniform boundedness of Birkhoff sums.
It is convenient to state the result in the case that $\beta(f)=0$:

\begin{lem}\label{morrislemma}
Suppose $(X,T,f)\in\X_0$.
If
\begin{equation}\label{morrisbound}
\sup_{n\ge1} \sup_{x\in X} S_n f(x) 
< \infty
\end{equation}
then there exists a subordination maximizing set for $f$.
\end{lem}
\begin{proof}
By \cite[Thm.~1]{M07}, 
condition (\ref{morrisbound}) implies that $f$ satisfies the subordination property
(cf.~Remark \ref{subordinationremark}(b)),
and by \cite[Prop.~1]{M07} this is equivalent to the existence of a closed subset $K\subseteq X$
such that $\mu\in\m(X,T)$ is $f$-maximizing if and only if $\supp(\mu)\subseteq K$.
The closed $T$-invariant set $Y:= \cap_{n=0}^\infty T^n \left( \cap_{i=0}^\infty T^{-i}(K) \right)$ then has the analogous property:
$\mu\in\m(X,T)$ is $f$-maximizing if and only if $\supp(\mu)\subseteq Y$.
In other words, $Y$ is a
subordination maximizing set
for $f$.
\end{proof}

An important consequence is the following criterion
for a minimax set to be completely maximizing:

\begin{prop}\label{morrisminimaxcompletely} (Completely Maximizing Criterion)
 
\noindent Suppose $(X,T,f)\in\X_0$.
If $Y\subseteq X$ is a minimax set for $f$, and
\begin{equation}\label{morristypebound}
\sup_{n\ge 1} \sup_{x\in Y} S_nf(x) < \infty\,,
\end{equation}
then $Y$ is completely maximizing for $f$,
and is also dynamically minimal.
\end{prop}
\begin{proof}
Condition (\ref{morristypebound}) implies, by Lemma \ref{morrislemma}, that 
there is a closed $T$-invariant subset $Z\subseteq Y$ that is a subordination maximizing set 
for the restricted system $(Y,T,f)$, in other words
\begin{equation}\label{mymmax}
\m(Z,T)= \m_{\max}(Y, T,f)\,.
\end{equation}

Now $Y$ is minimax for $(X,T,f)$, and hence maximizable for $(X,T,f)$, so
parts  (b) and (c) of
Lemma \ref{obvious} mean that
$\beta(Y,T,f) = \beta(X,T,f)$ and
\begin{equation}\label{mmaxinclusion}
\m_{\max}(Y, T,f) \subseteq  \m_{\max}(X, T,f)\,.
\end{equation}

Combining (\ref{mymmax}) and (\ref{mmaxinclusion}) gives
\begin{equation}\label{mzmmaxinclusion}
\m(Z,T) \subseteq  \m_{\max}(X, T,f)\,.
\end{equation}

But $Y$ is minimax for $f$, so all measures in $\m(Y,T)\cap\m_{\max}(X,T,f)$ must have 
their support equal to all of $Y$.
This means that all measures in $\m(Z,T)$ must have their support equal to $Y$.
But this means that in fact
$Z=Y$.

This implies that (\ref{mzmmaxinclusion}) becomes
\begin{equation*}
\m(Y,T) \subseteq  \m_{\max}(X, T,f)\,.
\end{equation*}
In other words, $Y$ is a completely maximizing set for $(X,T,f)$.

Since $Y$ is both minimax and completely maximizing, it is necessarily dynamically minimal, by
Lemma \ref{completelymaximizingminimaximpliesdynamicallyminimal}.
\end{proof}

\section{ Maximizable families}\label{maximizablefamiliessection}

In this section, we consider the space $\lip(X)$ of Lipschitz real-valued functions on $X$.
Collections of invariant measures will define subsets of $\lip(X)$, via ergodic optimization, in the following way:

\begin{notation}\label{maximizabledomaindefn}
Let $(X,T)\in\D$.
For a subset $\n\subseteq \m(X,T)$, we write
$$
\lip^\n(X,T)  := \left\{f\in \lip(X): \m_{\max}(f) \cap \n \neq\emptyset\right\}\,,
$$
$$
\lip_{\subseteq}^\n(X,T) := \left\{f\in \lip(X): \m_{\max}(f) \subseteq \n\right\}\,.
$$
When $Z$ is a closed $T$-invariant subset of $X$, 
we write
\begin{equation*}\label{LZdef}
\lip_Z(X,T):= \lip^{\m(Z,T)}(X,T)
= \left\{f\in \lip(X): Z\text{ is $f$-maximizable}\right\}\,,
\end{equation*}
\begin{equation*}
\lip_Z^\subseteq(X,T):= \lip^{\m(Z,T)}_\subseteq(X,T)
=  \left\{f\in \lip(X): \m_{\max}(f) \subseteq \m(Z,T)\right\}
 \,.
\end{equation*}
\end{notation}

\begin{rem}
In the sequel, the subset $\n\subseteq \m(X,T)$ will be chosen as either
$\n=\m(Z,T)$ or  $\n=\per(Z,T)$ (the set of periodic measures on $Z$, cf.~Definition \ref{periodic_measures_defn}), where $Z\subseteq X$ is a closed invariant subset.
\end{rem}

The following lemmas are straightforward:

\begin{lem}\label{LipZnonemptyiff}
Let $(X,T)\in\D$.
A closed $T$-invariant subset  $Z\subseteq X$ is non-empty
if and only if $\lip_Z(X,T)$ is non-empty.
\end{lem}
\begin{proof}
If $Z=\emptyset$ then $\m(Z,T)=\emptyset$, so  
$Z$ is not $f$-maximizable for any $f\in\lip(X)$, i.e.~$\lip_Z(X,T)=\emptyset$.
If $Z\neq\emptyset$ then $Z$ is $f$-maximizable for constant functions, so $\lip_Z(X,T)\neq\emptyset$.
\end{proof}

\begin{lem}\label{LZclosed}
Let $(X,T)\in\D$.
If $\n$ is a closed subset of $\m(X,T)$,
then $\lip^\n(X,T)$ is a closed subset of $\lip(X)$.
In particular,  if $Z\subseteq X$ is a closed $T$-invariant subset, then
$\lip_Z(X,T)$ is a closed subset of $\lip(X)$.
\end{lem}
\begin{proof}
Suppose $f_i\in \lip^\n(X,T)$, with $f_i\to f$ in $\lip(X)$, so in particular $f_i\to f$ in the topology of $C(X)$.
There exist $f_i$-maximizing measures $\mu_i\in\n$, and by passing to a subsequence we may assume that
$\mu_i\to\mu\in\n$.
But 
\begin{equation}\label{intfexpression}
\int f\,d\mu = \left(\int f\, d\mu -\int f\, d\mu_i\right) +\left(\int f\, d\mu_i - \int f_i\, d\mu_i\right) +  \int f_i\, d\mu_i\,,
\end{equation}
and noting that $\int f_i\, d\mu_i \ge \int f_i\, dm$ for all $m\in\m(X,T)$,
then letting $i\to\infty$  in (\ref{intfexpression}), we see that
$\int f\, d\mu\ge \int f\, dm$ for all $m\in\m(X,T)$, 
so $\mu$ is $f$-maximizing.
But $\mu\in\n$, therefore $\mu\in\m_{\max}(f)\cap\n$, in other words $f\in \lip^\n(X,T)$, so $\lip^\n(X,T)$ is closed.
\end{proof}

The notion of maximizable set 
has the following important generalisation:

\begin{defn}\label{maximizablefamily}
Let $(X,T)\in\D$.
A
\emph{maximizable family}
is a collection $\ZZ$ of closed $T$-invariant subsets 
such that
for all $f\in \lip(X)$ there exists $Z\in\ZZ$ that
is $f$-maximizable.
\end{defn}

Maximizable families $\ZZ$ for dynamical systems $(X,T)$ will be a key tool in the development
that follows: if $\ZZ$ is \emph{non-trivial} (i.e.~does not contain $X$ itself)
then in some cases it is possible to deduce properties of $(X,T)$ from properties of the members of $\ZZ$.
A first step is the straightforward observation that every maximizable family induces a covering of $\lip(X)$:

\begin{lem}\label{LZcoversL}
If $(X,T)\in\D$, with maximizable family $\ZZ$, then
\begin{equation}\label{coverL}
\lip(X) = \bigcup_{Z\in\ZZ} \lip_Z(X,T)\,.
\end{equation} 
\end{lem}

When the maximizable family is countable, a key consequence is the following:

\begin{lem}\label{baireconsequence}
If $(X,T)\in\D$,
and $\ZZ=\{Z_i\}_{i=1}^\infty$ is a countable maximizable family, then
$$\bigcup_{i=1}^\infty int\left(\lip_{Z_i}(X,T)\right)$$ is dense in $\lip(X)$.
\end{lem}
\begin{proof}
Since $\ZZ=\{Z_i\}_{i=1}^\infty$ is maximizable,
 Lemma \ref{LZcoversL} gives
$\lip(X) = \bigcup_{i=1}^\infty \lip_{Z_i}(X,T)$.
Since $\lip(X)$ is a Baire space, and each $\lip_{Z_i}(X,T)$ is closed,
 the Baire Category Theorem means that the union of their interiors is dense in $\lip(X)$
(see e.g.~\cite[Thm.~3.34]{AB99}).
\end{proof}

For certain $\n$, in particular
if $\n=\m(Z,T)$, the sets
 $\lip^\n(X)$ and $\lip_\subseteq^\n(X)$ have a common interior, a result that will be useful later.

\begin{lem}\label{Nclosedconvexsameinteriors}
For $(X,T)\in\D$,
if $\n$ is a closed convex subset of $\m(X,T)$
then
$$
int\left(\lip^\n(X,T) \right) = int\left( \lip_\subseteq^\n(X,T) \right) .
$$
\end{lem}
\begin{proof}
If $\n=\emptyset$ then $\lip^\n(X,T)=\emptyset=\lip^\n_\subseteq(X,T)$, so the required equality holds.

Suppose that $\n\neq\emptyset$.
Now $ \lip_\subseteq^\n(X,T) \subseteq \lip^\n(X,T)$,
so
 $ int\left(\lip_\subseteq^\n(X,T)\right)$ is contained in $int\left(\lip^\n(X,T)\right)$.
To prove the reverse inclusion 
$
int\left(\lip^\n(X,T) \right) \subseteq int\left( \lip_\subseteq^\n(X,T) \right)
$, 
note that it is equivalent to showing that
$
int\left(\lip^\n(X,T) \right) \subseteq  \lip_\subseteq^\n(X,T)
$. 
To prove this inclusion,
suppose
for a contradiction that there exists  
$f\in  int\left(\lip^\n(X,T)\right) \setminus  \lip_\subseteq^\n(X,T)$.

So there exists $\mu\in\m_{\max}(f)\setminus \n$.
Since $\n$ is a non-empty closed convex subset of the locally convex space $\m(X,T)$, and $\mu\notin\n$,
there is a continuous linear functional $l$ 
(defined on the space of signed measures on $X$ equipped with the weak$^*$ topology)
that strongly separates $\n$ and $\mu$ (see \cite[Cor.~5.59]{AB99}).
Identifying $l$ with some $\phi\in C(X)$,
this means
there exists $\delta>0$ such that $\int \phi\, d\mu =3\delta$ and $\int \phi\, d\nu\le 0$ for all $\nu\in\n$.
Since $\lip(X)$ is uniformly dense in $C(X)$, 
by the Stone-Weierstrass theorem,
there exists $\psi\in \lip(X)$ with $\|\phi-\psi\|_{C(X)}\le\delta$,
so $\int \psi\, d\mu \ge 2\delta\ge \delta +\int\psi\, d\nu$ for all $\nu\in\n$.
It follows that
$\int (f+\epsilon\psi)\, d\mu \ge \delta\epsilon + \int (f+\epsilon\psi)\, d\nu$  for all $\nu\in\n$, $\epsilon>0$.
So if $\epsilon>0$ then $\n$ is disjoint from $\m_{\max}(f+\epsilon\psi)$,
and therefore $f+\epsilon\psi \notin \lip^\n(X,T)$,
thus contradicting the fact that
$f$ is in the interior of $\lip^\n(X,T)$.
\end{proof}

The following important result relates denseness in a space of restricted Lipschitz functions to denseness in an open set of Lipschitz functions defined on the whole space $X$.

\begin{thm}\label{densedense} (Subsystem Reduction)

\noindent
Suppose $(X,T)\in\D$, with $Z$ a closed $T$-invariant subset of $X$.
Let $U$ be a non-empty open subset of $\lip(X)$ such that $Z$ is maximizable for all $f\in U$.
If $\n\subseteq \m(X,T)$ is such that $\lip^{\n\cap\m(Z,T)}(Z,T)$ is dense in $\lip(Z)$,
then $\lip^\n(X,T) \cap U$ is dense in $U$.
\end{thm}
\begin{proof}
Let $f\in U$, and choose $\delta>0$ small enough that 
\begin{equation}\label{ballinU}
\{g\in\lip(X):\|f-g\|_{\lip(X)}< \delta\} \subseteq U \,.
\end{equation}
Since $\lip^{\n\cap\m(Z,T)}(Z,T)$ is dense in $\lip(Z)$, there exists $h_\delta\in\lip(Z)$ with $\|h_\delta\|_{\lip(Z)}<\delta$
such that $f|_Z + h_\delta \in \lip^{\n\cap\m(Z,T)}(Z,T)$,
and so
there exists $\mu\in\n\cap\m(Z,T)$ such that
\begin{equation}\label{singletonmu}
\mu\in\m_{\max}(Z,T,f|_Z + h_\delta)\,.
\end{equation}

By McShane's extension theorem \cite{Mc34}, the Lipschitz function $h_\delta$ can be extended to a 
function
 $\widehat{h}_\delta\in \lip(X)$ with the same Lipschitz norm (see \cite[Thm.~1.33]{Wea18}),
so $\|\widehat{h}_\delta\|_{\lip(X)} = \|h_\delta\|_{\lip(Z)}<\delta$,
and therefore  (\ref{ballinU}) implies that
\begin{equation}\label{perturbationinU}
f+\widehat{h}_\delta \in U\,.
\end{equation}

The fact that $f+\widehat{h}_\delta \in U$ means that $Z$ is maximizable for  $f+\widehat{h}_\delta$.
Consequently, Lemma \ref{obvious}(b) and (c) give that $\beta(Z,T,f|_Z+h_\delta)=\beta(X,T,f+\widehat{h}_\delta)$,
and 
\begin{equation}\label{maxmeasuresetscontained}
\m_{\max}(Z,T,f|_Z+h_\delta)\subseteq \m_{\max}(X,T,f+\widehat{h}_\delta)\,.
\end{equation}
Combining (\ref{singletonmu}) and (\ref{maxmeasuresetscontained}) gives that 
$\mu\in  \m_{\max}(X,T,f+\widehat{h}_\delta)$.
Since we also have that
 $\mu\in\n\cap\m(Z,T)\subseteq \n$,
 it follows that
\begin{equation}\label{fperturbationinLipN}
f+\widehat{h}_\delta \in \lip^\n(X,T)\,,
\end{equation}
so
(\ref{perturbationinU}) and (\ref{fperturbationinLipN}) give
\begin{equation}\label{fperturbationinLipNandU}
f+\widehat{h}_\delta \in \lip^\n(X,T) \cap U \,.
\end{equation}
Since $f$ was an arbitrary member of $U$, and $\|\widehat{h}_\delta\|_{\lip(X)}<\delta$ can be chosen arbitrarily small, it follows from (\ref{fperturbationinLipNandU}) that
$\lip^\n(X,T) \cap U$ is dense in $U$.
\end{proof}

\section{ Typical periodic optimization and the structural theorem}\label{periodictposection}

In this section, we begin the
study of the main theme of the article, typical periodic optimization.
In particular, the tools developed in
Sections \ref{maximizableminimaxsection}, \ref{completelymaximizingsetssection} and \ref{maximizablefamiliessection}
will be applied so as to establish a key structural theorem (Theorem \ref{abstracttpovariant}).

\begin{defn}\label{periodic_measures_defn}
For $(X,T)\in\D$,
if $x\in X$ is periodic with least period $n$, then its orbit $\{x,\ldots,T^{n-1}(x)\}$
supports a unique $T$-invariant probability measure, namely
$$
\mu=\frac{1}{n} \sum_{j=0}^{n-1}\delta_{T^j(x)}\,,
$$
and any such $\mu$ will be called a \emph{periodic measure}. We define
$$
\per(X,T):=\left\{\mu\in\m(X,T): \mu \text{ is periodic}\right\},
$$
the set of periodic measures for the dynamical system $(X,T)$.
\end{defn}

The following definition was discussed in Section \ref{introsection}:

\begin{defn}
Given $(X,T)\in\D$,
 a continuous function $f:X\to\R$ is said to have the \emph{periodic optimization} property
 if there exists $m\in\per(X,T)$ such that 
$$\int f\, dm > \int f\, d\mu\text{ for all }\mu\in\m(X,T)\setminus\{m\}.
$$
\end{defn}

The periodic optimization property leads to the following definition:

\begin{defn}\label{tpodefn}
A dynamical system $(X,T)\in\D$ has the \emph{typical periodic optimization (TPO)} 
property
if there is an open dense subset $P\subseteq \lip(X)$ such that every $f\in P$ has the periodic
optimization property.
For brevity, we may simply say that $(X,T)$ \emph{has TPO}.
\end{defn}

\begin{rem}
Note that TPO is stronger than saying that Lipschitz functions with the periodic optimization property are generic,
i.e.~constitute a residual subset of $\lip(X)$. 
\end{rem}

If $\lip^\per_{\, !}(X,T)$ is defined to be the set of \emph{all} Lipschitz functions with 
the periodic optimization property, then we can express
\begin{equation}\label{lip!per}
\lip^\per_{\, !}(X,T) := \bigcup_{\mu\in\per(X,T)} \lip^{\{\mu\}}_\subseteq(X,T).
\end{equation}

In general $\lip^\per_{\, !}(X,T)$ need not be an open subset of $\lip(X)$, so its interior
$$
P(X,T):= int\left(  \lip^\per_{\, !}(X,T) \right)
$$
will play the role of the set $P$ from Definition \ref{tpodefn}, i.e.~if $(X,T)$ has TPO
then $P(X,T)$ will be dense in $\lip(X)$.

There are two other open subsets of $\lip(X)$ that, for a given $(X,T)\in\D$, are naturally connected to periodic optimization.
The first of these, denoted by $P_{-}(X,T)$, is given by considering the interiors
of the sets $\lip^{\{\mu\}}_\subseteq(X,T)$ that appear
in (\ref{lip!per}), and then taking their union
\begin{equation}\label{pminusxt}
P_{-}(X,T):= \bigcup_{\mu\in\per(X,T)}  int\left( \lip^{\{\mu\}}_\subseteq(X,T)\right)
= \bigcup_{\mu\in\per(X,T)}  int\left( \lip^{\{\mu\}}(X,T)\right),
\end{equation}
where we note that the second equality in (\ref{pminusxt})
is by Lemma \ref{Nclosedconvexsameinteriors}
with $\n=\{\mu\}$.
The set $P_-(X,T)$ consists of those Lipschitz functions with the so-called \emph{locking property}
(as defined in \cite{Boc19, BZ15}, cf.~\cite{B00, Jen00} for related usage of the term \emph{locking}).
By comparison of (\ref{lip!per}) and (\ref{pminusxt}), the open set $P_{-}(X,T)$ is a subset of  $\lip^\per_{\, !}(X,T)$,
and hence also a subset of $P(X,T)=  int\left(  \lip^\per_{\, !}(X,T) \right)$.

By contrast there is also a \emph{superset} $P_+(X,T)$ of $P(X,T)$ that may be regarded as a natural open subset of $\lip(X)$ in the context of TPO. To describe $P_+(X,T)$ we first say that $f$ has the \emph{weak periodic optimization} property if at least one of its maximizing measures is periodic, i.e.~ if there exists $m\in\per(X,T)$ such that 
$\int f\, dm \ge \int f\, d\mu$ for all $\mu\in\m(X,T)$, and then define
$\lip^\per(X,T)$ as the set of all Lipschitz functions with the weak periodic optimization property.
The set $\lip^\per(X,T)$ can be expressed as
\begin{equation}\label{lipperexpression}
\lip^\per(X,T):=\lip^{\per(X,T)}(X,T) =  \bigcup_{\mu\in\per(X,T)} \lip^{\{\mu\}}(X,T),
\end{equation}
and we define
$$
P_+(X,T):= int\left( \lip^\per(X,T) \right).
$$
Clearly $\lip^\per_{\, !}(X,T)$ is a subset of $\lip^\per(X,T)$,
therefore
$P(X,T)= int\left(  \lip^\per_{\, !}(X,T) \right)$
is a subset of
$P_+(X,T)= int\left( \lip^\per(X,T) \right)$. Summarising, we have:
\begin{equation}\label{containmentp}
P_-(X,T) \subseteq P(X,T) \subseteq P_+(X,T) \subseteq \lip^\per(X,T).
\end{equation}

To investigate the extent to which periodic optimization is typical, it would be most satisfactory if the natural 
open sets
$P_-(X,T)$, $P(X,T)$, $P_+(X,T)$ all have the \emph{same closure} in $\lip(X)$,
and indeed this is the case.
Although 
$
\overline{P_-(X,T)} = \overline{ P(X,T)} =  \overline{ P_+(X,T)}
$
does hold for arbitrary dynamical systems $(X,T)\in\D$,
a general proof 
is a little protracted for our purposes,
so in the following we give a short proof valid in the case
that $T$ is Lipschitz
(which is sufficient for the 
specific classes of systems studied in Section \ref{symbolicdynamicssection} onwards), 
and briefly indicate a route to the more general result.

\begin{lem}\label{four_closures}
If $(X,T)\in\D$ 
then
\begin{equation}\label{4closuresequal}
\overline{P_-(X,T)} = \overline{ P(X,T)} =  \overline{ P_+(X,T)} = \overline{\lip^\per(X,T)}.
\end{equation}
\end{lem}
\begin{proof}
Let us suppose that $T$ is Lipschitz, so each iterate $T^i$ is Lipschitz as well, in other words there exist $K_i>0$ such that
$d(T^i(x),T^i(x')) \le K_i d(x,x')$ for all $x,x'\in X$, $i\in\N$.
In view of (\ref{containmentp}), to prove (\ref{4closuresequal}) it suffices to show that 
\begin{equation}\label{suffices_lipper_p_minus}
\lip^\per(X,T)\subseteq \overline{P_-(X,T)}.
\end{equation}
Suppose that $f\in \lip^\per(X,T)$. So there exists a periodic $f$-maximizing measure $\mu$, with corresponding orbit $Q:=\supp(\mu)$ of cardinality $q\in\N$, say.
Note that if $g\in\lip(X)$, $x\in X$, then
\begin{equation}\label{g_x_bound}
S_qg(x)- q \int g\, d\mu \le |g|_{\lip(X)} d(x,Q) \sum_{i=0}^{q-1} K_i,
\end{equation}
since if $y\in Q$ is such that $d(x,y)=d(x,Q):=\min_{z\in Q} d(x,z)$ then the lefthand side of
(\ref{g_x_bound}) equals $\sum_{i=0}^{q-1} (g(T^ix) - g(T^iy))$, which is readily bounded using the 
Lipschitz constants for $g$ and the $T^i$.

For any $\epsilon>0$, define $f_\epsilon:= f -\epsilon d(\cdot, Q)$, and
\begin{equation}\label{U_epsilon_defn}
U_\epsilon:= \left\{ g\in \lip(X):  |g|_{\lip(X)} \sum_{i=0}^{q-1} K_i < \epsilon/2 \right\}.
\end{equation}
We claim that $\mu$ is the unique $(f_\epsilon+g)$-maximizing measure for all $g\in U_\epsilon$.
This means that the open set $V_\epsilon := \{f_\epsilon+g:g\in U_\epsilon\}$ is a subset of 
$int(\lip^{\{\mu\}}(X,T))$, and since $f$ is the limit, in the Lipschitz topology, of members of
$\cup_{\epsilon>0} V_\epsilon$, it follows that
$f\in \overline{ int(\lip^{\{\mu\}}(X,T)) }$.
Since
$
P_{-}(X,T)
= \bigcup_{m\in\per(X,T)}  int\left( \lip^{\{m\}}(X,T)\right)$
(cf.~(\ref{pminusxt})),
we see that
$f\in \overline{P_-(X,T)}$, so (\ref{suffices_lipper_p_minus}) is proved.

It remains to justify the claim that $\mu$ is the unique $(f_\epsilon+g)$-maximizing measure
for all  $g\in U_\epsilon$, in other words that if
$\nu\in\m(X,T) \setminus \{\mu\}$ then
\begin{equation}\label{mu_nu_ineq}
\int (f_\epsilon + g)\, d\nu
<
\int (f_\epsilon + g)\, d\mu .
\end{equation}
For this first note that, in view of (\ref{g_x_bound}), (\ref{U_epsilon_defn}),
\begin{equation}\label{eps2d}
\frac{1}{q}S_q g(x) - \int g\, d\mu \le \frac{1}{q} \frac{\epsilon}{2} d(x,Q) \le \frac{\epsilon}{2} d(x,Q)
\
\text{ for all }x\in X.
\end{equation}
The $T$-invariance of $\nu$, together with (\ref{eps2d}), gives
\begin{equation}\label{inview}
\int g\, d\nu = \int \frac{1}{q} S_q g\, d\nu
= \int \left( \frac{1}{q}S_q g - \int g\, d\mu \right)\, d\nu + \int g\, d\mu
\le \frac{\epsilon}{2} \int d(\cdot,Q)\, d\nu + \int g\, d\mu,
\end{equation}
and (\ref{inview}) implies
\begin{equation}\label{inview2}
\int (f_\epsilon + g)\, d\nu
\le  \int f\, d\nu -\frac{\epsilon}{2}\int d(\cdot,Q)\, d\nu + \int g\, d\mu
<  \int f\, d\nu  + \int g\, d\mu,
\end{equation}
where the final inequality uses that $\nu\neq\mu$, therefore $\int d(\cdot,Q)\, d\nu >0$.
Now $\mu$ is a maximizing measure for both $f$ and $f_\epsilon$, so
$\int f\, d\nu \le \int f\, d\mu = \int f_\epsilon\, d\mu$, and combining this with (\ref{inview2}) gives the required
inequality (\ref{mu_nu_ineq}).

For general $(X,T)\in\D$, i.e.~without the Lipschitz assumption on $T$,
the approach of \cite{BZ15} can be used
to show that if 
$f\in \lip^\per(X,T)$
then $f\in \overline{P_-(X,T)}$,
by first noting that if $\mu\in\per(X,T)\cap\m_{\max}(f)$ then there exists $C_\mu>0$
such that $W(\mu,\nu)\le C_\mu \int d(\cdot,\supp(\mu))\, d\nu$,
where $W$ denotes the Wasserstein distance (see Definition \ref{wassersteindefn}) on $\m(X)$,
from which it follows that for $\epsilon>0$ the functions
$f_\epsilon:=f-\epsilon\, d(\cdot,\supp(\mu))$ lie in 
$int(\lip_\subseteq^{\{\mu\}}(X,T))$, and hence that
$f\in \overline{ int(\lip_\subseteq^{\{\mu\}}(X,T))} \subseteq \overline{P_-(X,T)}$.
\end{proof}

In view of Lemma \ref{four_closures}, various characterisations of 
typical periodic optimization
can be summarised as follows:

\begin{cor}\label{equivalentTPO}
For $(X,T)\in\D$, the following are equivalent:
\begin{itemize}
\item[(a)] $(X,T)$ has the TPO property.
\item[(b)] $P(X,T)$ is dense in $\lip(X)$.
\item[(c)] $\lip_{\, !}^{\per}(X,T)$  is dense in $\lip(X)$.
\item[(d)] $P_+(X,T)$ is dense in $\lip(X)$.
\item[(e)] $\lip^{\per}(X,T)$  is dense in $\lip(X)$.
\item[(f)]  $P_-(X,T)$ is dense in $\lip(X)$.
\end{itemize}
\end{cor}
\begin{proof}
Note that (b) implies (a) because $P(X,T)$ is open, 
and (a) implies (b) since $P(X,T)$ is the largest open set with the property that all its members have the periodic optimization property. 
The equivalence of (b), (d), (e) and (f) is immediate from Lemma \ref{four_closures},
and the equivalence of (c) to these follows because
$P(X,T)\subseteq \lip_{\, !}^\per(X,T)\subseteq\lip^\per(X,T)$.
\end{proof}

Clearly, if $(X,T)\in\D$ has no periodic orbits then it does not have TPO.
In fact a necessary condition for 
TPO is that periodic measures form a dense
subset of the ergodic measures:

\begin{lem}\label{periodicdenseinergodic}
If  $(X,T)\in\D$  has 
TPO
then $\per(X,T)$ is weak$^*$ dense in $\e(X,T)$.
\end{lem}
\begin{proof}
If  $(X,T)\in\D$  has the typical periodic optimization property, then $P(X,T)$ is dense in $\lip(X)$.
Therefore $P(X,T)$ is dense in $C(X)$, with respect to the uniform topology on $C(X)$, since 
$\lip(X)$ is densely embedded in $C(X)$.
This denseness means, by \cite[Lem.~3.4]{Mor10},
that the set
$\mathcal{U}=\e(X,T) \cap \left(\cup_{f\in P(X,T)} \m_{\max}(f) \right)
$
is weak$^*$ dense in $\e(X,T)$.

Now each $f\in P(X,T)$ is such that $\m_{\max}(f)$ consists of a single periodic measure,
so $\mathcal{U}\subseteq \per(X,T)$.
It follows that $\per(X,T)$ is weak$^*$ dense in $\e(X,T)$, as required.
\end{proof}

\begin{rem}\label{not_follow_dense_m}
For TPO to hold, it is not necessary that $\per(X,T)$ be dense in all of $\m(X,T)$:
see e.g.~the (topologically transitive) shift space described in
Example \ref{tpo_despite_periodics_not_dense_in_m}.
\end{rem}

Theorem \ref{densedense} and Corollary \ref{equivalentTPO} together have the following consequence:

\begin{cor}\label{nequalsper}
Let $(X,T)\in\D$,
with $Z$ a closed $T$-invariant subset of $X$.
Let $U$ be a non-empty open subset of $\lip(X)$ such that $Z$ is maximizable for all $f\in U$.
If $(Z,T)$ has TPO, then 
 $P(X,T) \cap U$ is dense in $U$.
\end{cor}
\begin{proof}
The TPO for $(Z,T)$ means that
$\lip^\per(Z,T)=\lip^{\per(Z,T)}(Z,T)$ is dense in $\lip(Z)$, by Corollary \ref{equivalentTPO}.
Define  $\n=\per(X,T)$, so that $\n\cap\m(Z,T)=\per(Z,T)$,
and $\lip^{\per}(X,T)=\lip^\n(X,T)$.
Theorem \ref{densedense}
gives that
 $\lip^\n(X,T) \cap U$ is dense in $U$,
in other words $\lip^\per(X,T) \cap U$ is dense in $U$.
But $P(X,T)$ is dense in $\lip^\per(X,T)$, by Corollary  \ref{equivalentTPO},
so $P(X,T) \cap U$ is dense in $U$, as required.
\end{proof}

Combining  Lemma \ref{baireconsequence} and Corollary \ref{nequalsper},
we obtain the following key structural theorem:

\begin{thm}\label{abstracttpovariant}  (Structural Theorem)

\noindent 
Suppose $(X,T)\in\D$.
Let $\ZZ=\{Z_0\} \cup \{Z_i\}_{i=1}^\infty$ be a countable maximizable family
 such that $(Z_i,T)$ has TPO for all $i\ge1$. 
Then the open set
$$P(X,T) \cup int\left(\lip_{Z_0}(X,T)\right)$$ 
is dense in $\lip(X)$.
\end{thm}
\begin{proof}
For each $i\ge0$, define the open set
$$U_i:=int(\lip_{Z_i}(X,T))\,.
$$
Since $\ZZ$ is countable and maximizable,
Lemma \ref{baireconsequence} implies that
\begin{equation}\label{uidenseinL_tpbo}
\bigcup_{i=0}^\infty U_i\ \text{ is dense in }\lip(X)\,.
\end{equation}
Now $\ZZ$ is a maximizable family, so each $Z_i$ is maximizable for all $f\in \lip_{Z_i}(X,T)$, and hence for all
$f\in U_i$. 
For $i\ge1$,  since 
$(Z_i,T)$ has TPO 
by hypothesis, 
 Corollary \ref{nequalsper} then implies that
\begin{equation*}
P(X,T) \cap U_i\ \text{ is dense in }U_i,\text{ for all }i\ge 1.
\end{equation*}
Therefore
$\bigcup_{i=1}^\infty P(X,T)\cap U_i$ is dense in $\bigcup_{i=1}^\infty U_i$,
or in other words
\begin{equation*}
P(X,T) \cap \left( \bigcup_{i=1}^\infty  U_i\right) \text{ is dense in } \bigcup_{i=1}^\infty U_i.
\end{equation*}
It follows that
\begin{equation}\label{u_zero_p}
U_0 \cup \big( P(X,T) \cap \left( \bigcup_{i=1}^\infty  U_i\right) \big) \text{ is dense in } \bigcup_{i=0}^\infty U_i.
\end{equation}
Combining (\ref{uidenseinL_tpbo})
and (\ref{u_zero_p}), we see that
\begin{equation*}
U_0 \cup P(X,T) \text{ is dense in } \lip(X),
\end{equation*}
in other words $int(\lip_{Z_0}(X,T)) \cup P(X,T)$ is dense in $\lip(X)$, as required.
\end{proof}

If a dynamical system $(X,T)$ has a countable maximizable family $\ZZ=\{Z_0\} \cup \{Z_i\}_{i=1}^\infty$ as in Theorem \ref{abstracttpovariant}, the question of whether or not $(X,T)$ has TPO focuses attention on the invariant set $Z_0$.
We shall refer to $Z_0$ as the \emph{boundary}\footnote{The terminology is motivated by the case of symbolic dynamics,
where we show (see Sections \ref{symbolicdynamicssection} and \ref{SFTmaximizablefamily}) that there is a canonical choice of countable maximizable family $\ZZ=\{Z_0\} \cup \{Z_i\}_{i=1}^\infty$ in which the role of the boundary $Z_0$ is played by the \emph{Markov boundary} of the shift space $X$.}
 of $\ZZ$, and say that the boundary is \emph{fragile} if
$\lip_{Z_0}(X,T)$ has empty interior in $\lip(X)$.
By the structural theorem, a fragile boundary means that $(X,T)$ has the TPO property:

\begin{thm}\label{fragile_boundary_theorem} (Fragile boundary implies TPO)

\noindent
Suppose $(X,T)\in\D$.
Let $\ZZ=\{Z_0\} \cup \{Z_i\}_{i=1}^\infty$ be a countable maximizable family
 such that $(Z_i,T)$ has TPO for all $i\ge1$,
and $\lip_{Z_0}(X,T)$ has empty interior in $\lip(X)$.
Then $(X,T)$ has TPO.
\end{thm}
\begin{proof}
If $int(\lip_{Z_0}(X,T))=\emptyset$ then Theorem \ref{abstracttpovariant} implies that $P(X,T)$ is dense in
$\lip(X)$, so $(X,T)$ has TPO (cf.~Corollary \ref{equivalentTPO}).
\end{proof}

The simplest example of a fragile boundary is of course when this boundary is empty, 
prompting the following definition:

\begin{defn}\label{CMFTPO}
Suppose $(X,T)\in\D$.
A countable maximizable family $\ZZ=\{Z_i\}_{i=1}^\infty$ is said to be a
\emph{countable maximizable family with TPO} 
if
$(Z_i,T)$ has
TPO
for all $i\in\N$.
\end{defn}

\begin{thm}\label{abstracttpo} (Countable maximizable family with TPO implies TPO)

\noindent Suppose $(X,T)\in\D$.
If $\ZZ=\{Z_i\}_{i=1}^\infty$ is a countable maximizable family with TPO,
then $(X,T)$ has TPO. 
\end{thm}
\begin{proof}
By setting
$Z_0=\emptyset$
in  Theorem  \ref{abstracttpovariant},
the open set $P(X,T)$ is seen to be dense in $\lip(X)$,
so $(X,T)$ has TPO.
\end{proof}

\begin{rem}\label{non_trivial_CMF}
Theorems \ref{fragile_boundary_theorem} and \ref{abstracttpo} provide a roadmap to establishing TPO for dynamical systems $(X,T)$ admitting a \emph{non-trivial} countable maximizable family with either TPO or, more generally, a fragile boundary.
As noted after Definition \ref{maximizablefamily}, we shall not be concerned with trivial maximizable families: for example
Theorem \ref{abstracttpo} has a tautological interpretation with the trivial choice $\ZZ=\{X\}$
(if $T:X\to X$ is surjective), which clearly cannot yield new information. Note that Theorem \ref{abstracttpo} is formally a special case of Theorem \ref{fragile_boundary_theorem},
and one may wonder whether there even exist dynamical systems with fragile boundary but without a (non-trivial) countable maximizable family with TPO;
we shall see in Section \ref{specimen_section} that such systems do exist, and have a rather different character from those
that can be handled using Theorem \ref{abstracttpo}. 
\end{rem}

\section{ The subset closing property}\label{closingpropertysection}

While the development so far has been in the context of general topological dynamical systems,
henceforth we shall be interested in restricting attention to systems satisfying particular assumptions.
We introduce the following weak hyperbolicity condition, that will be useful in the sequel: 

\begin{defn}\label{closingpropertydefn}
Suppose $(X,T)\in\D$.
A non-empty subset $A\subseteq X$ has the \emph{closing property} relative to $(X,T)$ if there exists $C_A>0$
such that for all $x\in A$, if $T^n(x)\in A$ for some $n\in\N$, then there exists a periodic point $z=T^n(z)\in X$ such that
\begin{equation}\label{sumofdistances}
\sum_{i=0}^{n-1} d\left(T^i(x), T^i(z)\right) \le C_A\,.
\end{equation}
A closed $T$-invariant subset $Y\subseteq X$ is said to have the \emph{subset closing property} if 
there exists a non-empty subset $A\subseteq Y$, where $A$ is open in $Y$,
such that $A$ has the closing property relative to $(X,T)$.
\end{defn}

\begin{rem}
If $(X,T)\in\D$ is open and distance-expanding, or a hyperbolic homeomorphism with local product structure
(such as an Axiom A diffeomorphism), then every closed invariant $Y\subseteq X$ has the subset closing property: 
the existence of the periodic point is a consequence of Anosov's closing lemma
(see e.g.~\cite[Thm.~6.4.15]{KH95}, \cite[Cor.~4.3.7]{URM22}), and the exponential instability of orbits
in such systems (cf.~\cite[Cor.~6.4.17]{KH95}) implies the uniform bound (\ref{sumofdistances}).
Beyond uniform hyperbolicity, various other systems have (non-periodic) invariant sets with the subset closing property: some of these will be useful in Section \ref{symbolicdynamicssection}, in the context of symbolic dynamics. 
\end{rem}

The significance of the subset closing property is in connection with minimax sets.
By combining the Minimax Birkhoff Bound (Proposition \ref{minimaxnegativeBirkhoffcor}),
and the Completely Maximizing Criterion (Proposition \ref{morrisminimaxcompletely}),
we are able to derive the existence of a completely maximizing set:

\begin{thm}\label{minimaxclosingcompletelymaximizing}
Suppose $(X,T)\in\D$, 
and that $f:X\to\R$ is Lipschitz.
If $Y\subseteq X$ is a minimax set for $f$, and $(Y,T)$ has the subset closing property,
then $Y$ is completely maximizing for $f$, and is dynamically minimal.
\end{thm}
\begin{proof}
Without loss of generality, assume that $\beta(f)=0$.
To prove the theorem it suffices to show that
\begin{equation}\label{aimtoshow}
\sup_{n\ge 1} \sup_{x\in Y} S_nf(x) <\infty\,,
\end{equation}
since in this case the fact that $Y$ is minimax for $f$ implies that
$Y$ is completely maximizing for $f$, and is dynamically minimal,
by Proposition \ref{morrisminimaxcompletely}.

Since $(Y,T)$ has the subset closing property, there exists a non-empty subset $A\subseteq Y$, that is open with respect to $Y$,
and has the closing property relative to $(X,T)$. Let 
$N_A\in\N$ be as in 
the Minimax Birkhoff Bound
(Proposition \ref{minimaxnegativeBirkhoffcor}),
and let $C_A>0$ be as in Definition \ref{closingpropertydefn}.

Let $x\in Y$, $n\in\N$.
If $T^i(x)\notin A$ for $0\le i\le n-1$ then
Proposition \ref{minimaxnegativeBirkhoffcor} gives 
\begin{equation}\label{headestimatenonly}
\sum_{i=0}^{n-1} f(T^i(x)) \le N_A\|f\|_{C(X)}\,.
\end{equation}
If on the other hand
$
\{ 0\le i\le n-1: T^i(x) \in A\}
$
is non-empty, then let $n_1$ denote its minimum element, and $n_2$ its maximum element,
so that $0\le n_1 \le n_2 \le n-1$.
We will derive an upper bound on the sum
$
 S_n f(x) = \sum_{i=0}^{n-1}f(T^i(x))
$
by estimating 
in turn
the three sums\footnote{If $n_1=n_2$ then by convention we set $\sum_{i=n_1}^{n_2-1} f(T^i(x))=0$.}
$\sum_{i=0}^{n_1-1} f(T^i(x))$, 
$\sum_{i=n_1}^{n_2-1} f(T^i(x))$,
and
$\sum_{i=n_2}^{n-1} f(T^i(x))$. 

Firstly, since $T^i(x)\notin A$ for $0\le i\le n_1-1$,
Proposition \ref{minimaxnegativeBirkhoffcor} gives 
\begin{equation}\label{headestimate}
\sum_{i=0}^{n_1-1} f(T^i(x)) \le N_A\|f\|_{C(X)}\,.
\end{equation}

Secondly, since $y:=T^{n_1}(x)\in A$, and $T^{n_2-n_1}(y)=T^{n_2}(x)\in A$,
and since $A$ has
the closing property relative to $(X,T)$, there exists a periodic point $z=T^{n_2-n_1}(z)\in X$ 
such that
\begin{equation}\label{sumofdistancesintheorem}
\sum_{i=0}^{n_2-n_1-1} d\left(T^i(y), T^i(z)\right) \le C_A\,.
\end{equation}
The integral of $f$ with respect to the periodic measure supported by the orbit of $z$ is non-positive,
since
$\beta(f)=0$,
in other words
\begin{equation}\label{gperiodicnonpositive}
\sum_{i=0}^{n_2-n_1-1}f(T^i(z)) \le 0\,.
\end{equation}
It follows that
\begin{equation}\label{fboundprelip}
\sum_{i=n_1}^{n_2-1} f(T^i(x))
=
\sum_{i=0}^{n_2-n_1-1} f(T^i(y))
\le
\sum_{i=0}^{n_2-n_1-1} \left(  f(T^i(y))-  f(T^i(z)) \right).
\end{equation}
Using the Lipschitz property of $f$  in (\ref{fboundprelip}),
and then using (\ref{sumofdistancesintheorem}),
 gives
\begin{equation}\label{middlefestimate}
\sum_{i=n_1}^{n_2-1} f(T^i(x))
\le |f|_{\lip(X)} 
\sum_{i=0}^{n_2-n_1-1} d(T^i(y),T^i(z))
\le  C_A  |f|_{\lip(X)}
 \,.
\end{equation}

Thirdly, note that $T^{n_2}(x)\in A$, but $T^i(x)\notin A$ for $n_2+1\le i\le n-1$,
so
\begin{equation}\label{tailfbound}
\sum_{i=n_2}^{n-1} f(T^i(x)) = f(T^{n_2}(x)) + \sum_{i=n_2+1}^{n-1}f(T^i(x)) \le (1+N_A) \|f\|_{C(X)},
\end{equation}
by Proposition \ref{minimaxnegativeBirkhoffcor}.

Combining (\ref{headestimate}), (\ref{middlefestimate}) and (\ref{tailfbound}) gives
\begin{equation}\label{mainsnfboundexplicit}
S_n f(x) \le C_A  |f|_{\lip(X)} + (1+2N_A) \|f\|_{C(X)}.
\end{equation}
Since (\ref{headestimatenonly}) and (\ref{mainsnfboundexplicit}) together account for all $x\in Y$ and $n\in\N$, 
the desired bound
(\ref{aimtoshow}) holds.
\end{proof}

\section{ Symbolic dynamics}\label{symbolicdynamicssection}

Up until now we have developed a theory of countable maximizable families, with application to the TPO problem,
in the setting of rather general dynamical systems.
For the remainder of this article, however, our attention will be focused on \emph{symbolic} dynamical systems.
For convenience we consider one-sided shift spaces, and will in particular
exploit the fact that Contreras' TPO Theorem
(see Lemma \ref{contrerasSFTlemma}) is valid for shifts of finite type.
Analogues of all results in Sections \ref{symbolicdynamicssection} to \ref{magicmorsesection}
also hold for two-sided shift spaces, with proofs that differ only at the level of notation (and
are therefore omitted), and are underpinned by the analogue of Contreras' Theorem for two-sided shifts of finite type,
proved in \cite{HLMXZ19} (cf.~the remarks following the proof of Lemma \ref{contrerasSFTlemma}).

\begin{notation}\label{shift_space_notation} (Topology, metric, shift map, subshifts)

For a finite non-empty set $A$ (referred to as an \emph{alphabet}), the \emph{full shift} on $A$ is the set 
$$A^\N=\{ (x_i)_{i=1}^\infty : x_i\in A\text{ for all }i\ge 1\}\,,$$ 
which is compact when equipped with the product topology.

For any $\alpha\in(0,1)$, 
the space $A^\N$ is metrized by the distance function
$d=d_\alpha$ defined by
$$d_\alpha(x,y)=\alpha^{N(x,y)},$$
where $N(x,y):=\inf\{i\in\N: x_i \neq y_i \} -1$.

The \emph{shift map} $\sigma: A^\N\to A^\N$, defined by $(\sigma(x))_i=x_{i+1}$ for all $i\ge1$, is continuous.
A \emph{shift space} on $A$, also referred to as a \emph{subshift of $A^\N$}, is a closed 
subset
$X\subseteq A^\N$, which is shift-invariant in the sense that $\sigma(X)=X$.
More generally if $X$ and $Y$ are shift spaces, with $Y\subseteq X$, we say that $Y$ is a \emph{subshift of $X$}.
It will be convenient to not insist that a shift space, or a subshift, be non-empty: in what follows, various 
canonically defined subshifts of $X$ will be investigated, which in certain circumstances are empty.

For a given shift space $X$, we write $\sigma$ for the restricted shift map $\sigma|_X$,
and
 regard $(X,\sigma)$ as a topological dynamical system.
\end{notation}

\begin{rem}
For any $\alpha\in(0,1)$, and $\gamma\in(0,1]$, the space of $\gamma$-H\"older functions on a shift space $X$,
with respect to the distance function $d_\alpha$, equipped with the usual H\"older norm, is precisely the space of Lipschitz
functions on $X$, with respect to the distance function $d_{\alpha^\gamma}$.
Consequently, although all of the results in the rest of the article are stated in terms of the space of Lipschitz functions,
for any of the distance functions $d_\alpha$, the same results hold for any space of $\gamma$-H\"older functions, 
with respect to the same family of distance functions.
\end{rem}

\begin{notation} (Words, languages, cylinder sets)

For an alphabet $A$, a \emph{word} or \emph{block} on $A$ will mean a finite string $w=a_1\ldots a_n \in A^n$, $n\ge 0$,
where $|w|=n$ is its \emph{length}.
Let $A^*$ denote the set of all words on $A$ (including the empty word $\epsilon$, i.e.~the unique word of length 0).

If $x\in A^\N$, we define
\begin{equation}\label{xij}
x_{[i,j]}=x_ix_{i+1}\ldots x_j\,,
\end{equation}
with the convention that $x_{[i,j]}=\epsilon$ if $i>j$, and
a word $w\in A^*$ is
said to \emph{appear} in a sequence $x \in A^\N$
if $x_{[i,j]} = w$ for some $i,j$.
We write
\begin{equation}\label{wappearsinx}
w\subset x
\end{equation}
to mean that $w$ appears in $x$.
For a shift space $X$ on $A$, the word $w\in A^*$ is \emph{allowed (in $X$)}
if it appears in some $x\in X$.
The \emph{language} of $X$, denoted $\w(X)$, is the set of all its allowed words.
For $w \in\w(X)$, the corresponding \emph{cylinder set} $[w]$ is defined
by
$$
[w]= \left\{x\in X: x_{[1,|w|]}=w\right\},
$$
and for subshifts $Z\subseteq X$ we write
$$
[w]_Z= [w]\cap Z =\left\{x\in Z: x_{[1,|w|]}=w\right\}.
$$
For $n\in \N_0$, let $\w_n(X)$ denote the set of length-$n$ words in $\w(X)$,
and for $S\subseteq\N_0$, denote
\begin{equation}\label{S_length_words}
\w_S(X):= \bigcup_{n\in S} \w_n(X).
\end{equation}

We extend the notation (\ref{xij}), (\ref{wappearsinx}) to words $w=a_1\ldots a_n\in A^*$, using
$w_{[i,j]}$ to denote the \emph{sub-word} $v=a_i\ldots a_j\in A^*$,
in which case $v$ is said to \emph{appear} in $w$,
written as $v\subset w$.
Each sub-word $w_{[1,k]}$, for $1\le k\le |w|$, is called a \emph{prefix} of $w$,
and each sub-word $w_{[k,|w|]}$, for $1\le k\le |w|$, is called a \emph{suffix} of $w$.

Given $u, v\in A^*$, the \emph{concatenation} $uv$ denotes the word $w\in A^*$ 
of length $|w|=|u|+|v|$, with $w_{[1,|u|]}=u$ and $w_{[|u|+1,|u|+|v| ]}=v$,
and if $x\in A^\N$ then $ux$ denotes the sequence $y\in A^\N$ with
$y_{[1,|u|]}=u$ and $\sigma^{|u|}(y)=x$.
If $n\in\N$ then $u^n$ denotes the word $w$ with $x_{[i|u|+1,(i+1)|u|]}=u$
for $0\le i\le n-1$,
and $u^\infty$ denotes the periodic sequence $x\in A^\N$ with $x_{[i|u|+1,(i+1)|u|]}=u$
for all $i\ge 0$; in this latter case, $u$ is said to be a \emph{periodic word} for $x$,
and also for the periodic orbit of $x$ under $\sigma$, and for the periodic measure
$ \frac{1}{|u|} \sum_{j=0}^{|u|-1} \delta_{\sigma^j(x)}$
 supported by this orbit.
\end{notation}

\begin{notation}
Let $X$ be a shift space with language $\w(X)$. For
$w\in \mathbb{W}(X)$, its \emph{predecessor set} $P_X(w)$ is defined by 
$$P_X(w):=\{u\in \mathbb{W}(X): uw\in\w(X)\}\,,$$
and its \emph{follower set} $F_X(w)$ is defined by 
$$F_X(w):=\{u\in \mathbb{W}(X): wu\in\w(X)\}\,.$$
\end{notation}

\begin{lem}\label{lem-00}
	Let $X$ be a shift space, and suppose $w\in\w(X)$. Then: 
\begin{itemize}
		\item[(a)]$F_X(uw)\subseteq F_X(w)$ for all $u\in P_X(w)$.
		\item[(b)]If $u\in P_X(w)$ is such that $F_X(uw)= F_X(w)$, then $F_X(uwv)=F_X(wv)$ for all $v\in F_X(w)$.
	\end{itemize}
\end{lem}
\begin{proof}
Straightforward.
\end{proof}

\begin{notation}
For a shift space $X$, define
\begin{equation}\label{equ:inftywords}
\w_\infty(X)=\left\{w\in \w(X) : |\{F_X(uw): u\in P_X(w)\}|=\infty\right\}\,,
\end{equation}
the set of allowed words with infinitely many follower sets,
and
\begin{equation}\label{equ:finitewords}
\mathbb{W}_{<\infty}(X):=\{w\in \mathbb{W}(X) : |\{F_X(uw): u\in P_X(w)\}|<\infty\}\,,
\end{equation}
the set of allowed words with finitely many follower sets.
\end{notation}

The following notion was introduced by Thomsen \cite{Tho06}.

\begin{defn}
For a shift space $X$,
define its \emph{Markov boundary}, denoted $\partial_M X$, by
\begin{equation}\label{equ:markovboundary}
\partial_MX:=\{x\in X:  w\in \mathbb{W}_\infty(X) \text{ for all finite words $w$ that appear in } x\}.
\end{equation}
\end{defn}

\begin{lem} For any shift space $X$, its Markov boundary  
$\partial_MX$ is a (possibly empty) subshift of $X$.
\end{lem}
\begin{proof}
See \cite[Lem.~2.1]{Tho06}.
\end{proof}

Recall that a shift space $X$ is \emph{sofic} if the collection $\{F_X(w):w\in \w(X)\}$ of distinct follower sets is finite
(or, equivalently, if the collection of its distinct predecessor sets is finite).
There are a number of equivalent definitions of sofic systems, some dating back to the original article of Weiss
\cite{Wei73}, including the characterisation of sofic shifts as precisely the continuous factors of shifts of finite type;
see \cite[\S 3.2]{Bru22}, \cite[\S 6.1]{Kit98} and \cite[Ch.~3]{LM95} for more on sofic shift spaces.
The following characterisation will be important for our purposes:

\begin{lem}\label{sofic_iff_empty_boundary}
A shift space $X$ is sofic if and only if $\partial_M X=\emptyset$.
\end{lem}
\begin{proof}
See \cite[Thm.~2.3]{Tho06}.
\end{proof}

\begin{defn}\label{sync_defn}
Given a shift space $X$, a word $w\in\w(X)$ is called \emph{synchronizing} for $X$ if whenever
$u\in P_X(w)$ and $v\in F_X(w)$, then $uwv\in\w(X)$.
\end{defn}

\begin{rem}
The notion of synchronization was introduced by Blanchard \& Hansel \cite{BH86} in the context of coded shift spaces, which are necessarily topologically transitive. 
Some authors say that a shift space is \emph{synchronized} if it has a synchronizing word,
but many authors follow \cite{BH86} in reserving this term only for those shift spaces that are in addition 
topologically transitive.
In this article we do not require topological transitivity, and in order to avoid confusion we will simply say that $X$ \emph{has a synchronizing word} when this is the case.
\end{rem}

In Section \ref{periodictposection} we showed, in Lemma \ref{periodicdenseinergodic},
that a necessary condition for TPO is that
periodic measures are dense in the set of ergodic measures.
It was also noted, in Remark \ref{not_follow_dense_m}, that the density of periodic measures in the set of all invariant measures
is \emph{not} a necessary condition for TPO; having introduced the relevant symbolic dynamics notation
in this section, we are now able to justify this statement via the following example, 
mentioned in Remark \ref{not_follow_dense_m}, of a topologically transitive shift space $X$ with TPO but with its periodic
measures not weak$^*$ dense in $\m(X,\sigma)$.

\begin{exam}\label{tpo_despite_periodics_not_dense_in_m}
Let $A=\{a,b\}$, and define $X$ to be the closure of the set of shift iterates of sequences of the form
$a^{n_1}b^{n_1}a^{n_2}b^{n_2}a^{n_3}b^{n_3}\ldots$
where $n_i\in\N$, $i\in\N$.
The shift space $X$ is topologically transitive, 
with synchronizing word $ba$,
and 
$\partial_M X = \{a^\infty, b^\infty\} \cup \{a^nb^\infty : n\ge 1\}$,
and since $\partial_M X$ is a shift of finite type, we will see  that $(X,\sigma)$ has TPO
(either by Theorem \ref{partialX_tpo_implies_X_tpo} and the fact that shifts of finite type have TPO, or alternatively because $X$ is eventually sofic, cf.~Definition \ref{presofic_defn}, Theorem \ref{pre_sofic_theorem}).
The set $\per(X,\sigma)$ of periodic measures is dense in the set $\e(X,\sigma)$
of ergodic measures
(indeed this follows from the fact that $X$ has TPO, cf.~Lemma \ref{periodicdenseinergodic}),
however, as noted in Remark \ref{not_follow_dense_m}, $\per(X,\sigma)$ is not dense in all of $\m(X,\sigma)$.
To see this, note that
if $\mu$ is a periodic measure on $X$, supported by a periodic orbit other than
the fixed points $a^\infty$ and $b^\infty$, then $\mu([a])=1/2=\mu([b])$;
consequently, for $\epsilon\in (0,1/2)\cup (1/2,1)$, the (invariant but non-ergodic) measure $\epsilon \delta_{a^\infty}+(1-\epsilon)\delta_{b^\infty}$ is not a weak$^*$ limit of periodic measures.
\end{exam}

\begin{lem}\label{uwsinz}
Let $X$ be a shift space, and suppose  $w\in\w(X)$ is a synchronizing word for the subshift $Z\subseteq X$.
If $x,y\in Z$ can be written as $x=uwr$, $y=vws$, for $u,v\in \w(Z)$, and $r,s\in Z$, then $uws\in Z$.
\end{lem}	
\begin{proof}
Now $uw\subset x\in Z$, so $uw\in\w(Z)$, and if $n\in\N$ then $ws_{[1,n]}\subset y\in Z$, so $ws_{[1,n]}\in \w(Z)$.
Since $w$ is a synchronizing word for $Z$ then $uws_{[1,n]}\in \w(Z)$ for all $n\in\N$.
So there exists $z^{(n)}\in Z$ such that $uws_{[1,n]} \subset z^{(n)}$, and we may assume that
$z^{(n)}_{[1,|u|+|w|+n]}=uws_{[1,n]}$, so that $z^{(n)}\to uws$ as $n\to\infty$, and therefore $uws\in Z$ since $Z$ is closed.
\end{proof}

\begin{defn}\label{d-1} Let $X$ be a shift space. For $w\in\w(X)$,
define $X_w$ to be 
	the set of all $x\in X$ 
such that  if $t\in\w(X)$ satisfies $tw \subset x$, then $F_X(tw)=F_X(w)$.
\end{defn}

\begin{rem}\label{xwremark}
Note that 
$X_w$ contains all those $x\in X$ such that $w$ does not appear in $x$.
In general $X_w$ may be the empty set (this is the case if $X$ has no synchronizing word),
and may equal $X$ (if $w$ is a synchronizing word for $X$).
\end{rem}

\begin{lem}\label{xwiswsynchronizing}
For any shift space $X$,
every $w\in \w(X)$
is a synchronizing word for $X_w$.
\end{lem}
\begin{proof}
Suppose $u,v\in\w(X_w)$ are such that $uw\in\w(X_w)$ and $wv\in\w(X_w)$, so that there exist $x,y\in X_w$ with $uw\subset x$ and $wv\subset y$.
Writing 
\begin{equation}\label{ydef}
y=rwvs
\end{equation}
 for some $r\in\w(X_w)$ and $s\in X_w$, we claim that
\begin{equation}\label{zdef}
z:=uwvs
\end{equation}
belongs to $X_w$, which will then mean that $uwv\in \w(X_w)$, and hence that $w$ is a synchronizing word for $X_w$.

Since $w$ appears in $z$, to check that $z\in X_w$ amounts to showing that if $tw\subset z$ for some $t\in \w(X)$, then
$F_X(tw)=F_X(w)$.
For this we write 
\begin{equation}\label{tuwexplicit}
t=t_1\ldots t_k, \quad u=u_1\ldots u_l , \quad w=w_1\ldots w_m,
\end{equation} 
and consider three cases separately.

In the first case, suppose $z_{[i+1,i+k+m]}=tw$ for some $i\ge l$. From (\ref{ydef}), (\ref{zdef}), (\ref{tuwexplicit}), this means that
$tw\subset wvs=\sigma^{|r|}(y)$.
But $\sigma^{|r|}(y)$ belongs to $X_w$, since $y\in X_w$,
so the definition of $X_w$ means that
$F_X(tw)=F_X(w)$, as required.

In the second case, suppose that $k\le l$, and  $z_{[i+1,i+k+m]}=tw$ for some $0\le i\le l-k$.
From (\ref{zdef}), (\ref{tuwexplicit}), this means that
$tw\subset uw$.
But $uw\subset x$, therefore $tw\subset x\in X_w$, and the definition of $X_w$ then means that
 $F_X(tw)=F_X(w)$, as required.

In the third case, suppose that $tw\subset z$, but that the conditions of the first and second cases do not hold.
In particular, $z_{[i+1,i+k+m]}=tw$ for some $i<l$, so there exists $1\le j\le k$ such that
the length-$j$ prefix of $t$ is the length-$j$ suffix of $u$, in other words
\begin{equation}\label{tu}
t_1\ldots t_j = u_{l-j+1}\ldots u_l\,.
\end{equation}
Now $u_{l-j+1}\ldots u_l w \subset uw\subset x\in X_w$, so the definition of $X_w$ gives
\begin{equation}\label{uwxstar}
F_X(u_{l-j+1}\ldots u_l w) = F_X(w),
\end{equation}
and combining (\ref{tu}), (\ref{uwxstar}) gives
\begin{equation}\label{twostartw}
F_X(t_1\ldots t_j w)=F_X(w).
\end{equation}
Now let $q\in\w(X)$ be such that
\begin{equation}\label{twq}
t_1\ldots t_j wq = tw\,.
\end{equation}
From (\ref{twostartw}) and Lemma \ref{lem-00},
\begin{equation}\label{twwq}
F_X(t_1\ldots t_j wq) = F_X(wq),
\end{equation}
and combining  (\ref{twq}) with (\ref{twwq}) gives
\begin{equation}\label{twwq2}
F_X(tw)=F_X(wq).
\end{equation}
Now $wq\subset wvs = \sigma^{|r|}(y)\in X_w$, and (\ref{twq}) means that $w$ is a suffix of $wq$,
so the definition of $X_w$ means that 
\begin{equation}\label{5starwqw}
F_X(wq)=F_X(w).
\end{equation}
Combining (\ref{twwq2}) and (\ref{5starwqw}) gives $F_X(tw)=F_X(w)$, as required.
\end{proof}

In general, the presence of a synchronizing word has consequences regarding periodic points:

\begin{lem}\label{synchronizingwordperiodicpoint}
Suppose $Z$ is a shift space with synchronizing word $w\in\w(Z)$.
If $x=(x_i)_{i=1}^\infty\in Z$, with $x\in [w]$ and $\sigma^n(x)\in [w]$ for some $n\in\N$,
then the period-$n$ point
$$
 (x_1\ldots x_n)^\infty
$$
belongs to $Z$.
\end{lem}
\begin{proof}
The hypotheses mean that $w$ is both a prefix and a suffix of the word $x_{[1,n+|w|]}$, so
defining 
$u=x_{[1,n]}=x_1\ldots x_n$
and
$v=x_{[|w|+1,|w|+n]}=x_{|w|+1}\ldots x_{|w|+n}
$
gives that
\begin{equation}\label{uwwv}
uw = x_{[1,n+|w|]} = wv\,.
\end{equation}
In particular, (\ref{uwwv}) means that 
\begin{equation}\label{bothbelong1}
uw\in \w(Z) 
\end{equation}
and
\begin{equation}\label{bothbelong2}
wv\in\w(Z)\,.
\end{equation}
Since $w$ is synchronizing, (\ref{bothbelong1}) and  (\ref{bothbelong2}) 
imply that $uwv\in\w(Z)$.
But  (\ref{uwwv}) means that $uwv$ can be written as $uuw$, hence 
\begin{equation}\label{usquaredbelongs}
u^2w=uuw\in\w(Z)\,.
\end{equation}
Combining (\ref{bothbelong2}) and (\ref{usquaredbelongs}), and the fact that $w$ is synchronizing, gives that
$uuwv\in\w(Z)$, and rewriting this using (\ref{uwwv}) we see that $u^3w=uuuw\in\w(Z)$.
Continuing inductively, we see that $u^mw\in \w(Z)$ for all $m\ge1$, and hence that $u^\infty\in Z$.
In other words, $(x_1\ldots x_n)^\infty\in Z$, as required.
\end{proof}

An easy consequence is the following:

\begin{cor}\label{synchronizedclosingproperty}
If $Z$ is a shift space with synchronizing word $w$,
then the cylinder set $[w]_Z$ has the closing property relative to $(Z,\sigma)$.
\end{cor}
\begin{proof}
By Lemma \ref{synchronizingwordperiodicpoint},
if $x\in [w]_Z$, and $\sigma^n(x)\in [w]_Z$ for some $n\in\N$,
then the period-$n$ point
$
 z=(x_1\ldots x_n)^\infty
$
belongs to $Z$, and moreover satisfies
\begin{equation}
\label{shadowingshiftiteratescor}
d_\alpha(\sigma^i(x),\sigma^i(z)) \le \alpha^{n-i} d_\alpha(\sigma^n(x),z)\ \text{ for }0\le i\le n-1\,.
\end{equation}
It follows that
\begin{equation*}
\sum_{i=0}^{n-1} d_\alpha(\sigma^i(x),\sigma^i(z)) < \frac{\alpha}{1-\alpha}\,,
\end{equation*}
so (\ref{sumofdistances}) holds with $C_{[w]_Z}= \frac{\alpha}{1-\alpha}$.
\end{proof}

As a consequence we deduce:

\begin{cor}\label{l-s-a}
If $X$ is a shift space, and $w\in\w(X)$, and $Y\subseteq X_w$ 
is a subshift such that $w\in\w(Y)$,
then $Y$ has the subset closing property. 
\end{cor}
\begin{proof}
Since $w$ is a synchronizing word for $X_w$, 
Corollary \ref{synchronizedclosingproperty} implies that
the cylinder set $[w]_{X_w}$ has the closing property relative to $(X_w,\sigma)$.
If $w\in\w(Y)$ then the non-empty set $[w]_Y$ is open in $Y$, and has the closing property relative to $(X_w,\sigma)$
since it is a subset of $[w]_{X_w}$, which has the closing property relative to $(X_w,\sigma)$.
Therefore $Y$ has the subset closing property.
\end{proof}

Any subshift that is not contained in $\partial_M X$ necessarily has a synchronizing word, and indeed is a subset of some $X_w$: 

\begin{lem}\label{l-1} Let $X$ be a shift space. If $Y\subseteq X$ is any subshift that is not contained in the Markov
boundary $\partial_M X$, then 
there exists $w\in\w(Y)$
such that
$
Y \subseteq X_w
$,
and in particular $Y$ has a synchronizing word.
\end{lem}
\begin{proof}Since $Y\setminus \partial_MX\neq\emptyset$, 
there exists  $w_0\in \mathbb{W}(Y)\cap \mathbb{W}_{<\infty}(X)$.
	Assume, for a contradiction, that $Y$ is not 
contained in $X_w$
for any $w\in\w(Y)$.
This means that there does not exist $w\in \mathbb{W}(Y)$ such that $F_X(w)=F_X(uw)$
	for all 
$ u\in P_Y(w)$.
So there exists a sequence
$(w_i)_{i=1}^\infty$, where each $w_i\in \w(Y)$, such that
for all $i\in \mathbb{N}$,
the word $w_{i-1}$ is a suffix of $w_{i}$,
and 
	  $F_X(w_{i})$ is a proper subset of $F_X(w_{i-1})$.
This means that 
	 $ |\{F_X(uw_0): u\in P_Y(w_0)\}|\ge  |\{F_X(w_i): i\in\mathbb{N}\}|=\infty$, 
thus contradicting the fact that $w_0\in \mathbb{W}_{<\infty}(X)$. 
So in fact $Y$ is 
contained in $X_w$
for some $w\in\w(Y)$.
Since $w$ is a synchronizing word for $X_w\supseteq Y$, and 
$w\in\w(Y)$, then $w$ is synchronizing for $Y$.
	\end{proof}

Finally, we are able to connect the concepts from Sections \ref{maximizableminimaxsection} and  \ref{completelymaximizingsetssection} on maximizable and completely maximizing sets,
with the
above results on the subshifts $X_w$ and the Markov boundary,
to prove the following important result
about those Lipschitz functions without any maximizing measure supported within the Markov boundary:

\begin{thm}\label{l-6} If $X$ is a shift space, and $f\in\lip(X)$ is such that $\partial_M X$ is not $f$-maximizable,
then $f$ has a completely maximizing set $Y$ that is dynamically minimal.
For any such 
$Y$, there exists
$w\in\w(Y)$
such that 
$Y\subseteq X_w$.
\end{thm}
\begin{proof}
By Lemma \ref{minmaxexistence}, there exists a minimax set $Y$ for $f$;
in particular, there is some $f$-maximizing measure $\mu$ with $\supp(\mu)=Y$.
The subshift $Y$ is not a subset of $\partial_M X$, since otherwise $\mu$ would be supported inside the Markov boundary, contradicting the assumption that $f$ has no such maximizing measure.
But since $Y$ is not a subset of $\partial_M X$,  Lemma \ref{l-1} implies that there exists $w\in\w(Y)$
such that 
 $Y\subseteq X_w$,
and this inclusion means that $Y$ has the subset closing property, by Corollary \ref{l-s-a}.

As a minimax set with the subset closing property, $Y$ must be completely maximizing for $f$,
and dynamically minimal,
by Theorem \ref{minimaxclosingcompletelymaximizing}.
\end{proof}

\begin{cor}
If $X$ is a sofic shift, then every $f\in\lip(X)$
has a completely maximizing set that is dynamically minimal.
\end{cor}
\begin{proof}
Since $X$ is sofic, Lemma \ref{sofic_iff_empty_boundary} implies that $\partial_M X=\emptyset$,
therefore $f\in\lip(X)$ does not have a maximizing measure supported inside $\partial_M X$,
i.e.~$\partial_M X$ is not $f$-maximizable.
The result then follows from Theorem \ref{l-6},   
\end{proof}

\begin{rem}
The hypothesis in Theorem \ref{l-6},
that $f$ should not have any maximizing measure supported inside $\partial_M X$,
cannot be omitted.
For example let the shift space $X$ be dynamically minimal, but supporting precisely two ergodic invariant measures $\mu$, $\mu'$; such an $X$ has no synchronizing word (cf.~Lemma \ref{synchronizingwordperiodicpoint}), 
and therefore $\partial_M X = X$.
As noted in Remark \ref{completelymaximizingremark}, provided $f\in C(X)$
satisfies $\int f\, d\mu \neq \int f\, d\mu'$ then there cannot be a completely maximizing subset for $f$,
and since $\lip(X)$ is densely embedded in $C(X)$, such an $f$ may be chosen to be Lipschitz.
\end{rem}

\section{ Typical periodic optimization for sofic and eventually sofic shifts}\label{SFTmaximizablefamily}

In this section we will be able to combine 
Theorem \ref{l-6} with the framework of countable
maximizable families developed in Section \ref{maximizablefamiliessection},
in order to prove various theorems announced in Section \ref{introsection}.
So as to introduce a suitable countable maximizable family, we first define:

\begin{defn}
Let $Z$ be a shift space with synchronizing word $w\in\w(Z)$.
For $q\in\N$, define the associated \emph{$(w,q)$-step subshift} $Z_{w,q}\subseteq Z$ by
\begin{equation}\label{zwq}
Z_{w,q}:=\left\{ x\in Z: \sigma^n(x)\in \bigcup_{j=0}^{q-1} \sigma^{-j}[w]\text{ for all }n\ge0\right\}\,,
\end{equation}
or in other words
$$
Z_{w,q}=
\bigcap_{n=0}^\infty \sigma^{-n}\left( \bigcup_{j=0}^{q-1} \sigma^{-j}[w] \right)
\,.
$$
\end{defn}

It turns out that the corresponding $(w,q)$-step subshift $Z_{w,q}$ is of
\emph{finite type} (see Proposition \ref{zwqsft} below).
In order to prove that result, we first require the following:

\begin{lem}\label{zwqpresftlemma}
Suppose $X$ is a shift space,
and $Z\subseteq X$ is a subshift with synchronizing word $w\in\w(Z)$.
Let $q\in\N$.
If $x,y\in Z_{w,q}$ can be written as $x=uvr$, $y=vs$, for $u,v\in \w(Z_{w,q})$, $r,s\in Z_{w,q}$,
where $|v|\ge q+|w|$, then $uvs\in Z_{w,q}$.
\end{lem}
\begin{proof}
Now $y\in Z_{w,q}$, so $y\in \bigcup_{j=0}^{q-1}\sigma^{-j}[w]$,
so there exists $t\in\w(Z_{w,q})$ with $|t|\le q-1$ such that
\begin{equation}\label{y1tw}
y_{[1,|tw|]}=tw\,.
\end{equation}
Since $|v|\ge q+|w|\ge |tw|$,
comparison of (\ref{y1tw}) with the expression $y=vs$ implies that 
\begin{equation}\label{vtwp}
v=twp
\end{equation}
 for some $p\in\w(Z_{w,q})$, so that
\begin{equation}\label{ytwps}
y=twps\,.
\end{equation}
Now $x=uvr$, so (\ref{vtwp}) implies that
\begin{equation}\label{xutwpr}
x=utwpr\,.
\end{equation}
Since $x,y\in Z$, and $w$ is synchronizing for $Z$,
Lemma \ref{uwsinz}, together with (\ref{ytwps}) and (\ref{xutwpr}),
implies that
\begin{equation}\label{zutwps}
z:=utwps \in Z\,.
\end{equation}
In other words, $z=uvs\in Z$. 

It remains to show that $z=uvs\in Z_{w,q}$, in other words to show that
\begin{equation}\label{sigmanz}
\sigma^n(z)\in \bigcup_{j=0}^{q-1} \sigma^{-j}[w]\quad\text{for all }n\ge 0.
\end{equation}
If $n\le |u|$ then, since $x\in Z_{w,q}$, the word $w$ appears in 
$x_{[n+1,n+q+|w|-1]}
=
z_{[n+1,n+q+|w|-1]}$, thus indeed (\ref{sigmanz}) holds.

If $n>|u|$ then, since $\sigma^{n-|u|}(y)\in Z_{w,q}$, then
$\sigma^{n-|u|}(y)\in \bigcup_{j=0}^{q-1} \sigma^{-j}[w]$, so $w$ appears
in
$y_{[n-|u|+1,n-|u|+q+|w|]}
=
z_{[n+1,n+q+|w|]}$, and again  (\ref{sigmanz}) holds, so the proof is complete.
\end{proof}

Recalling (see e.g.~\cite[Thm.~2.1.8]{LM95}) that a shift space $Y$ is an \emph{$m$-step shift of finite type (SFT)}
if every word of length at least $m$ in $Y$ is a synchronizing word, we can now prove the following:

\begin{prop}\label{zwqsft}
Suppose $Z$ is a shift space with synchronizing word $w$.
For all $q\in\N$, the associated $(w,q)$-step subshift $Z_{w,q}$ is a $(q+|w|)$-step SFT.
\end{prop}
\begin{proof}
We must show that if $t,u,v\in\w(Z_{w,q})$ are such that
$tv, vu\in\w(Z_{w,q})$, with $|v|\ge q+|w|$,
then $tvu\in\w(Z_{w,q})$.
There exists $x,y\in Z_{w,q}$ with $x=tvr$ and $y=vus$, for $r,s\in Z_{w,q}$.
By Lemma \ref{zwqpresftlemma}, the fact that $|v|\ge q+|w|$
implies that $z=tvus \in Z_{w,q}$, and therefore $tvu\in\w(Z_{w,q})$, as required.
\end{proof}

The significance of shifts of finite type in the context of the TPO property is fundamental:

\begin{lem}\label{contrerasSFTlemma}
If $Z$ is a SFT, then $(Z,\sigma)$ has the TPO property.
\end{lem}

Since the shift map acting on a one-sided SFT is an open expanding map
(see \cite[Thm.~1]{Pa66}, and cf.~\cite[p.~112]{URM22}), 
Lemma \ref{contrerasSFTlemma} follows from \cite{C02} (cf.~also \cite[p.~2609]{Jen19} for a brief discussion of
the proof strategy).
A rather different approach to some elements of the proof of Lemma \ref{contrerasSFTlemma}, 
following \cite{HLMXZ19} (cf.~also \cite{Boc19}), is that a
Ma\~n\'e lemma (see Remark \ref{sofic_non_mane} for further discussion) holds for Lipschitz functions
$f:Z\to\R$ when $Z$ is a SFT,
meaning that such an $f$ can, without loss of generality, be assumed to satisfy $f\le \beta(f)$.
If $f$ has no periodic maximizing measure then the maximal $\sigma$-invariant subset of $f^{-1}(\beta(f))$
does not contain a periodic orbit, but can be shown to be well approximated by periodic orbits, in a certain precise
sense detailed
in \cite{HLMXZ19}, using tools valid for SFTs developed in \cite{bressaudquas},
so that if $Y$ is such an approximating periodic orbit, then functions of the form
$f_\epsilon:=f-\epsilon d(\cdot,Y)$ approximate $f$ in the Lipschitz topology, and their unique maximizing measure is the
periodic one supported on $Y$.

Given a shift space $X$,
we shall be particularly interested in the $(w,q)$-step subshifts $(X_w)_{w,q}$ corresponding to 
the various $X_w$; for ease of notation we shall 
write $X_{w,q}$ instead of $(X_w)_{w,q}$.
Specifically, we define\footnote{Note that if the shift space $X$ has a synchronizing word $w\in\w(X)$, then $X_w=X$
(cf.~Remark \ref{xwremark}), so the definition of $X_{w,q}$ given by (\ref{zwq}) coincides with that of (\ref{xwq}).}:

\begin{defn}
If $X$ is a shift space, $w\in \w(X)$, and $q\in\N$, define
\begin{equation}\label{xwq}
X_{w,q}:=\left\{ x\in X_{w} : \sigma^n(x)\in \bigcup_{j=0}^{q-1} \sigma^{-j}[w]\text{ for all }n\ge0\right\}\,.
\end{equation}
\end{defn}

\begin{cor}
For a shift space $X$, with $w\in\w(X)$, and $q\in\N$, the subshift $X_{w,q}$ is a $(q+|w|)$-step SFT.
\end{cor}
\begin{proof}
Immediate from Lemma \ref{xwiswsynchronizing}
and
Proposition \ref{zwqsft}.
\end{proof}

We wish to consider the family of SFTs $X_{w,q}$, and possibly adjoin the Markov boundary: 

\begin{notation}
Let $X$ be a shift space.
Define
$$\ZZ_{SFT}:=\{X_{w,q}:w\in\w(X), q\in\N\}$$
and
$$\ZZ_{SFT}^\partial:= \ZZ_{SFT} \cup \{\partial_MX\}\,.$$
\end{notation}

Since the Markov boundary is a closed invariant subset of $X$, we will be interested in the subset of $\lip(X)$
corresponding to functions with a maximizing measure supported there:

\begin{notation} For  a shift space $X$ we shall write
$$\lip_{\partial}(X):=
\lip_{\partial_MX}(X,\sigma)
=
\left\{f\in\lip(X): 
\partial_MX\text{ is $f$-maximizable} \right\},$$
and
$$\lip_{\partial}^\subseteq(X):=
\lip_{\partial_MX}^\subseteq(X,\sigma)
=
\left\{f\in\lip(X): 
\mathcal{M}_{\max}(f)\subseteq \mathcal{M}(\partial_MX,\sigma) \right\}.$$
\end{notation}

Note that $\lip_{\partial}(X)$ is closed in $\lip(X)$, by Lemma \ref{LZclosed},
and is empty if and only if  $\partial_MX$ is, by Lemma \ref{LipZnonemptyiff}.

\begin{lem}\label{l-3}
For any shift space $X$,
$$\text{int}\left(\lip^\subseteq_{\partial}(X)\right) = \text{int}\left(\lip_{\partial}(X)\right)\,.$$
	\end{lem}
\begin{proof}
If $\n=\m(\partial_M(X),\sigma)$
then
 $\lip_\partial(X)=\lip^\n(X)$ and
$\lip_\partial^\subseteq(X)=\lip^\n_\subseteq(X)$,
and the result follows from Lemma \ref{Nclosedconvexsameinteriors},
since $\n$ is closed and convex.
\end{proof}

We are now able to show that the family $\ZZ_{SFT}^\partial$ is maximizable:

\begin{prop}\label{sftmaximizablefamilylip}
Let $X$ be a shift space.
The family $$\ZZ_{SFT}^\partial=\{X_{w,q}:w\in\w(X), q\in\N\} \cup \{\partial_MX\}$$ is maximizable.
\end{prop}
\begin{proof}
If $f\in \lip_\partial(X)$ then, by definition of this set, $\partial_MX$ is maximizable for $f$.
Now let $f\in \lip(X)\setminus \lip_\partial(X)$.
By Theorem \ref{l-6}, 
$f$ has a completely maximizing set $Y$ that is dynamically minimal,
and there exists $w\in\w(X)$ 
with
$Y\subseteq X_w$,
and $Y\cap[w]_X\neq\emptyset$. 
 The dynamical minimality of $(Y,\sigma)$ means that
(see e.g.~\cite[Lem.~1.14]{furstenberg})
there exists $q\in\N$ such that 
$Y\subseteq \bigcup_{j=0}^{q-1} \sigma^{-j}[w] $.
But $Y$ is $\sigma$-invariant, so 
\begin{equation}\label{YXwq}
Y\subseteq X_{w,q}\,.
\end{equation}
Since $Y$ is completely maximizing for $f$, it is in particular maximizable for $f$, so (\ref{YXwq}) means that
$X_{w,q}$ is maximizable for $f$.  
\end{proof}

Having established that
the countable family
$\ZZ_{SFT}^\partial$
is maximizable, we are now able to prove a 
typical periodic optimization
theorem.
The following corresponds to Theorem \ref{sofictpointrotheorem} in Section \ref{introsection}:

\begin{thm}\label{sofictpotheorem}
Every sofic shift space has the TPO property.
\end{thm}
\begin{proof}
If $X$ is sofic then $\partial_M X=\emptyset$, by Lemma \ref{sofic_iff_empty_boundary}.
The family
$$\ZZ_{SFT}=\{X_{w,q}:(w,q)\in \w(X)\times\N\}$$
is countable, and maximizable by Proposition \ref{sftmaximizablefamilylip},
and the SFT $X_{w,q}$ has TPO
for each $(w,q)\in \w(X)\times\N$, by Lemma \ref{contrerasSFTlemma}.
So $\ZZ_{SFT}$ is a countable maximizable family with TPO, and therefore $X$ has TPO, by
Theorem \ref{abstracttpo}.
\end{proof}

\begin{rem}\label{sofic_non_mane}
As noted in Section \ref{introsection},
a novelty of the TPO results presented in this paper is 
that the shift spaces $(X,\sigma)$ need not admit a \emph{Ma\~n\'e lemma}\footnote{The nomenclature for a result of this kind goes back to \cite{B00}, in view of
\cite{Ma92, Ma96},
and is sometimes called the Ma\~n\'e-Conze-Guivarc'h lemma (see \cite{B11, Mo09})
in view of \cite{conzeguivarch};
other important early papers on the Ma\~n\'e lemma include \cite{CLT01, LT03, Sav99}, and
further history and context is given in e.g.~\cite{Boc18, Jen06, Jen19}.},
i.e.~a result guaranteeing that for all $f\in\lip(X)$, there exists $\phi\in\lip(X)$ satisfying
$f+\phi-\phi\circ \sigma \le \beta(f)$.
The existence of such a function $\phi$ was essential for the TPO results of 
\cite{C02, HLMXZ19, LZ24}, facilitating subsequent closing, shadowing and perturbation techniques.
In general shift spaces $X$, the Ma\~n\'e lemma need not hold, even 
in the relatively benign case when $X$ is sofic:
a simple example is given by defining
$X\subseteq\{a,b,c\}^\N$ to be the shift orbit closure of sequences of the form
$$a^{n_1}ba^{n_2}ca^{n_3}ba^{n_4}ca^{n_5}\ldots,$$ where each $n_i\ge 1$.
This $X$ is sofic, being a factor of the shift of finite type $Y\subseteq\{a_1,a_2,a_3,a_4,b,c\}^\N$
whose allowed length-2 words are $a_1a_1$, $a_1b$, $ba_2$, $a_2a_3$, $a_3a_3$, $a_3c$, $ca_4$, $a_4a_1$.
Define the Lipschitz function $f:X\to\R$ to depend on a single coordinate as follows:
$f(x)=0$ if $x_1=a$, $f(x)=1$ if $x_1=b$, $f(x)=-1$ if $x_1=c$.
Now $\int f\, d\mu=0$ for any 
 $\mu\in\m(X,\sigma)$, so $\beta(f)=0$.
We claim that
$
f(x)+\phi(x)-\phi (\sigma x)\le 0
$
does not hold for any Lipschitz (or indeed any continuous) function $\phi$,
for if it did then
defining $z^{(n)}:=a^{n}ba^\infty$ for $n\ge 1$,
and evaluating the inequality at $x=x_i=\sigma^i(z^{(n)})$ for $0\le i\le n$ would yield
$1+ \phi(a^nba^\infty) - \phi(a^\infty) \le 0$,
and letting $n\to\infty$ gives a contradiction, since
 the continuity of $\phi$ means that $\lim_{n\to\infty} \phi(a^nba^\infty) = \phi(a^\infty)$.
\end{rem}

Another consequence of the fact that
$\ZZ_{SFT}^\partial$
is maximizable, which follows from the structural theorem (Theorem \ref{abstracttpovariant}),
is that
all shift spaces have the
\emph{typical periodic or boundary optimization} property:

\begin{thm}\label{general_tpbo_theorem} (Typical Periodic or Boundary Optimization Theorem)

\noindent
If $X$ is a shift space,
then $P(X,\sigma)\cup int( \lip_{\partial}(X))$ is dense in $\lip(X)$.
\end{thm}
\begin{proof}
The countable family $\ZZ_{SFT}^\partial = \ZZ_{SFT}\cup\{\partial_MX\}$ is maximizable,
by Proposition \ref{sftmaximizablefamilylip}.
Now $\ZZ_{SFT}=\{X_{w,q}:(w,q)\in \w(X)\times\N\}$,
and $X_{w,q}$ has TPO 
for each $(w,q)\in \w(X)\times\N$, by Lemma \ref{contrerasSFTlemma}, since $X_{w,q}$ is a SFT.
The result then follows from Theorem \ref{abstracttpovariant}.
\end{proof}

An important consequence of Proposition \ref{sftmaximizablefamilylip} is the following criterion for $X$ to have 
the typical periodic optimization property:

\begin{prop}\label{Markov_boundary_in_Y_TPO}
Let $X$ be any shift space.
Suppose $\partial_M X\subseteq Y \subseteq X$, where $Y$ is a subshift with TPO.
Then $X$ has TPO.
\end{prop}
\begin{proof}
The countable family $\ZZ_{SFT}^\partial=\{X_{w,q}:w\in\w(X), q\in\N\} \cup \{\partial_MX\}$ is maximizable,
by Proposition
 \ref{sftmaximizablefamilylip}.
Since $Y$ is a subshift satisfying $\partial_M X\subseteq Y \subseteq X$, it follows that
$$\ZZ:=\{X_{w,q}:w\in\w(X), q\in\N\} \cup \{ Y\}$$ is also maximizable.
Each $X_{w,q}$ has TPO
by Lemma  \ref{contrerasSFTlemma},
and $Y$ has TPO by assumption.
So $\ZZ$ is a countable maximizable family with TPO.
It follows from Theorem \ref{abstracttpo} 
that $X$ has TPO.
\end{proof}

As a corollary we deduce the following, corresponding to Theorem \ref{partialX_tpo_implies_X_tpo} in Section \ref{introsection}:

\begin{thm}\label{partialX_tpo_implies_X_tpo_repeated}
Let $X$ be any shift space.
If the Markov boundary $\partial_M X$ has TPO, then $X$ has TPO.
\end{thm}
\begin{proof}
Setting $Y=\partial_M X$, the result follows from Proposition \ref{Markov_boundary_in_Y_TPO}.
\end{proof}

\begin{notation}
Let $\s$ denote the class of all shift spaces (over all alphabets),
with the convention (cf.~Notation \ref{shift_space_notation}) that the empty set is a shift space.
Define the  \emph{Markov boundary map}
$$
\partial_M : \s\to\s
$$
in the obvious way, 
i.e.~$\partial_M(X)=\partial_M X$ is the Markov boundary of $X$.
\end{notation}

As a simple consequence of Theorem \ref{partialX_tpo_implies_X_tpo_repeated},
a sufficient condition for TPO is that
 some iterated image under the Markov boundary map has TPO:

\begin{cor}\label{iteration_markov_boundary_tpo}
Let $X$ be any shift space.
If $\partial_M^{n} (X)$ has TPO for some $n\ge0$, then $X$ has TPO.
\end{cor}
\begin{proof}
Apply Theorem \ref{partialX_tpo_implies_X_tpo_repeated} 
to deduce, successively,
that the iterated Markov boundaries $\partial^{n-i}_M X$ have TPO, for $1\le i\le n$. 
\end{proof}

By Theorem \ref{sofictpotheorem}, every sofic shift has TPO,
so Corollary \ref{iteration_markov_boundary_tpo} prompts the following definition:

\begin{defn} \label{presofic_defn}
For $n\in\N_0$,
a shift space $X$ is said to be
\emph{eventually sofic of level $n$}
if $\partial_M^{n} (X)$ is a non-empty sofic shift space.
A shift space is said to be
\emph{eventually sofic} if it is
eventually sofic of level $n$
 for some $n\in\N_0$.
\end{defn}

\begin{rem}
If $X$ is eventually sofic then its level, in the sense of Definition \ref{presofic_defn},
is well-defined, since if $Y$ is a non-empty sofic shift then $\partial_M Y=\emptyset$,
by Lemma \ref{sofic_iff_empty_boundary}.
A shift space is eventually sofic of level 0 if and only if it is a non-empty sofic shift.
\end{rem}

The following corresponds to Theorem \ref{eventuallysofictpointrotheorem} in Section \ref{introsection}:

\begin{thm}\label{pre_sofic_theorem}
Every eventually sofic shift space has TPO.
\end{thm}
\begin{proof}
If $X$ is eventually sofic then  $\partial_M^{n} (X)$ is sofic for some $n\ge0$,
and therefore  $\partial_M^{n} (X)$ has TPO
by Theorem \ref{sofictpotheorem}.
It follows from Corollary \ref{iteration_markov_boundary_tpo} that $X$ has TPO.
\end{proof}

We can now show that certain specific shift spaces are eventually sofic (of level at most one\footnote{Eventually sofic shifts of
arbitrarily  high level will be investigated in Section \ref{TPO_preserving}.}) and hence have TPO.
Given $S\subseteq\N_0$, the corresponding \emph{$S$-gap shift} 
is defined as the set of sequences on the alphabet $\{0,1\}$, such that
the number of 0's between any two consecutive 1's is an integer from $S$
(the well known \emph{even shift}, mentioned in Section \ref{introsection}, is a particular example).
The class of $S$-gap shifts is much studied
(see e.g.~\cite[\S 1.2]{LM95}, and for recent progress see \cite{climenhaga, climenhagathompson, climenhagathompsonyamamoto}).
For our purposes it is convenient to view $S$-gap shifts as special cases of
\emph{$S$-graph} shifts, defined as follows:

\begin{defn}\label{sgraph_defn}
Let $A$ be an alphabet.
Let $G$ be a finite, directed, simple graph, with vertex set $V_G=A$ and edge set $E_G$.
Let $S =(S_a)_{a\in A}$ be an indexed collection of non-empty subsets $S_a\subseteq \N$. 
The corresponding \emph{$S$-graph shift} 
$X_{G,S}$ is defined (cf.~\cite{Di22}) to be the shift orbit closure (in the full shift $A^\N$) of the set
$$
\left\{ a_1^{n_1}a_2^{n_2}a_3^{n_3}\cdots : a_i\in A, \ n_i\in S_{a_i},\ (a_i,a_{i+1})\in E_G \text{ for all }i\in\N\right\}.
$$
\end{defn}

An $S$-graph shift is sofic if and only if the sequence of differences of consecutive terms (written in ascending order) in each $S_a$ is eventually periodic \cite[Prop.~3.2]{Di22},
and is a shift of finite type if and only if each $S_a$ is either finite or cofinite \cite[Prop.~3.1]{Di22}.
Irrespective of whether they are sofic, all $S$-graph shifts have TPO:

\begin{thm}\label{sgraphtpotheorem}
Every $S$-graph shift
is eventually sofic, 
and in particular has  
TPO.
\end{thm}
\begin{proof}
With notation as in Definition \ref{sgraph_defn}, note that
if $(a_1,a_2)\in E_{G}$ with $a_1\neq a_2$, then $a_1 a_2$ is a synchronizing word for $X_{G,S}$,
and therefore not in $\w_\infty(X_{G,S})$.
It follows that $\partial_M X_{G,S} \subseteq \{a^\infty: a\in A\}$.
In particular, $\partial_M X_{G,S}$ is either empty (in which case $X_{G,S}$ is sofic) or a shift of finite type consisting of a finite union of fixed points (in which case $X_{G,S}$ is eventually sofic of level one); in either case, Theorem \ref{pre_sofic_theorem}
implies that $X_{G,S}$ has TPO.
\end{proof}

The following corresponds to Theorem \ref{sgaptpointrotheorem} in Section \ref{introsection}:

\begin{cor}\label{s_gap_cor}
Every $S$-gap shift
is eventually sofic, 
and in particular has  
TPO.
\end{cor}
\begin{proof}
An $S$-gap shift is a special case of an $S$-graph shift, so the result follows from
Theorem \ref{sgraphtpotheorem}. 
\end{proof}

The \emph{context free shift} of \cite[\S 1.2]{LM95} is defined as the set of those sequences
in the full shift $\{a,b,c\}^\N$ that do not have sub-words of the form $ab^mc^na$ for $m\neq n$.
The context free shift is not sofic (see \cite[p.~68]{LM95}), but we can show that it does have TPO:

\begin{thm}\label{context_free_theorem}
The context free shift is eventually sofic, and in particular has TPO.
\end{thm}
\begin{proof}
Let $X \subseteq \{a,b,c\}^\N$ denote the context free shift.
We claim that its Markov boundary $\partial_M X$ is the shift of finite type consisting of those
sequences on the alphabet $\{b,c\}$ which do not contain the word $cb$ as a sub-word.
To prove this, it suffices to show that 
\begin{equation}\label{context_free_boundary}
\w(\partial_M X)=\w_\infty(X)= \left\{ b^n c^m : m,n\in\N_0\right\}.
\end{equation}
Note that both of the words $a$ and $cb$ are synchronizing in $X$, 
from which it follows that 
$\w_\infty(X) \subseteq \{ b^m c^n : m,n\in\N_0 \}$,
so it remains to prove the reverse inclusion.

For any $w\in\w(X)$, define
$M(w):=\min \{m\in\N_0: w c^m a\in \w(X)\}$.

Now suppose $w=b^n c^m$ for some $m,n\in\N_0$, where $m+n\neq 0$.
For any $k\in\N$, the definition of $X$ means that 
$M(a b^{km}w)
=M(a b^{km+n}c^m)
=n+(k-1)m$.
Therefore the follower sets $\{ F_X(ab^{km} w)\}_{k\in\N}$ are pairwise distinct,
so
$|\{F_X(uw):u\in P_X(w)\}|=\infty$,
and hence $w\in\w_\infty(X)$.
But $w$ was an arbitrary non-empty word in $\left\{ b^n c^m : m,n\in\N_0\right\}$, so we have shown that
$\left\{ b^n c^m : m,n\in\N_0\right\} \subseteq \w_\infty(X)$, as required.

Since $\partial_M X$ is a non-empty shift of finite type, $X$ is eventually sofic of level 1, and in particular has TPO, by Theorem \ref{sgraphtpotheorem}. 
\end{proof}

\section{Typical periodic optimization for eventually fragile shifts}\label{specimen_section}

In this section we will investigate the following consequence
of Theorem \ref{general_tpbo_theorem} 
in the case where the Markov boundary of $X$ is non-empty.

\begin{cor}\label{lip_boundary_empty_interior}
Let $X$ be a shift space.
If $\lip_{\partial}(X)$ has empty interior in $\lip(X)$, then $X$ has TPO.
\end{cor}
\begin{proof}
If $ int( \lip_{\partial}(X))=\emptyset$ then 
Theorem \ref{general_tpbo_theorem} implies that
$P(X,\sigma)$ is dense in $\lip(X)$, so $X$ has TPO (cf.~Corollary \ref{equivalentTPO}).
\end{proof}

Note that Corollary \ref{lip_boundary_empty_interior} can also be viewed as a consequence
of Theorem \ref{fragile_boundary_theorem}, which was stated in terms of
the boundary of countable maximizable families for general dynamical systems. 
We develop the notion of fragile boundary (see Section \ref{periodictposection}) as follows:

\begin{defn}\label{fragile_defn}
A non-sofic shift space $X$ is said to be \emph{fragile} if 
 $\lip_{\partial}(X)$ has empty interior in $\lip(X)$.
A shift space $X$ is said to be \emph{eventually fragile} if $\partial_M^n X$ is fragile, for some $n\ge 0$.
\end{defn}

The following corresponds to Theorem \ref{eventually_fragile_tpo_intro} in Section \ref{introsection}:

\begin{thm}\label{eventually_fragile_tpo}
Every eventually fragile shift space has TPO.
\end{thm}
\begin{proof}
Every fragile shift space has TPO, by Corollary \ref{lip_boundary_empty_interior}.
If $X$ is eventually fragile, then $\partial_M^n X$ is fragile for some $n\ge 0$,
and therefore $\partial_M^n X$ has TPO.
Corollary \ref{iteration_markov_boundary_tpo} then implies that $X$ itself has TPO.
\end{proof}

It is not a priori obvious that (eventually) fragile shift spaces even exist, though it turns out that they do;
in Section \ref{TPO_preserving} we will describe methods for their construction.
To investigate this class, we begin with the following definition:

\begin{defn}\label{lipschitz_subordination_principle}
We say that a shift space $X$ enjoys the 
\emph{subordination principle}\footnote{In view of its formulation in terms of $\lip(X)$, a more precise term would be the \emph{Lipschitz subordination principle}. As noted earlier (cf.~Definition \ref{subordinationmaximizingdefn} and Remark \ref{subordinationremark}(b)), the subordination principle was originally articulated by Bousch \cite{B01} for the class of Walters functions.}
if every $f\in\lip(X)$ has a subordination maximizing set.
\end{defn}

\begin{notation}
For a shift space $X$, with subshift $Z\subseteq X$, we shall write
$$
\lip^\subseteq_Z(X)
:=
\lip^\subseteq_Z(X,\sigma)
=
\left\{f\in \lip(X): \m_{\max}(f) \subseteq \m(Z,\sigma)\right\},
$$
and also introduce
$$
\lip^=_Z(X):=\left\{f\in \lip(X): \m_{\max}(f)=\m(Z,\sigma)\right\}.
$$
\end{notation}

\begin{lem}\label{subordination_minimal_equality}
Suppose the shift space $X$ enjoys the subordination principle.
If the subshift $Z\subseteq X$ is dynamically minimal, then $\lip^\subseteq_Z(X)= \lip^=_Z(X)$.
\end{lem}
\begin{proof}
The inclusion $\lip^=_Z(X) \subseteq \lip^\subseteq_Z(X)$ is immediate from the definitions.
To prove the reverse inclusion, suppose $f\in \lip^\subseteq_Z(X)$, and let $Y\subseteq X$ be a subordination maximizing set for $f$, so that 
\begin{equation}\label{subord_lipschitz}
\m_{\max}(f) = \m(Y,\sigma).
\end{equation}
But $f\in \lip^\subseteq_Z(X)$, so $\m_{\max}(f) \subseteq \m(Z,\sigma)$, in other words
$\m(Y,\sigma)\subseteq \m(Z,\sigma)$, and therefore $Y\cap Z \neq \emptyset$.
But $Z$ is dynamically minimal, therefore $Z\subseteq Y$, so $\m(Y,\sigma)=\m(Z,\sigma)$, and
(\ref{subord_lipschitz}) implies that $\m_{\max}(f) = \m(Z,\sigma)$, and hence $f\in\lip^=_Z(X)$, as required.
\end{proof}

\begin{defn}\label{variable_specification_defn}
A shift space $X$ is said to have \emph{specification}
if there exists $N\in\N$ such that for all
$u,w\in \w(X)$, there exists $v\in \w_N(X)$ with $uvw\in\w(X)$.
It has \emph{variable specification}\footnote{Our definitions, in particular the terminology \emph{variable specification},
follow \cite{klo}, wherein various specification-like conditions are discussed. The variable specification condition is actually referred to as
specification in \cite{Ber88}, and
 \emph{almost specification} in \cite{jung}
(though the latter term
 is used by other authors, in particular \cite{thompson, yamamoto}, to mean different things).}
if there exists $N\in\N$ such that for all
$u,w\in \w(X)$, there exists $v\in \bigcup_{n\le N}\w_n(X)$ with $uvw\in\w(X)$.  
\end{defn}

\begin{rem}
Notions of specification originated with Bowen's influential article \cite{bowen}.
All topologically transitive sofic shifts (and hence all topologically transitive SFTs) have variable specification (see \cite{klo}).
A shift space with variable specification is synchronized (i.e.~it is topologically transitive, and has a synchronizing word),
see  \cite[Thm.~1]{Ber88}
and \cite[Lem.~3.1]{jung}. 
\end{rem}

Variable specification implies the subordination principle:

\begin{prop}\label{variable_specification_implies_subordination}
If a shift space has variable specification, then it enjoys the subordination principle.
\end{prop}
\begin{proof}
Let $X$ be a shift space with variable specification, and let $N\in\N$ be as in Definition \ref{variable_specification_defn}. 
To show that $X$ enjoys the subordination principle, let $f\in\lip(X)$,
and without loss of generality assume that 
\begin{equation}\label{wlog_beta_zero}
\beta(f)=0.
\end{equation}
Recalling that, for $\alpha\in(0,1)$, the distance function on $X$ is $d=d_\alpha$ as in Notation \ref{shift_space_notation},
we claim that
\begin{equation}\label{sub_spec_birkhoff_bound}
\sup_{n\ge 1}\sup_{x\in X} S_nf(x)\le \alpha(1-\alpha)^{-1} |f|_{\lip(X)}+N \, \|f\|_{C(X)}.
\end{equation}
Note that (\ref{sub_spec_birkhoff_bound}) implies,
by Lemma \ref{morrislemma}, that there exists a subordination maximizing set for $f$, and hence that $X$ enjoys the subordination principle.

To prove (\ref{sub_spec_birkhoff_bound}), note that if it were false then there 
would exist $x\in X$ and $n\in\mathbb{N}$ with
\begin{equation}\label{sn-lower-111}
S_nf(x)- \alpha(1-\alpha)^{-1} |f|_{\lip(X)} - N\, \|f\|_{C(X)}>0,
\end{equation}
so defining $w\in\w(X)$ by $w:=x_{[1,n]}$
and recursively applying Definition \ref{variable_specification_defn},
there would exist 
$$x':=wv^{(1)}wv^{(2)}wv^{(3)}\cdots\in X,$$
where the words $v^{(k)}\in \w(X)$ satisfy $|v^{(k)}|\le N$ for all $k\in\N$.
However, we claim that such an $x'$ satisfies
\begin{equation}\label{limsup_x_prime}
\limsup_{M\to\infty} \frac{1}{M} S_Mf(x')
\ge 
\frac{1}{n+N}\left( S_nf(x)- \alpha(1-\alpha)^{-1} |f|_{\lip(X)} - N\, \|f\|_{C(X)} \right),
\end{equation}
so that combining
 (\ref{limsup_x_prime}) with  (\ref{sn-lower-111}),
and the fact that
$\beta(f)\ge \limsup_{M\to\infty} \frac{1}{M} S_Mf(x')$
(see e.g.~\cite[Prop.~2.2]{Jen06})
 would yield
$
\beta(f)>0
$, thereby contradicting (\ref{wlog_beta_zero}).

So to prove the result, it suffices to establish (\ref{limsup_x_prime}).
For this, first define
\begin{equation}\label{i_k_defn}
i_k := \sum_{j=1}^k \left(n + |v^{(j)}| \right) \ \text{ for }k\in\N,
\end{equation}
and note that
\begin{equation}\label{close_returns}
d\left(x, \sigma^{i_k}(x')\right) \le \alpha^n \ \text{ for all } k\in\N,
\end{equation}
since $x$ and $\sigma^{i_k}(x')$ share the length-$n$ prefix $w$.

We will estimate the limit supremum in (\ref{limsup_x_prime}) by considering 
Birkhoff sums along the subsequence $(i_m)_{m=1}^\infty$,
and begin by expressing
\begin{equation}\label{i_m_expression}
S_{i_m}f(x') = \sum_{k=1}^m S_{n+|v^{(k)}|}f \left(\sigma^{i_k}(x')\right).
\end{equation}
Note that the summand in (\ref{i_m_expression}) can be bounded by
\begin{equation}\label{k_summand}
 S_{n+|v^{(k)}|}f \left(\sigma^{i_k}(x')\right)
\ge
S_{n}f \left(\sigma^{i_k}(x')\right) -  |v^{(k)}| \,  \|f\|_{C(X)}
\ge
S_{n}f \left(\sigma^{i_k}(x')\right) -  N \,  \|f\|_{C(X)} .
\end{equation}
To estimate the righthand side of (\ref{k_summand}),
the fact that $f$ is Lipschitz gives 
\begin{equation}\label{x_and_x_prime_lip_sum}
S_nf(x)-S_nf\left(\sigma^{i_k}(x')\right)
\le |S_nf|_{\lip(X)} \,  d\left(x, \sigma^{i_k}(x')\right),
\end{equation}
but
$|f\circ \sigma^i |_{\lip(X)}\le \alpha^{-i}|f|_{\lip(X)}$ for $0\le i\le n-1$, therefore
\begin{equation}\label{sum-lip-111}
|S_nf|_{\lip(X)}\le \sum_{i=0}^{n-1} \alpha^{-i} |f|_{\lip(X)}\le \alpha^{1-n}(1-\alpha)^{-1}|f|_{\lip(X)},
\end{equation}
so combining (\ref{close_returns}), (\ref{x_and_x_prime_lip_sum}) 
and (\ref{sum-lip-111}) gives
\begin{equation}\label{x_and_x_prime_lip_sum_alpha}
S_nf(x)-S_nf\left(\sigma^{i_k}(x')\right)
\le \alpha(1-\alpha)^{-1} |f|_{\lip(X)}.
\end{equation}
Combining (\ref{k_summand}) and (\ref{x_and_x_prime_lip_sum_alpha}) gives
\begin{equation}\label{x_prime_summand_bound_independent_of_k}
S_{n+|v^{(k)}|}f \left(\sigma^{i_k}(x')\right) \ge S_nf(x) - \alpha(1-\alpha)^{-1} |f|_{\lip(X)}  - N\, \|f\|_{C(X)},
\end{equation}
so using (\ref{x_prime_summand_bound_independent_of_k}) in (\ref{i_m_expression}) gives
\begin{equation}\label{i_m_birkhoff_sum_bound}
S_{i_m}f(x')
\ge
m \left( S_nf(x) - \alpha(1-\alpha)^{-1} |f|_{\lip(X)}   - N\, \|f\|_{C(X)} \right).
\end{equation}
Noting from (\ref{i_k_defn}) that $i_m \le m(n+N)$, the bound (\ref{i_m_birkhoff_sum_bound}) yields
\begin{equation}\label{final_limsup_bound}
\limsup_{m\to\infty} \frac{1}{i_m} S_{i_m}f(x')
\ge \frac{1}{n+N} \left( S_nf(x) - \alpha(1-\alpha)^{-1} |f|_{\lip(X)}  - N\, \|f\|_{C(X)} \right),
\end{equation}
and since $\limsup_{M\to\infty} \frac{1}{M} S_Mf(x') \ge \limsup_{m\to\infty} \frac{1}{i_m} S_{i_m}f(x')$,
we see that (\ref{final_limsup_bound}) implies the required estimate  (\ref{limsup_x_prime}).
\end{proof}

\begin{cor}\label{var_spec_minimal_equality}
Suppose the shift space $X$ has variable specification.
If the subshift $Z\subseteq X$ is dynamically minimal, then $\lip^\subseteq_Z(X)= \lip^=_Z(X)$.
\end{cor}
\begin{proof}
Immediate from
Lemma \ref{subordination_minimal_equality} and Proposition \ref{variable_specification_implies_subordination}.
\end{proof}

As we shall soon see, the property of being dynamically minimal but not uniquely ergodic
(i.e.~supporting more than one shift-invariant probability measure)
is a useful mechanism for creating fragility in shift spaces. First we require:

\begin{cor}\label{var_spec_Z_minimal_nue}
If a shift space $X$ has variable specification,
and a subshift $Z\subseteq X$ is dynamically minimal but not uniquely ergodic,
 then $\lip^\subseteq_Z(X)$  
has empty interior in $\lip(X)$.
\end{cor}
\begin{proof}
We must show that every $f\in\lip^\subseteq_Z(X)$ can be approximated in $\lip(X)$
by functions that do not belong to $\lip^\subseteq_Z(X)$. 
Now $\lip^\subseteq_Z(X) = \lip^=_Z(X)$ 
by Corollary \ref{var_spec_minimal_equality},
and $Z$ is not uniquely ergodic, so $\m(Z,\sigma)$ is not a singleton.
So if $f \in \lip^\subseteq_Z(X) = \lip^=_Z(X)$ then $f$ does not have a unique maximizing measure.
But the set of Lipschitz functions with a unique maximizing measure is a residual subset of $\lip(X)$
(see e.g.~\cite[Thm.~2.4]{Jen19}),
and in particular dense\footnote{An appeal to the literature on typical uniqueness of maximizing measures is convenient, though a direct proof is also 
straightforward:
if $f\in \lip^\subseteq_Z(X) = \lip^=_Z(X)$ then there exist two distinct measures $\mu,\nu\in \m(Z,\sigma)$,
and $g\in\lip(X)$, with $\int g\, d\mu > \int g\, d\nu$, so that $f+\epsilon g\notin \lip^=_Z(X) = \lip^\subseteq_Z(X)$ for all $\epsilon>0$,
yet $f+\epsilon g\to f$ in $\lip(X)$ as $\epsilon\to0$.}, since $\lip(X)$ is a Baire space;
so $f$ can indeed be  approximated in $\lip(X)$
by functions that do not belong to $\lip^\subseteq_Z(X)$, as required.
\end{proof}

We are particularly interested in shift spaces $X$ where the hypotheses of 
Corollary \ref{var_spec_Z_minimal_nue} are satisfied for the 
Markov boundary subshift $Z=\partial_M X$, prompting the following definition:

\begin{defn}\label{specimen_defn}
Define a \emph{specimen}\footnote{This is a portmanteau term, its prefix reflecting the (variable) specification condition, and final syllable an acronymic description of the Markov boundary being \emph{m}inimal with \emph{e}rgodic measures 
\emph{n}on-unique.}
  shift space to be one with variable specification, whose Markov boundary is dynamically minimal but not uniquely ergodic.

\noindent
A shift space $X$ is said to be \emph{eventually specimen} if there exists $n\in\N_0$ such that the iterated Markov boundary $\partial_M^n X$ is a specimen shift space.
\end{defn}

First we show that specimen shift spaces  are fragile, and hence have the typical periodic optimization property:

\begin{thm}\label{specimen_implies_TPO}
Every specimen shift space is fragile, and hence has TPO.
\end{thm}
\begin{proof}
Let $X$ be a specimen shift space, and let $Z=\partial_M X$.
By Corollary \ref{var_spec_Z_minimal_nue}, it follows that
$$\lip^\subseteq_Z(X) = \lip^\subseteq_{\partial_M X}(X) = \lip^\subseteq_\partial(X)$$ 
has empty interior in $\lip(X)$.
Now $\lip^\subseteq_\partial(X)$ and $\lip_\partial(X)$ have identical interior in
$\lip(X)$, by Lemma \ref{l-3}, so  $\lip_\partial(X)$ has empty interior in $\lip(X)$.
It follows that $X$ is fragile (note that $X$ is clearly non-sofic, since its Markov boundary is non-empty). 
It follows from Corollary \ref{lip_boundary_empty_interior}
(or Theorem \ref{eventually_fragile_tpo}) that $X$ has TPO.
\end{proof}

We readily deduce that eventually specimen shift spaces are eventually fragile, and therefore have the TPO
property:

\begin{thm}\label{eventually_specimen_implies_TPO}
Every eventually specimen shift space is eventually fragile, and has TPO.
\end{thm}
\begin{proof}
If $X$ is eventually specimen
then $\partial_M^n X$ is specimen for some $n\in\N_0$,
so  $\partial_M^n X$ is fragile (hence $X$ is eventually fragile) and has TPO by Theorem \ref{specimen_implies_TPO}.
It follows that $X$ has TPO, by
Corollary \ref{iteration_markov_boundary_tpo}.
\end{proof}

The results in this section have been somewhat abstract, as so far we have not given specific examples of
eventually fragile (or specimen) shift spaces; this lack of concreteness will be rectified in the following 
Section \ref{TPO_preserving}.

\section{TPO-preserving operations on shift spaces}\label{TPO_preserving}

In Section \ref{SFTmaximizablefamily} we showed that if $\partial_M X$ has TPO, or more generally some iterated Markov boundary $\partial_M^n X$ has TPO, then $X$ itself must have TPO (see Theorem \ref{partialX_tpo_implies_X_tpo_repeated}
and Corollary \ref{iteration_markov_boundary_tpo}).
Having shown that sofic shifts have TPO (Theorem \ref{sofictpotheorem}), we deduced that certain non-sofic shift spaces of interest had TPO, by showing that their Markov boundary is a shift of finite type, and therefore sofic
(see Theorem \ref{sgraphtpotheorem}, Corollary \ref{s_gap_cor}, and Theorem \ref{context_free_theorem}); in particular these shifts are eventually sofic of level $n=1$.
These examples prompt two questions: Do eventually sofic shift spaces of level $n\ge 2$ exist?  And are there shift spaces whose Markov boundary is strictly sofic (i.e.~sofic but not of finite type)?

In this section we answer both of these questions in the affirmative, by introducing a method for realising any shift space $Y$ as the Markov boundary of a shift space $X$ (where $X$ has a synchronizing word); in fact, as we shall see, for a given $Y$ there are many such $X$ with $\partial_M X=Y$.
The method will also lead to constructions of shift spaces satisfying the conditions 
introduced in Section \ref{specimen_section}, i.e.~those which are (eventually) specimen, and hence (eventually) fragile, and therefore have TPO despite not being eventually sofic.
In fact the approach developed here leads to various \emph{TPO-preserving operations} on shift spaces,
in the sense that if a shift space $Y$ is known to have TPO, then its image under a suitable operation will also have TPO;
as such, the methods give a procedure for constructing new TPO spaces from old ones.

The method is not, however, restricted to producing new TPO systems from old ones: it also allows for the construction of shift spaces $X$ with the TPO property, from spaces $Y$ (satisfying $\partial_M X=Y$) that do \emph{not} have this property. 
On the other hand the method
does not always produce spaces with TPO; indeed it 
will be used, in Section \ref{magicmorsesection}, to construct a shift space without TPO, despite
its periodic measures being dense in the space of all invariant measures.

The starting point is the following definition:

\begin{defn}\label{interspersion_defn}
For an alphabet $A$, let $A':= A \sqcup \{\b\}$, i.e.~the disjoint union of $A$ with a singleton set whose unique element is denoted by $\b$.
If $Y$ is a shift space on the alphabet $A$,
and $\v$ is a subset of the language $\w(Y)$,
define the \emph{interspersion} (or \emph{$\v$-interspersion of $Y$}), denoted by
$Y_\v$, to be the smallest shift space,
with alphabet $A'$, that contains
all sequences of the form
\begin{equation}\label{interspersed_sequence}
\b v^{(1)}\b v^{(2)}\b v^{(3)}\ldots
\end{equation}
where $v^{(i)}\in \v$ for all $i\in\N$.
That is, $Y_\v$ is the shift orbit closure of the union of all sequences of the form (\ref{interspersed_sequence}).

\noindent The symbol $\b$, that belongs to the alphabet for $Y_\v$ but not the alphabet for $Y$, is referred to as the \emph{magic symbol} for $Y_\v$.
\end{defn}

A notable property of any interspersion $Y_\v$ of $Y$ is that the Markov boundary
$\partial_M Y_\v$ is contained in $Y$:

\begin{lem}\label{S(Z)properties} If $Y$ is any shift space, and $\v\subseteq\w(Y)$,
then
\begin{itemize}
\item[(a)]
The magic symbol is a synchronizing word for $Y_\v$.
\item[(b)]
$\partial_M( Y_\v )\subseteq Y \cap  Y_\v$.
\end{itemize}
\end{lem}
\begin{proof}
(a) 
Let $\b$ denote the magic symbol in $Y_{\v}$.
The elements of $\w(Y_\v)$ are precisely those of the form
\begin{equation}\label{yvform}
w\b w^{(1)}\b w^{(2)}\b\cdots \b w^{(l)}\b w'
\end{equation}
where $w$ is a suffix of some word in $\v$, and $w'$ is the prefix of some word in $\v$,
and $w^{(1)},\ldots, w^{(l)}\in\v$ for some $l\in\N_0$.

So if $u\b,\b v\in\w(Y_\v)$ then 
$u=w\b u^{(1)}\b u^{(2)}\b\cdots \b u^{(m)}$ for some suffix $w$ of a word in $\v$,
and $u^{(1)},\ldots, u^{(m)}\in\v$ for some $m\in\N_0$,
and $v=  v^{(1)}\b v^{(2)}\b\cdots \b v^{(n)}\b w'$
for some prefix
$w'$ of a word in $\v$,
and $v^{(1)},\ldots, v^{(n)}\in\v$ for some $n\in\N_0$.
So
$$
u\b v = w\b u^{(1)}\b u^{(2)}\b\cdots \b u^{(m)} \b  v^{(1)}\b v^{(2)}\b\cdots \b v^{(n)}\b w',
$$
which is of the form (\ref{yvform}), so
$u\b v\in\w(Y_\v)$, and therefore $\b$ is a synchronizing word in $Y_\v$.

(b) The inclusion $\partial_M( Y_\v )\subseteq  Y_\v$ follows from the definition of the Markov boundary, so it remains
to show that  $\partial_M( Y_\v )\subseteq  Y$.
Suppose $x\in Y_\v \setminus Y$. Then $\b\subset x$.
But $\b$ is a synchronizing word in $Y_\v$, so $\b\in \w_{<\infty}(Y_\v)$.
Therefore $x\notin \partial_M(Y_\v)$.
So we have shown that $Y_\v\setminus Y \subseteq Y_\v \setminus \partial_M Y_\v$.
Since $\partial_M Y_\v\subseteq Y_\v$, it follows that $\partial_M Y_\v \subseteq Y$, as required.
\end{proof}

In many situations, it will be useful to know that the shift space $Y$ is itself a subset of 
the interspersion $Y_\v$; this motivates the following definition.

\begin{defn}
Let $Y$ be a shift space.
A subset $\v\subseteq\w(Y)$ is said to be \emph{$Y$-descriptive}
if for all $y\in Y$, there exists a sequence $(s_j)_{j=1}^\infty$ with $s_j\to\infty$
as $j\to\infty$, such that for each $j\in\N$, the prefix $y_{[1,s_j]}$ is the suffix of some word in $\v$.

\noindent
(Equivalently, $\v$ is $Y$-descriptive if $\w(Y)$ is the suffix-closure of $\v$, i.e.~$Y$ is the smallest
shift space whose language contains every suffix of a word in $\v$).
\end{defn}

\begin{lem}\label{V_lengths_infinity_Y_contained}
Let $Y$ be any shift space. If $\v\subseteq\w(Y)$
is $Y$-descriptive,
then $Y \subseteq Y_\v$.
\end{lem}
\begin{proof}
Suppose $y\in Y$.
There exists a sequence $s_j\to\infty$
such that
$$y_{[1,s_j]}=u^{(j)}_{\left[|u^{(j)}|-s_j+1,|u^{(j)}|\right]}$$ for some $u^{(j)}\in\v$ for all $j\in\N$.
So $y_{[1,s_j]} \b$ is the prefix of some $y^{(j)}\in Y_\v$.
So $y^{(j)}\to y$ as $j\to\infty$, and $Y_\v$ is closed by definition, so $y\in Y_\v$.
Since $y$ was an arbitrary element of $Y$, it follows that $Y \subseteq Y_\v$, as required.
\end{proof}

The double inclusion $\partial_M Y_\v\subseteq Y\subseteq Y_\v$ established in Lemmas
 \ref{S(Z)properties}
and \ref{V_lengths_infinity_Y_contained}
means that
we are now able to derive 
another important consequence of Proposition \ref{Markov_boundary_in_Y_TPO},  
that if $\v$ is $Y$-descriptive, then the corresponding interspersion is a TPO-preserving operation:

\begin{thm}\label{TPO_descriptive_TPO}
Let $Y$ be any shift space with TPO. Suppose $\v \subseteq\w(Y)$ is $Y$-descriptive.
Then the interspersion $Y_\v$ also has TPO.
\end{thm}
\begin{proof}
Since $\v \subseteq\w(Y)$,
Lemma \ref{S(Z)properties}(b)
implies that
$\partial_M( Y_\v )\subseteq Y$.
Since moreover $\v$  is $Y$-descriptive,
Lemma \ref{V_lengths_infinity_Y_contained} implies that $Y\subseteq Y_\v$.
So $\partial_M( Y_\v )\subseteq Y\subseteq Y_\v$, and $Y$ has TPO, so Proposition \ref{Markov_boundary_in_Y_TPO}
implies that $Y_\v$ has TPO.
\end{proof}

The following strengthening of Lemma \ref{S(Z)properties}(b) implies, in particular, that all shift spaces without a synchronizing word are the Markov boundary
of some shift space with a synchronizing word; the sense in which interspersion is a (right) inverse to the operator
$\partial_M$ will be explored further below, notably in  Theorem \ref{prescribed_markov_boundary}.

\begin{prop}\label{partial_Y_v_equals_Y}
Suppose $Y$ is a shift space that  does not have a synchronizing word.
 If $\v\subseteq\w(Y)$
is $Y$-descriptive,
then $\partial_M( Y_\v ) = Y$.
\end{prop}
\begin{proof}
We know that $\partial_M( Y_\v )\subseteq Y$ by Lemma \ref{S(Z)properties}(b), so it remains to establish that
$Y\subseteq \partial_M( Y_\v )$.
Suppose $x\in Y$.
Since $Y$ does not have a synchronizing word, then $\partial_M Y=Y$ by Thomsen 
\cite[Thm.~3.1]{Tho06}.\footnote{Since the statement in \cite[Thm.~3.1]{Tho06} includes a 
topological transitivity assumption on $Y$,
we give a short proof here that, even without this assumption, if $Y\setminus \partial_M Y\neq\emptyset$ then $Y$ must have a synchronizing word. If $y\in Y\setminus \partial_M Y$ then there exists $w\subset y$ such that
$\{F_Y(uw):u\in P_Y(w)\}$ has finite cardinality $c(w)\ge 1$. If $c(w)=1$ then $w$ is synchronizing,
while if $c(w)>1$ then there exists $w^{(1)}=u^{(1)}w\in\w(Y)$ such that $F_Y(u^{(1)}w)\neq F_Y(w)$, so $c(w^{(1)})\le c(w)-1$. Proceeding inductively, there exists $n\le c(w)-1$, and words
$w^{(i)}=u^{(i)}w^{(i-1)}\in\w(Y)$ with $1 \le c(w^{(i+1)})\le c(w^{(i)})-1$ for $1\le i\le n-1$,  and $c(w^{(n)})=1$. 
In particular, $w^{(n)}$ is a
synchronizing word in $Y$, as required.}
So $x\in \partial_M Y$.
By definition of $\partial_M Y$, this means that if $w\in\w(Y)$
with $w\subset x$, then
$w\in\w_\infty(Y)$.
By definition of $\w_\infty(Y)$, this means that
\begin{equation}\label{Y_follower_infinite}
|\{F_Y(uw): u\in P_Y(w)\}|=\infty.
\end{equation}

Now  $\v$
is $Y$-descriptive,
so
\begin{equation}\label{YsubsetYv}
Y \subseteq Y_\v
\end{equation}
by Lemma \ref{V_lengths_infinity_Y_contained}.
From (\ref{YsubsetYv}) we deduce that $w\in\w(Y_\v)$, and that
\begin{equation}\label{predecessor_Y_v}
P_{Y_\v}(w) \supseteq P_Y(w).
\end{equation}
Moreover, for each $u\in P_Y(w)$,
\begin{equation}\label{follower_Y_v}
F_{Y_\v}(uw) \supseteq F_Y(uw)  \text{ and } F_{Y_\v}(uw)\cap \w(Y)=F_Y(uw).
\end{equation}
Thus, by \eqref{follower_Y_v}, we know that 
\begin{equation}\label{follower_Y_v-1}
F_{Y_\v}(u_1w)\neq F_{Y_\v}(u_2w)
\end{equation}
whenever $u_1,u_2\in P_Y(w)$ and $F_Y(u_1w)\neq F_Y(u_2w)$.
It follows from (\ref{Y_follower_infinite}), (\ref{predecessor_Y_v}) and (\ref{follower_Y_v-1}) that
$$
|\{F_{Y_\v}(uw): u\in P_{Y_\v}(w)\}|\ge |\{F_{Y_\v}(uw): u\in P_{Y}(w)\}|\ge |\{F_{Y}(uw): u\in P_{Y}(w)\}|=\infty,
$$
and hence that $w\in \w_\infty(Y_\v)$.

Since $w$ was an arbitrary word appearing in $x$, we deduce that $x\in \partial_M Y_\v$.

But $x$ was an arbitrary element in $Y$, so we deduce that $Y\subseteq \partial_M Y_\v$, as required.
\end{proof}

The following consequence of Proposition \ref{partial_Y_v_equals_Y} will be used in Section \ref{magicmorsesection}, when a particular 
shift space, defined as a
$\v$-interspersion,
is shown to not have TPO, despite its periodic measures being dense in the space of all invariant measures.

\begin{cor}\label{minimal_partial_Y_v_equals_Y}
Suppose $Y$ is a shift space that  is dynamically minimal but non-periodic.
If $\v\subseteq\w(Y)$
is $Y$-descriptive,
then  $\partial_M( Y_\v ) = Y$.
\end{cor}
\begin{proof}
Immediate from Proposition \ref{partial_Y_v_equals_Y}, since a dynamically minimal non-periodic shift space does not have a synchronizing word (cf.~Lemma \ref{synchronizingwordperiodicpoint}).
\end{proof}

For the remainder of this section, we shall be
concerned with the following restricted kind of interspersion, that will be sufficient
(see Theorem \ref{prescribed_markov_boundary})
 for solving the problem
of constructing shift spaces with a given Markov boundary.

\begin{defn}
Let $Y$ be a shift space.
Given a subset $S\subseteq\N_0$, 
and (cf.~(\ref{S_length_words})) $\w_S(Y)=\bigcup_{n\in S} \w_n(Y)$,
 define the corresponding \emph{length-constrained interspersion}
$\i_S(Y)$ to be the $\w_S(Y)$-interspersion of $Y$,
\begin{equation}
\i_S(Y) := Y_{\w_S(Y)}.
\end{equation}
This defines a self-map
$$\i_S:\s\to\s$$
on the class $\s$ of all shift spaces,
referred to as the \emph{$S$-interspersion map}.

\noindent In the special case that $S=\N_0$, the map $\i_{\N_0}$ will be denoted simply by $\i$, and referred to as \emph{unconstrained interspersion}.
\end{defn}

\begin{rem}
For $S\subseteq\N_0$, the shift space $\i_S(Y)$ is the closure of the set of shift-orbits of all sequences of the form
(\ref{interspersed_sequence}), where the length of each word $v^{(i)}$ must belong to $S$.
In particular, when $Y$ is the singleton shift space $Y=\{a^\infty\}$ then $\i_S(Y)$ is precisely the
$S$-gap shift (cf.~Section \ref{SFTmaximizablefamily}) corresponding to the set $S$,
i.e.~the shift orbit closure of all sequences of the form $\b a^{i_1}\b a^{i_2}\b a^{i_3}\ldots$ where each $i_j\in S$.

More generally, if $(Y,\sigma)$ is topologically transitive, and $y\in Y$ has a dense orbit, then $\i_S(Y)$
is the \emph{$(S,y)$-gap shift} in the sense of \cite{RS25}. 
\end{rem}

First we consider the case when $S$ is infinite:

\begin{lem}\label{S_infinite_Y_descriptive}
Let $Y$ be a shift space, and suppose $S\subseteq\N_0$.
If $S$ is infinite then
the collection $\w_S(Y)=\bigcup_{n\in S} \w_n(Y)$
is $Y$-descriptive.
\end{lem}
\begin{proof}
Suppose $S$ is infinite. Let $S=\{s_1,s_2,\ldots\}$ be enumerated such that $s_j < s_{j+1}$ for all $j\in\N$.
Then for all $y\in Y$, the sequence $s_j\to\infty$
and is such that $y_{[1,s_j]}\in\w_S(Y)$ for all $j\in\N$.
So $\w_S(Y)$ is $Y$-descriptive.
\end{proof}

Length-constrained interspersion using \emph{any} infinite $S\subseteq\N_0$
therefore gives a right inverse to the Markov boundary operator,
when restricted to those
shift spaces without  a synchronizing word:

\begin{cor}\label{partial_Y_v_equals_Y_cor}
Suppose $Y$ is a shift space that  does not have a synchronizing word.
 If $S\subseteq\N_0$ is infinite,
then $\partial_M( \i_S(Y) ) = Y$.
\end{cor}
\begin{proof}
Since $S$ is infinite,
Lemma \ref{S_infinite_Y_descriptive} implies that $\w_S(Y)$ is $Y$-descriptive.
Since in addition $Y$ does not have a synchronizing word,
Proposition \ref{partial_Y_v_equals_Y} implies that $\partial_M( Y_{\w_S(Y)})=Y$,
in other words
 $\partial_M( \i_S(Y) ) = Y$.
\end{proof}

In particular, the conclusion of Corollary \ref{partial_Y_v_equals_Y_cor} holds
when $Y$ is dynamically minimal but non-periodic, a result that will be used 
to exhibit examples of specimen (and therefore fragile) shift spaces (see Corollary \ref{Z_min_nue_interspersion_specimen}):

\begin{cor}\label{boundary_S_equals_Y}
Suppose $Y$ is a shift space that  is dynamically minimal but non-periodic,
and $S\subseteq\N_0$ is infinite.
Then $\partial_M( \i_S(Y))=Y$.
\end{cor}
\begin{proof}
Immediate from  Corollary \ref{partial_Y_v_equals_Y_cor}, since a dynamically minimal non-periodic shift space does not have a synchronizing word (cf.~Lemma \ref{synchronizingwordperiodicpoint}).
\end{proof}

The following is a simple consequence of  Lemmas \ref{V_lengths_infinity_Y_contained} and \ref{S_infinite_Y_descriptive}:

\begin{cor}\label{length_constrained_S_infinite_finite}
Let $Y$ be a shift space.
If $S\subseteq\N_0$ is infinite then $Y \subseteq \i_S(Y)$.
\end{cor}
\begin{proof}
Suppose $S$ is infinite.
By Lemma \ref{S_infinite_Y_descriptive},
 $\w_S(Y)$ is $Y$-descriptive, and so
 Lemma \ref{V_lengths_infinity_Y_contained} implies that $Y \subseteq Y_{\w_S(Y)}= \i_S(Y)$.
\end{proof}

If  $S$ is finite then
the relation between
$Y$ and $\i_S(Y)$
is completely different to that of Corollary \ref{length_constrained_S_infinite_finite}:

\begin{lem}\label{S_finite_disjoint}
Let $Y$ be a shift space.
If $S\subseteq\N_0$ is finite then $Y\cap\i_S(Y)=\emptyset$.
\end{lem}
\begin{proof}
Suppose $S$ is finite.
All words in $\w(\i_S(Y))$ of length strictly greater than $\max S$ contain the magic symbol $\b$, whereas no words in $\w(Y)$ contain $\b$.
So if $y\in Y$, and $w\subset y$ with $|w|>\max S$, then $w\notin \w(\i_S(Y))$, so $y\notin \i_S(Y)$.
Therefore $Y\cap\i_S(Y)=\emptyset$.
\end{proof}

\begin{cor}\label{S_finite_sofic}
Let $Y$ be a shift space.
If $S\subseteq\N_0$ is finite then $\i_S(Y)$ is sofic.
\end{cor}
\begin{proof}
A shift space is sofic if and only if its Markov boundary is empty, by
Lemma \ref{sofic_iff_empty_boundary}, so it suffices to show that $\partial_M  \i_S(Y) = \emptyset$.
By Lemma \ref{S(Z)properties}(b),
$$\partial_M  \i_S(Y)
= \partial_M Y_{\w_S(Y)}
\subseteq Y \cap Y_{\w_S(Y)}
= Y \cap \i_S(Y),
$$
but $Y\cap\i_S(Y)=\emptyset$
by Lemma \ref{S_finite_disjoint},
therefore $\partial_M  \i_S(Y) =\emptyset$, as required.
\end{proof}

\begin{cor}
If $y$ is any sequence on some alphabet $A$, and $S\subseteq\N_0$ is finite, 
then the corresponding $(S,y)$-gap subshift of \cite{RS25} is sofic.
\end{cor}
\begin{proof}
If $Y$ denotes the shift orbit closure of $y$, then the $(S,y)$-gap subshift is $\i_S(Y)$,
and the result follows from Corollary \ref{S_finite_sofic}.
\end{proof}

We can now show that for \emph{any} subset $S\subseteq \N_0$, the corresponding length-constrained interspersion is a TPO-preserving operator:

\begin{thm}\label{TPO_descriptive_TPO_S_infinite}
If $Y$ is any shift space with TPO, and $S\subseteq\N_0$,
then $\i_S(Y)$ has TPO.
\end{thm}
\begin{proof}
If $S$ is finite then $\i_S(Y)$ is sofic, by Corollary \ref{S_finite_sofic}, so $\i_S(Y)$ has TPO.
If $S$ is infinite, then $\w_S(Y)$ is $Y$-descriptive,
by Lemma \ref{S_infinite_Y_descriptive},
so $\i_S(Y)=Y_{\w_S(Y)}$
has TPO by Theorem \ref{TPO_descriptive_TPO}, as required.
\end{proof}

We now focus on the problem of realising all shift spaces $Y$ as Markov boundaries
of shift spaces $X$ (where $X$ has a synchronizing word).
Note that by
Proposition \ref{partial_Y_v_equals_Y}
(cf.~also Corollaries \ref{minimal_partial_Y_v_equals_Y} and \ref{boundary_S_equals_Y}),
this problem is already solved in the case where $Y$ does not have a synchronizing word.
In the case of general $Y$, the solution to the problem will use length-constrained interspersion, though with the following stronger assumption on the set $S$:

\begin{defn}
An infinite subset $S\subseteq\N_0$,
with its distinct elements
enumerated as $s_1<s_2<\ldots$,
will be called \emph{superlinear} if
\begin{equation}\label{differences}
s_{n+1} - s_n < s_{n+2} - s_{n+1}\quad\text{for all }n\in\N.
\end{equation}
\end{defn}

\begin{thm}\label{prescribed_markov_boundary}
If $Y$ is a shift space, and $S\subseteq \N_0$ is superlinear, then 
the $S$-interspersion $\i_S(Y)$ is a shift space with
Markov boundary equal to $Y$; that is, $\partial_M ( \i_S(Y)) = Y$.
\end{thm}
\begin{proof}
Let $X:=\i_S(Y)$.
Note that $Y\subseteq X$, and that $Y$ is precisely the set of sequences in $X$ that do not contain the symbol $\b$.
By Lemma \ref{S(Z)properties}(a), 
$\b$ is a synchronizing word for $X$,
so any words containing it do not belong to $\w_\infty(X)$,
thus
$
\partial_M X \subseteq Y
$.
It remains to show that
$Y \subseteq \partial_M X$,
in other words that 
\begin{equation}\label{reverse inclusion2_later}
\w(Y) \subseteq \w_\infty(X).
\end{equation}
For this, let $w$ be an arbitrary word in $\w(Y)$.
There exists $N\in\N$ such that if
 $n>N$ then there exists $v^{(n)}\in\w(Y)$ such that $v^{(n)} w\in\w(Y)$ and
\begin{equation*}\label{length choice}
| v^{(n)}w| = s_n +1.
\end{equation*}
Now $\b v^{(n)}w\in\w(X)$, and we wish to consider the follower set $F_{X}(\b v^{(n)}w)$.
We claim that the length constraint 
in the definition of $X$ means that 
$F_{X}( \b v^{(n)}w)$ contains a word $u\in\w_{s_{n+1}-s_n}(X)$ 
such that 
$$u_{s_{n+1}-s_n}=\b\quad\text{and}\quad
u_i\neq \b\text{ for }1\le i\le s_{n+1}-s_n-1.$$
We now prove this claim. 
Since $v^{(n)}w\in \w_{s_n+1}(Y)$, it must appear as the prefix of a length-$s_{n+1}$ word in $\w(Y)$; 
that is, there exists 
$u'\in \w_{s_{n+1}-s_n-1}(Y)$ such that $v^{(n)}wu'\in \w_{s_{n+1}}(Y)$.
The definition of $X$ then means that $\b v^{(n)}wu'\b\in \w(X)$. 
Defining $u=u'\b$, we then see that 
 $u\in \w_{s_{n+1}-s_n}(X)$, $u\in F_{X}( \b v^{(n)}w)$, $u_{s_{n+1}-s_n}=\b$, 
and $u_i=u'_i\neq \b$ for $1\le i\le s_{n+1}-s_n-1$, so indeed $u$ satisfies the requirements of the claim.

On the other hand, $F_{X}( \b v^{(n)}w)$ does not contain any words of the form
$v=v'\b$ where $|v'|<s_{n+1}-s_n-1$ (the argument is similar to the one above, again using the length constraint).

Now let us compare $F_{X}( \b v^{(n)}w)$ to $F_{X}( \b v^{(m)}w)$, for $m>n>N$, and argue that these sets are distinct.
From the above, $F_{X}( \b v^{(n)}w)$ contains some word $u=u'\b$ where $|u|=s_{n+1}-s_n$,
but $F_{X}( \b v^{(m)}w)$ does not contain any word of the form $v=v'\b$
where $|v|<s_{m+1}-s_{m}$.
Now $s_{n+1}-s_n < s_{m+1}-s_{m}$ since $S$ is superlinear and $m>n$,
therefore $u$ does not belong to $F_{X}( \b v^{(m)}w)$,
and so the sets $F_{X}( \b v^{(n)}w)$ and $F_{X}( \b v^{(m)}w)$ are distinct.

So  $F_{X}( \b v^{(n)}w) \neq F_{X}( \b v^{(m)}w)$
for all $m>n>N$,
therefore the cardinality of the set
$
 \{ F_{X}(vw) : v\in P_{X}(w) \}
$
is infinite,
and hence $w\in \w_\infty(X)$.
But $w\in \w(Y)$ was arbitrary, so 
$\w(Y) \subseteq \w_\infty(X)$, which is the required inclusion (\ref{reverse inclusion2_later}),
so the result is proved.
\end{proof}

In the language of operators on the class $\s$ of all shift spaces, Theorem \ref{prescribed_markov_boundary}
can be interpreted in terms of right-inverses of the Markov boundary operator:

\begin{cor}\label{right_inverse}
If $S\subseteq \N_0$ is superlinear, then 
the $S$-interspersion map $\i_S:\s\to\s$ is a right-inverse
to the Markov boundary map $\partial_M:\s\to\s$, i.e.~$\partial_M\circ\i_S={\rm id}_\s$.
\end{cor}
\begin{proof}
Immediate from Theorem \ref{prescribed_markov_boundary}.
\end{proof}

\begin{rem}
The superlinear length-constrained interspersions $\i_S$ of Corollary \ref{right_inverse} are not the only right-inverses
to the Markov boundary map: for example the context free shift $X$ was shown, in the proof
of Theorem \ref{context_free_theorem}, to have Markov boundary equal to a certain shift of finite type $Y$,
however $X$ is not equal to any length-constrained interspersion of that $Y$.
\end{rem}

While Theorem \ref{TPO_descriptive_TPO_S_infinite} asserts, for completely general $S\subseteq\N_0$, that $\i_S(Y)$ has TPO whenever $Y$ does, the advantage of Theorem \ref{prescribed_markov_boundary}
is in allowing the construction of eventually sofic shifts of \emph{arbitrary level} $n$, 
via iterated interspersion (thereby answering the two questions posed at the start of this section):

\begin{cor}\label{superlinear_eventually_sofic_thm}
Suppose $Y$ is a sofic shift,  and $n\in\N_0$.
If $S_i$ is a superlinear subset of $\N_0$, for $1\le i\le n$,
 then
$$X= (\i_{S_n} \circ \cdots \circ \i_{S_1}) (Y)$$ 
is eventually sofic of level $n$.
In particular, the shift space $X$ has TPO.
\end{cor}
\begin{proof}
By
Theorem \ref{prescribed_markov_boundary} (or equivalently Corollary \ref{right_inverse}),
$\partial_M\circ \i_{S_i}$ is the identity map on $\s$, for each $1\le i\le n$.
So
 $\partial_M^n X = \partial_M^n \circ (\i_{S_n} \circ \cdots \circ \i_{S_1}) (Y)=Y$ is sofic and non-empty, 
and therefore $X$ is eventually sofic of level $n$.
In particular $X$ is eventually sofic, so Theorem \ref{pre_sofic_theorem} implies it has TPO.
\end{proof}

\begin{cor}\label{superlinear_eventually_sofic_cor}
Suppose $Y$ is a sofic shift, $S$ is a superlinear subset of $\N_0$, and $n\in\N_0$. Then
the iterated interspersion $X= \i_S^n(Y)$ is eventually sofic of level $n$.
In particular, the shift space $X$ has TPO.
\end{cor}
\begin{proof}
Immediate from Corollary \ref{superlinear_eventually_sofic_thm}, by choosing $S_i=S$ for all $1\le i\le n$.
\end{proof}

\begin{rem}
The example in Section \ref{introsection}, involving the construction of eventually sofic shifts of level $n$ whose iterated
Markov boundary is the even shift, corresponds to the choice, in Corollary \ref{superlinear_eventually_sofic_cor},
of $Y$ to be the even shift, and $S=\{2^i:i\in\N_0\}$.
\end{rem}

We now turn our attention away from eventually sofic shifts, and towards the fragile,
and specimen, shift spaces of Section \ref{specimen_section}.
In this case the unconstrained interspersion $\i=\i_{\N_0}$ 
(corresponding to the non-superlinear case $S=\N_0$) will be key,
since it guarantees a \emph{specification} condition:

\begin{lem}\label{unconstrained_interspersion_implies_specification}
Suppose $Z$ is any shift space, and let $X=\i(Z)$ be the corresponding unconstrained interspersion.
Then $X$ has specification, and hence variable specification.
\end{lem}
\begin{proof}
Let $\b$ denote the magic symbol 
for $X$.
Choosing $N=1$, 
we note that if
$u,w\in \w(X)$ then 
 $u\b w\in\w(X)$,
so the word $v:=\b\in \w_1(X)$ satisfies the condition of Definition \ref{variable_specification_defn}.
\end{proof}

\begin{cor}\label{Z_min_nue_interspersion_specimen}
Suppose $Z$ is a dynamically minimal non-uniquely ergodic shift space, and let $X=\i(Z)$ be the corresponding unconstrained interspersion.
Then $X$ is a specimen shift space.
In particular, $X$ is fragile, and has TPO.
\end{cor}
\begin{proof}
By Lemma \ref{unconstrained_interspersion_implies_specification}, $X$ has variable specification.
By Corollary \ref{boundary_S_equals_Y}, the Markov boundary of $X$ is equal to $Z$.
So the hypothesis about $Z$ means that $\partial_M X$ is
dynamically minimal and non-uniquely ergodic, therefore $X$ is specimen
(cf.~Definition \ref{specimen_defn}).
But specimen shifts are fragile, and have TPO, by Theorem \ref{specimen_implies_TPO}, so the result follows.
\end{proof}

In light of Corollary \ref{Z_min_nue_interspersion_specimen},
more concrete examples of specimen shift spaces  $X=\i(Z)$ can be given by choosing $Z$ explicitly.
Oxtoby \cite{oxtoby} was the first to construct examples of dynamically minimal non-uniquely ergodic shift spaces.
More generally, the class of Toeplitz shifts (introduced by Jacobs \& Keane \cite{jacobskeane})
was shown to provide many examples of dynamically minimal non-uniquely ergodic shift spaces
(see e.g.~\cite{downarowicz, markleypaul, williams}).
The concrete class of fragile shift spaces mentioned in Section \ref{introsection}, on the alphabet $\{0,1\}$,
are of this kind.
In particular, we are now able 
to prove the following, corresponding to Theorem \ref{X_tpo_does_not_imply_partialX_tpo} in Section \ref{introsection}: 

\begin{thm}
There exist shift spaces with TPO, but whose Markov boundary does not have TPO.
\end{thm}
\begin{proof}
Any specimen shift space $X$ has TPO, by
Theorem \ref{specimen_implies_TPO},
and its Markov boundary is dynamically minimal but not uniquely ergodic, so in particular does not contain any periodic orbits,
therefore $\partial_M X$ does not have TPO.
The class of specimen shift spaces is non-empty by 
Corollary \ref{Z_min_nue_interspersion_specimen}.
\end{proof}

Specific examples of specimen shifts can be used to provide similarly specific examples of eventually specimen shifts.
A generalisation of Corollary \ref{Z_min_nue_interspersion_specimen},
and a particular instance of Theorem \ref{eventually_specimen_implies_TPO}, is the following:

\begin{cor}
Suppose $Z$ is a dynamically minimal non-uniquely ergodic shift space.
Suppose $S_i$ is a superlinear subset of $\N_0$, for each $ i\in\N$.
For each $n\in\N$,
$$X_n= (\i_{S_n} \circ \cdots \circ \i_{S_1}\circ \i) (Z)$$ 
is an eventually specimen shift space,
and is eventually fragile, and has TPO.
Moreover, the family $\{X_n\}_{n\in\N}$ consists of pairwise distinct shift spaces.
\end{cor}
\begin{proof}
By Corollary \ref{Z_min_nue_interspersion_specimen}, $\i(Z)$ is specimen.
Each $X_n= (\i_{S_n} \circ \cdots \circ \i_{S_1}\circ \i) (Z)$
is therefore eventually specimen, because
$\partial_M^n (\i_{S_n} \circ \cdots \circ \i_{S_1}\circ \i) (Z) =\i(Z)$ 
by Corollary \ref{right_inverse}. 
Hence each $X_n$ is eventually fragile and has TPO, by Theorem \ref{eventually_specimen_implies_TPO}.

To see that the shift spaces $X_n$ are pairwise distinct,
simply note that if $m<n$ then there is a (magic) symbol appearing in certain sequences belonging to
$X_n= (\i_{S_n} \circ \cdots \circ \i_{S_1}\circ \i) (Z)$
that does not appear in any sequences belonging to
$X_m= (\i_{S_m} \circ \cdots \circ \i_{S_1}\circ \i) (Z)$.
\end{proof}

\section{Shift spaces without typical periodic optimization}\label{magicmorsesection}

A shift space $X$ without any periodic orbits clearly does not have TPO.
More generally, 
if $X$  has 
periodic orbits but
$\per(X,\sigma)$ is not weak$^*$ dense in $\e(X,\sigma)$, then
Lemma \ref{periodicdenseinergodic} implies that $X$ does not have TPO.
For example if $X\subseteq\{0,1,2,3\}^\N$ is defined as the shift orbit closure of
the set of sequences
of the form 
$$01^{|w^{(1)}|+1}w^{(1)} 01^{|w^{(2)}|+1}w^{(2)} 01^{|w^{(3)}|+1}w^{(3)}\ldots$$
where each $w^{(i)}\in\w(Z)$ for some 
dynamically minimal and uniquely ergodic
non-periodic subshift $Z\subseteq\{2,3\}^\N$,
then $X$ has  
synchronizing word 
0, and the unique ergodic measure supported by $Z\subseteq X$
cannot be approximated by periodic orbit measures, since the cylinder set $[1]$ is given mass zero by the former measure, and mass $1/2$ by all of the latter measures; this $X$ does not have TPO.

In this section we show that, even if
periodic measures are dense in the set of ergodic measures, 
indeed even if they are dense in the set of all invariant measures,
then typical periodic optimization need not hold
(see Theorem \ref{magic_Morse_non_tpo} below, and cf.~Theorem \ref{non_tpo} in Section \ref{introsection}).
For simplicity of exposition, we will do this by constructing a specific shift space $X$,
which we call the \emph{magic Morse shift}, by augmenting the well known Morse shift 
with certain sequences containing an additional (synchronizing) symbol.

Let $\theta$ denote the \emph{Morse substitution} on the alphabet $\{0,1\}$ (see e.g.~\cite[p.~7]{pytheasfogg}), given by $\theta(0)=01$, $\theta(1)=10$,
and extending in the usual way to a morphism of $\{0,1\}^*$ by setting $\theta(a_1\ldots a_k)= \theta(a_1)\ldots\theta(a_k)$.
The \emph{Morse sequence} 
$$\omega
= 0110100110010110\ldots$$
is the fixed point of $\theta$ beginning with the symbol 0
(see e.g.~\cite{pytheasfogg} for properties of the Morse sequence, which appeared in
Morse \cite{morse}, having previously been studied by Prouhet \cite{prouhet} and Thue \cite{thue}).
For $n\ge0$, if the length-$2^n$ word $\theta_n$ is defined by
$$\theta_n=\theta^n(0),$$
with the convention that $\theta^0(0)=0$, then each $\theta_n$ is a prefix of $\omega$.
The \emph{Morse shift} is defined to be the shift orbit closure $Z=\overline{\{\sigma^n(\omega) : n\ge 0\}}$,
and is well known
to be dynamically minimal and uniquely ergodic.
Let $\lambda$ denote the Morse measure, i.e.~the unique $\sigma$-invariant measure supported by $Z$.

Introducing a new symbol $\b\notin\{0,1\}$, define the \emph{magic Morse shift} $X$
to be the closure in $\{0,1,\b\}^\N$ (which as usual is equipped with the metric $d=d_\alpha$, for some $\alpha\in(0,1)$)
of the set of all points of the form
$$
\sigma^n(\theta_{n_1}\b\, \theta_{n_2}\b\cdots),
$$
where $n\ge 0$
and $(n_i)_{i=1}^\infty$  is a sequence
of natural numbers.
In other words, $X$ is the $\v$-interspersion of $Z$ (cf.~Definition \ref{interspersion_defn}),
where $\v=\{\theta_n: n\ge 0\}\subseteq \w(Z)$, with magic symbol $\b$.
In particular,
$\b$ is
a synchronizing word for $X$, 
by  Lemma \ref{S(Z)properties}(a),
and $Z = \partial_M X \subseteq X$ by Corollary \ref{minimal_partial_Y_v_equals_Y},
since $Z$ is dynamically minimal but non-periodic,
and $\v$ is $Z$-descriptive.

For $n\in\N$, $x\in X$, define the \emph{empirical measure} $E_n(x)$ by
$$
E_n(x)=\frac{1}{n}\sum_{i=0}^{n-1} \delta_{\sigma^i(x)}\,,
$$
noting that this is a probability measure, but is $\sigma$-invariant only when $\sigma^n(x)=x$.

A convenient tool in the following analysis will be the Wasserstein distance 
(see e.g.~\cite{villani, villani2}) 
on the set of Borel probability measures:

\begin{defn}\label{wassersteindefn}
For a compact metric space $(Y,d)$, the \emph{Wasserstein distance} $W$ on $\m(Y)$ is defined by
\begin{equation}\label{wasserstein_defn_eq}
W(\mu,\nu)= \max\left\{ \int \phi\, d\mu - \int \phi\, d\nu : \phi\in\lip(Y), |\phi|_{\lip(Y)}\le 1\right\}.
\end{equation}
\end{defn}

The Wasserstein distance
 is consistent with the weak$^*$ topology, and for our purposes the following simple properties
will be useful:

\begin{lem}\label{wasserstein_lemma}
If $(Y,d)$ is a compact metric space, the Wasserstein distance $W$ satisfies
\begin{itemize}
\item[(a)] If $\mu,\nu\in\m(Y)$ then $W(\mu,\nu)\le \text{diam}(Y)$.
\item[(b)] If $x,y\in Y$ then $W(\delta_x,\delta_y)=d(x,y)$.
\item[(c)] If $n\in\N$, and  $a_i\ge0$ for $1\le i\le n$ with $\sum_{i=1}^n a_i=1$,  and $\mu_i,\nu_i\in\m(Y)$  for $1\le i\le n$,
then 
$$
W\left( \sum_{i=1}^n a_i \mu_i, \ \sum_{i=1}^n a_i \nu_i \right)
\le \
\sum_{i=1}^n a_i  W(\mu_i, \nu_i ).
$$
\end{itemize}
\end{lem}
\begin{proof}
These follow easily from
(\ref{wasserstein_defn_eq}).
\end{proof}

Recalling that $X$ is equipped with the metric $d=d_\alpha$ for some $\alpha\in(0,1)$,
a key estimate will be on the Wasserstein distance between the Morse measure $\lambda$ and certain empirical measures:

\begin{lem}
If $\lambda$ is the Morse measure, and $x=\theta_n \b x'$ for some $n\in\N$, $x'\in X$, 
then
\begin{equation}\label{2_to_the_n}
W(E_{2^{n}}(x),\lambda)\le \left(\frac{2\a}{1-\a}\right) \frac{1}{2^{n}},
\end{equation}
and
\begin{equation}\label{2_to_the_n_plus_one}
W(E_{2^{n}+1}(x),\lambda)\le \left(\frac{1+\a}{1-\a}\right) \frac{1}{2^{n}+1}.
\end{equation}
\end{lem}
\begin{proof}
For each $n\in\N$, 
$\lambda$ is the weak$^*$ limit
$\lim_{M\to\infty} E_{2^nM}(\omega)$, and we can write
\begin{equation}\label{E_m_M}
E_{2^nM}(\omega)
=
\frac{1}{M}\sum_{m=0}^{M-1} E_{2^n}\left(\sigma^{m2^n}(\omega)\right),
\end{equation}
so that if it can be shown that
\begin{equation}\label{W_E_m}
W\left(E_{2^n}(x),  E_{2^n}\left(\sigma^{m2^n}(\omega)\right) \right) \le \left(\frac{2\a}{1-\a}\right) \frac{1}{2^{n}}
\end{equation}
for $0\le m\le M-1$, then (\ref{E_m_M}), (\ref{W_E_m}) and Lemma \ref{wasserstein_lemma} give
$$
W\left(E_{2^{n}}(x),  E_{2^nM}(\omega)\right)\le
\frac{1}{M} \sum_{m=0}^{M-1} W( E_{2^{n}}(x),  E_{2^n}\left(\sigma^{m2^n}(\omega)\right) )
\le \left(\frac{2\a}{1-\a}\right) \frac{1}{2^{n}},
$$
for all $M\in\N$, and hence (\ref{2_to_the_n}) follows from the fact that $\lambda=\lim_{M\to\infty} E_{2^nM}(\omega)$.

To prove (\ref{W_E_m}), each of the $2^n$ points in the support of $E_{2^n}(x)$ 
will be paired with a nearby point in the support of
$E_{2^n}(\sigma^{m2^n}(\omega))$, with the distance between these points 
equal to the Wasserstein distance between the corresponding Dirac measures, by Lemma \ref{wasserstein_lemma}(b).
Specifically, writing $\omega=\omega_1\omega_2\ldots$, so that
	$\omega=\theta^n(\omega)=\theta^n(\omega_1)\theta^n(\omega_2)\cdots$
for all $n\in\N$,
we see that 
 the length-$2^n$ prefix of
$\sigma^{m2^n}(\omega)$
is either $\theta^n(0)$ or $\theta^n(1)$,
for all $m\ge0$.

In the case that $\theta^n(0)=\theta_n$
is the length-$2^n$ prefix of
$\sigma^{m2^n}(\omega)$, since $\theta_n$ is also the length-$2^n$ prefix of $x=\theta_n \b x'$
then 
\begin{equation*}
d(\sigma^i(x),\sigma^i(\sigma^{m2^n}(\omega)))\le \alpha^{2^n-i} \quad\text{for }0\le i\le 2^n-1,
\end{equation*}
and therefore
$$
W\left(E_{2^n}(x),  E_{2^n}\left(\sigma^{m2^n}(\omega)\right) \right) \le
\frac{1}{2^n} \sum_{i=0}^{2^n-1} d(\sigma^i(x),\sigma^i(\sigma^{m2^n}(\omega)))
\le
\frac{1}{2^n} \sum_{i=0}^{2^n-1} \alpha^{2^n-i} < \frac{\alpha}{1-\alpha} \frac{1}{2^n},
$$
which in particular implies (\ref{W_E_m}).

In the case that $\theta^n(1)$
is the length-$2^n$ prefix of
$\sigma^{m2^n}(\omega)$,
since
$\theta^n(1)=\theta^{n-1}(1)\theta^{n-1}(0)$ and $\theta_n=\theta^{n-1}(0)\theta^{n-1}(1)$,
we see that $\sigma^{2^{n-1}}(\sigma^{m2^n}(\omega))$ and $x$
share the same length-$2^{n-1}$ prefix $\theta^{n-1}(0)=\theta_{n-1}$,
and
$\sigma^{m2^n}(\omega)$ and $\sigma^{2^{n-1}}(x)$
share the same length-$2^{n-1}$ prefix $\theta^{n-1}(1)$.
Consequently,
\begin{equation}\label{813}
d\left( \sigma^i(x), \sigma^i( \sigma^{2^{n-1}}(\sigma^{m2^n}(\omega)) )\right)
\le \alpha^{2^{n-1}-i} \quad\text{for }0\le i\le 2^{n-1}-1,
\end{equation}
and
\begin{equation}\label{814}
d\left( \sigma^i(\sigma^{2^{n-1}}(x)), \sigma^i(\sigma^{m2^n}(\omega) )\right)
\le \alpha^{2^{n-1}-i} \quad\text{for }0\le i\le 2^{n-1}-1,
\end{equation}
so the bounds (\ref{813}), (\ref{814}),
together with
an argument analogous to the one above, give
\begin{equation}\label{815}
W\left(E_{2^{n-1}}(x),E_{2^{n-1}}(\sigma^{2^{n-1}}(\sigma^{m2^n}(\omega)))\right) < \frac{\alpha}{1-\alpha}\frac{1}{2^{n-1}}
\end{equation}
and
\begin{equation}\label{816}
W\left(E_{2^{n-1}}(\sigma^{2^{n-1}}(x)),
E_{2^{n-1}}(\sigma^{m2^n}(\omega))\right) 
< \frac{\alpha}{1-\alpha}\frac{1}{2^{n-1}}.
\end{equation}
Now
\begin{equation*}
E_{2^n}(x)= \frac{1}{2}E_{2^{n-1}}(x) + \frac{1}{2}E_{2^{n-1}}(\sigma^{2^{n-1}}(x))
\end{equation*}
and
\begin{equation*}
E_{2^n}(\sigma^{m2^n}(\omega))
= \frac{1}{2} E_{2^{n-1}}(\sigma^{m2^n}(\omega))
+ \frac{1}{2} E_{2^{n-1}}(\sigma^{2^{n-1}}(\sigma^{m2^n}(\omega))),
\end{equation*}
so
\begin{align*}
&W \left(E_{2^n}(x),  E_{2^n}\left(\sigma^{m2^n}(\omega)\right) \right) \\
&\le
\frac{1}{2}W\left(E_{2^{n-1}}(x),E_{2^{n-1}}(\sigma^{2^{n-1}}(\sigma^{m2^n}(\omega)))\right) 
+
\frac{1}{2} W\left(E_{2^{n-1}}(\sigma^{2^{n-1}}(x)),
E_{2^{n-1}}(\sigma^{m2^n}(\omega))\right) ,
\end{align*}
and therefore (\ref{815}), (\ref{816}) yield the required inequality (\ref{W_E_m}).

Noting that $E_{2^n+1}(x)= \frac{2^n}{2^n+1} E_{2^n}(x)+\frac{1}{2^n+1} E_1(\b x')$, it follows from Lemma 
\ref{wasserstein_lemma}(c)
 that
$$
W(E_{2^{n}+1}(x),\lambda)\le  \frac{2^n}{2^n+1} W( E_{2^n}(x),\lambda) +\frac{1}{2^n+1} W(E_1(\b x'),\lambda),$$
and using (\ref{2_to_the_n}) together with the bound $ W(E_1(\b x'),\lambda)\le 1$
(see  Lemma \ref{wasserstein_lemma}(a))
 gives
(\ref{2_to_the_n_plus_one}).
\end{proof}

\begin{thm}\label{magic_Morse_non_tpo}
The magic Morse shift $X$ does not have TPO, though its periodic measures are weak$^*$ dense in the set of all invariant 
measures.
\end{thm}
\begin{proof}
To show that the TPO property does not hold,
let $\rho$ be any periodic measure on $X$.
There exists a periodic word $u$ for $\rho$ of the form $u=\theta_{n_1}\b\, \theta_{n_2}\b\ldots \theta_{n_k}\b$,
for some $k\in\N$ and $n_1,\ldots,n_k\in\N$.
Since the symbol $\b$ does not appear in the words $\theta_{n_1},\ldots, \theta_{n_k}$,
then $\rho([\b])=k/|u|$.
Writing 
$z^{(1)}=\b u^\infty$, 
and $z^{(j+1)}=\sigma^{2^{n_j}+1}(z^{(j)})$ for $1\le j\le k-1$,
the points $z^{(1)}, \ldots, z^{(k)}$ in $\supp (\rho)$ all begin with symbol $\b$, so are distance 1 from 
the Morse shift $Z\subseteq\{0,1\}^\N$,
therefore
\begin{equation}\label{distanceintegral}
\int d(\cdot, Z)\, d\rho \ge \frac{k}{|u|}=\rho([\b]).
\end{equation}
Writing $y^{(1)}=u^\infty$, 
and $y^{(j+1)}=\sigma^{2^{n_j}+1}(y^{(j)})$ for $1\le j\le k-1$, then
$$
\rho = \frac{1}{|u|} \sum_{j=1}^k  (2^{n_j}+1)  E_{2^{n_j}+1}(y^{(j)}).
$$
Now $W(E_{2^{n_j}+1}(y^{(j)}),\lambda)\le \left(\frac{1+\a}{1-\a}\right) \frac{1}{2^{n^j}+1}$ by 
(\ref{2_to_the_n_plus_one}),
therefore Lemma \ref{wasserstein_lemma}(c) gives
\begin{equation}\label{wasserstein_upper_bound}
W(\rho,\lambda) \le \frac{1}{|u|}\sum_{j=1}^k (2^{n_j}+1) W(E_{m_j}(y^{(j)}),\lambda)
\le \left(\frac{1+\a}{1-\a}\right)\frac{k}{|u|},
\end{equation}
and (\ref{distanceintegral}), (\ref{wasserstein_upper_bound}) together give
\begin{equation}\label{big_Delta_W}
W(\rho,\lambda) \le \left(\frac{1+\a}{1-\a}\right) \int d(\cdot, Z)\, d\rho .
\end{equation}

Now define $f:=-d(\cdot,Z)$. Clearly $f\in\lip(X)$, with $\beta(f)=0$, and $\lambda$ is the unique $f$-maximizing measure.
The inequality (\ref{big_Delta_W}) means that
\begin{equation}\label{Wcbetaphi_morse}
\int \varphi\, d\rho - \int \varphi\, d\lambda \le |\varphi|_{\lip(X)}
\left(\frac{1+\a}{1-\a}\right)
\left(\beta(f) -\int f\, d\rho \right)
\end{equation}
for all $\varphi\in\lip(X)$.
If moreover $|\varphi|_{\lip(X)}< \frac{1-\a}{1+\a}$, then (\ref{Wcbetaphi_morse}),
together with the fact that $\beta(f)-\int f\, d\rho = \int d(\cdot,Z)\, d\rho>0$,
 gives
\begin{equation}\label{Wcbetaphisimple_morse}
\int \varphi\, d\rho - \int \varphi\, d\lambda < \beta(f) -\int f\, d\rho \,,
\end{equation}
and $\beta(f)=\int f\, d\lambda$, so
(\ref{Wcbetaphisimple_morse}) becomes
\begin{equation}\label{Wcbetaphisimpler_morse}
\int (f+\varphi)\, d\rho < \int (f+ \varphi)\, d\lambda  \,,
\end{equation}
which means that $\rho$ is not $(f+\varphi)$-maximizing.
Defining $$B=\left\{g\in\lip(X):\|f-g\|_{\lip(X)}< \frac{1-\a}{1+\a} \right\},$$
we see that if $g\in B$ then 
$\rho$ is not $g$-maximizing.
But $\rho$ was an arbitrary periodic measure, so all members of the open ball $B$ do not have a periodic maximizing measure, and therefore $X$ does not have the TPO property.

One method of showing that the periodic measures on $X$ are dense in $\m(X,\sigma)$ is to note that both the closability and linkability properties of Gelfert \& Kwietniak \cite{GK18} are satisfied, though we here indicate a self-contained proof.

We first claim that periodic measures are dense in the set of ergodic measures. To see this, note that on the one hand the Morse measure
$\lambda$ is weak$^*$ approximated by the sequence of periodic measures concentrated on period-$(2^n+1)$ orbits
with periodic word $\theta_n\b$. For any other ergodic measure $\mu$, there exists some $x=\theta_{n_1}\b\, \theta_{n_2}\b\ldots \in X$ that is generic in the sense that $\mu=\lim_{n\to\infty} E_n(x)$, and the sequence $\rho_k$ of periodic
measures with periodic word
$\theta_{n_1}\b\, \theta_{n_2}\b\ldots \theta_{n_k}\b$ is readily checked to satisfy $\lim_{k\to\infty}\rho_k = \mu$.

If $\mu\in\m(X,\sigma)$ is not ergodic, then it can be weak$^*$ approximated by an equi-barycentre $\frac{1}{k}\sum_{i=1}^k \mu_i$ of (not necessarily distinct) ergodic measures $\mu_i$, each of which can (by the above argument) be approximated by a periodic measure $\rho_i$. If $p$ is a common multiple of the (least) periods of the orbits corresponding to the $\rho_i$,
then these measures have periodic words $w_i$ of common length $p$, 
where each $w_i$ begins with symbol $\b$,
and the self-concatenation $w_i^n$ is also a periodic word for $\rho_i$, for any $n\in\N$.
Concatenating these latter words gives a word $w^{(n)}:=w_i^n\ldots w_k^n$ (of length $knp$) that is the periodic word for some periodic measure $\rho^{(k)}$ on $X$. The $\rho^{(k)}$ converge to $\frac{1}{k}\sum_{i=1}^k \rho_i$,
and therefore approximate $\mu$, as required.
\end{proof}

\end{document}